\pdfoutput=1
\documentclass[11pt, reqno]{amsart}

\usepackage{
 amsmath, 
 amsxtra, 
 amsthm, 
 amssymb, 
 etex, 
 mathrsfs, 
 mathtools, 
 tikz-cd, 
 bbm,
 xr,
 comment}
\usepackage[all]{xy}
\usepackage{hyperref}

\newtheorem{theorem}{Theorem}[subsection]
\newtheorem{lemma}[theorem]{Lemma}
\newtheorem{conjecture}[theorem]{Conjecture}
\newtheorem{proposition}[theorem]{Proposition}
\newtheorem{corollary}[theorem]{Corollary}
\newtheorem{defn}[theorem]{Definition}

\newtheorem{notation}[theorem]{Notation}

\newtheorem{lthm}{Theorem} 
\newtheorem{corl}[lthm]{Corollary}

\theoremstyle{remark}
\newtheorem{remark}[theorem]{Remark}
\newtheorem{conv}[theorem]{Convention}
\usepackage[OT2,T1]{fontenc}

\DeclareSymbolFont{cyrletters}{OT2}{wncyr}{m}{n}
\DeclareMathSymbol{\Sha}{\mathalpha}{cyrletters}{"58}

\newcommand{\tupH}{\textup{H}}
\newcommand{\QQ}{\mathbb{Q}}
\newcommand{\Qp}{\mathbb{Q}_p}
\newcommand{\Qpn}{\QQ_{p,n}}
\newcommand{\ZZ}{\mathbb{Z}}
\newcommand{\Zp}{\mathbb{Z}_p}

\newcommand{\NN}{\mathbb{N}}
\newcommand{\DD}{\mathbb{D}}
\renewcommand{\AA}{\mathbb{A}}
\newcommand{\Brig}{\mathbb{B}_{\rig,\Qp}^+}
\newcommand{\AQp}{\AA_{\Qp}^+}
\newcommand{\Dcris}{\DD_{\mathrm{cris}}}

\newcommand{\TT}{\mathbf{T}}
\newcommand{\GL}{\mathbf{GL}}

\newcommand{\cBF}{\mathcal{BF}}
\newcommand{\BF}{\textup{BF}}

\newcommand{\cyc}{\textup{cyc}}

\newcommand{\Gr}{\textup{Gr}}
\newcommand{\aA}{\mathbf{A}}

\DeclareMathOperator{\Gal}{Gal}
\DeclareMathOperator{\Fil}{Fil}
\DeclareMathOperator{\Hom}{Hom}
\DeclareMathOperator{\Sel}{Sel}
\DeclareMathOperator{\coker}{coker}
\DeclareMathOperator{\Fitt}{Fitt}

\newcommand{\ord}{\mathrm{ord}}
\newcommand{\vp}{\varphi}
\newcommand{\Iw}{\mathrm{Iw}}
\newcommand{\HIw}{\tupH^1_{\Iw}}
\newcommand{\col}{\mathrm{Col}}
\newcommand{\image}{\mathrm{Im}}

\newcommand{\Char}{\mathrm{char}}

\newcommand{\cor}{\mathrm{cor}}

\newcommand{\Tw}{\mathrm{Tw}}
\newcommand{\lb}{[[}
\newcommand{\rb}{]]}

\newcommand{\rig}{\mathrm{rig}}
\newcommand{\p}{\mathfrak{p}}

\newcommand{\fM}{\mathfrak{M}}

\newcommand{\cF}{\mathcal{F}}
\newcommand{\cG}{\mathcal{G}}
\newcommand{\cK}{\mathcal{K}}
\newcommand{\cH}{\mathcal{H}}
\newcommand{\cL}{\mathcal{L}}
\newcommand{\cX}{\mathcal{X}}
\newcommand{\cO}{\mathcal{O}}

\newcommand{\cT}{\mathcal{T}}
\newcommand{\cV}{\mathcal{V}}

\newcommand{\BK}{\mathrm{BK}}

\newcommand{\bz}{\mathbf{z}}
\newcommand{\bh}{\mathbf{h}}
\newcommand{\Tr}{\mathrm{Tr}}
\newcommand{\Ind}{\mathrm{Ind}}
\newcommand{\cE}{\mathcal{E}}

\definecolor{Green}{rgb}{0.0, 0.5, 0.0}

\newcommand{\draftcolor}{Green}

\newcommand{\bgreen}{\begin{color}{\draftcolor}}
\newcommand{\egreen}{\end{color}}

\begin{document}

\title[Mazur--Tate conjectures for Rankin--Selberg convolutions]{On  analogues of Mazur--Tate type conjectures in the Rankin--Selberg setting}

\begin{abstract}
We study the Fitting ideals over the finite layers of the cyclotomic $\Zp$-extension of $\QQ$ of Selmer groups attached to the Rankin--Selberg convolution of two modular forms $f$ and $g$. Inspired by the Theta elements for modular forms defined by Mazur and Tate in ``Refined conjectures of the Birch and Swinnerton-Dyer type'', we define new Theta elements for Rankin--Selberg convolutions of $f$ and $g$ using Loeffler--Zerbes' geometric $p$-adic $L$-functions attached to $f$ and $g$.

Under certain technical hypotheses, we generalize a recent work of Kim--Kurihara on elliptic curves to prove a result  very close to the \emph{weak main conjecture} of Mazur and Tate for Rankin--Selberg convolutions. Special emphasis is given to the case where $f$ corresponds to an elliptic curve $E$ and $g$ to a two dimensional odd irreducible  Artin representation $\rho$ with  splitting field $F$. As an application, we give an upper bound of the dimension of the $\rho$-isotypic component of the Mordell--Weil group of $E$ over the finite layers of the cyclotomic $\Zp$-extension of $F$ in terms of the order of vanishing of our Theta elements. 
\end{abstract}

\author[A. Cauchi]{Antonio Cauchi}
\address[Cauchi]{Departament de Matem\`atiques\\ Universitat Polit\`ecnica de Catalunya\\ C. Jordi Girona 1-3\\ 08034 Barcelona\\ Spain}
\email{antonio.cauchi@upc.edu}

\author[A. Lei]{Antonio Lei}
\address[Lei]{D\'epartement de Math\'ematiques et de Statistique\\
Universit\'e Laval, Pavillion Alexandre-Vachon\\
1045 Avenue de la M\'edecine\\
Qu\'ebec, QC\\
Canada G1V 0A6}
\email{antonio.lei@mat.ulaval.ca}

\numberwithin{equation}{subsection}

\maketitle

\setcounter{tocdepth}{1}
\tableofcontents

\section{Introduction}\label{sec:intro}

\subsection{Refined main conjectures for elliptic curves} Let $E/\QQ$ be an elliptic curve over $\QQ$ and let $p$ be an odd prime of good reduction for $E$. In their seminal paper \cite{MaTa}, Mazur and Tate proposed refinements of the Birch and Swinnerton-Dyer conjecture and their $p$-adic analogue (cf. \cite{MTT}) for $E$, by working at ``finite layers'', which we now recall.

For $M$ a positive integer, let $\chi$ be an even Dirichlet character of conductor $M$. In \cite[\S 1]{MaTa}, the authors introduced the \emph{modular element} $\theta_{E,M} \in \QQ[\Gal(\QQ(\mu_M)/\QQ)/\{ \pm 1 \}]$, with the property that
\begin{equation} \label{eqn:mazur-tate1}
\chi(\theta_{E,M}) \approx L^{\mathrm{alg}}(E, \chi^{-1}, 1) 
\end{equation} 
where $L^{\mathrm{alg}}(E, \chi^{-1}, 1)$ denotes the algebraic part of the twisted $L$-value. Mazur and Tate conjectured that the order of vanishing of $\theta_{E,M}$ at $\chi$ is greater or equal to the dimension of the $\chi$-part of the Mordell--Weil group of $E$. They actually proposed more precise refinements of this conjecture, aiming at giving an explicit description of the structure of the Selmer group of $E$ over finite abelian extension of $\QQ$. 
Of particular interest to us are these refinements over the finite layers of the cyclotomic $\Zp$-extension of $\QQ$, which we now describe. 

Fix embeddings $\iota_p : \overline{\QQ} \hookrightarrow \overline{\QQ}_p$ and $\iota_{\infty} : \overline{\QQ} \hookrightarrow \mathbb{C}$. Denote by $K_n \subset \QQ(\mu_{p^{n+1}})$ the field  extension of $\QQ$ with Galois group $G_n : = \textup{Gal}(K_n/\QQ) \simeq \ZZ / p^n \ZZ$. As $\Gal(\QQ(\mu_{p^{n+1}})/\QQ) \simeq G_n \times \ZZ/(p-1)\ZZ$, we have a map $\textup{pr}_{G_n}:\Gal(\QQ(\mu_{p^{n+1}})/\QQ)/\{ \pm 1 \} \to G_n$. Finally, denote by $K_\infty= \cup_n K_n$ the cyclotomic $\Zp$-extension of $\QQ$. 

Given a field $\cK$, we write $G_\cK$ for the absolute Galois group of $\cK$.
Let $\Theta_n(E) \in \QQ[G_n]$ denote the image of $\theta_{E,p^{n+1}}$ under the map $\textup{pr}_{G_n}$. Under the hypothesis of irreducibility of the $G_\QQ$-representation $E[p]$ and that $\ord_p\left(L(E,1)/\Omega_E\right)=0$, $\Theta_n(E)$ belongs to $\Zp[G_n ]$ (cf. \cite[pp. 200-201]{Kurihara2002}). Inspired by the \emph{weak main conjecture} of Mazur and Tate (cf. {\cite[Conjecture 3]{MaTa}}), Kurihara proposed the following \emph{refined main conjecture} for the ``finite layer'' $K_n$. 

\begin{conjecture}[{\cite[Conjecture 0.3]{Kurihara2002}}]\label{kuriconj} Let $\cX_n(E)$ denote the Pontryagin dual of the Selmer group $\textup{Sel}(K_n,E[p^\infty])$. Assume that  $E(\QQ)[p]$ is trivial and that $p$ does not divide the Tamagawa number of $E$.
Then \[ (\Theta_n(E), \nu_n(\Theta_{n-1}(E))) = \textup{Fitt}_{\Zp[G_n ]}(\cX_n(E) ), \]
where $\textup{Fitt}_{\Zp[G_n ]}(\cX_n(E))$ denotes the Fitting ideal of $\cX_n(E)$ as a $\Zp[G_n ]$-module and  $\nu_n:\Zp[G_{n-1} ] \to \Zp[G_n ]$ is the trace map as in \S \ref{sec:Iwalg}.
\end{conjecture}

The conjecture of Kurihara gives more information than the usual Iwasawa main conjecture as it fully describes $\textup{Fitt}_{\Zp[G_n ]}(\cX_n(E) )$ and explains the growth of $\textup{Sel}(K_n,E[p^\infty])$ as $n$ goes to infinity. 

\subsection{The work of Kim and Kurihara}

In \cite{Kurihara2002},  Kurihara himself proved Conjecture \ref{kuriconj} for $E$ of good supersingular reduction at $p$, which additionally satisfies that $p$ does not divide $L^{\rm alg}(E,1)$ and that the representation $G_\QQ\rightarrow \GL(E[p])$ is surjective.

Under the assumption that $a_p(E)=0$, Pollack \cite{pollack05} reformulated (and proved) the conjecture in terms of the signed $p$-adic $L$-function attached to $E$. This was partly motivated by the explicit connection between the Theta elements and the $\pm$ $p$-adic $L$-functions attached to $E$ defined in \cite{pollack03}.

In \cite{KK}, Kim and Kurihara proved the following Theorem.

\begin{theorem}[{\cite[Theorem 1.14]{KK}}]\label{mainthmkk}
Let $E/\QQ$ be an elliptic curve with good reduction at an odd $p$, such that $E[p]$ is surjective if $E$ is not CM and $p \nmid \textup{Tam}(E)$. Assume one of the following:
\begin{itemize}
    \item[(ord)] $p \nmid a_p(E)$ and $a_p(E) \not \equiv 1\mod p$;
    \item[(ss)] $a_p(E)=0$.
\end{itemize}
Then \[ (\Theta_n(E), \nu_n(\Theta_{n-1}(E)))\subseteq \textup{Fitt}_{\Zp[G_n ]}(\cX_n(E) ). \]
\end{theorem}

\begin{remark}
In \cite{KK}, the full Conjecture \ref{kuriconj} is proved under various additional assumptions on $E$, such as the full equality of the ordinary and plus/minus Iwasawa Main conjectures. We refer to \cite[Theorems 1.18 and 1.20]{KK} for further details. 
\end{remark}

As a corollary of Theorem \ref{mainthmkk}, using \cite[Proposition 3]{MaTa}, Kim and Kurihara proved the following.

\begin{corollary}\label{cor:KKonrank}
 Under the assumptions of Theorem \ref{mainthmkk}, the order of vanishing of $\Theta_n(E)$ at $\chi:G_n \to \overline{\QQ}_p^*$ is greater or equal to the dimension of the $\chi$-part of the Mordell Weil group of $E(K_n)$.
\end{corollary}

\subsection{Theta elements for Rankin--Selberg convolutions}

The purpose of the present article is to generalize the results for elliptic curves stated above to the case of the Rankin--Selberg convolution of two normalized cuspidal eigen-newforms $f$ and $g$ of level $N_f,N_g$, characters $\epsilon_f,\epsilon_g$, and weights $k_f+2 > k_g+2\ge1$. The main reason for assuming that the weight of $f$ is strictly greater than that of $g$ is to ensure the existence of critical $L$-values of the Rankin--Selberg convolution. We assume throughout that $p\nmid N_fN_g$.

Let $\alpha_f,\beta_f$ and $\alpha_g, \beta_g$ denote the roots of  the Hecke polynomials of $f$ and $g$ at $p$ respectively. Throughout, we assume that $\alpha_f\ne\beta_f$ and $\alpha_g\ne\beta_g$. Let $L$ be a finite extension of $\QQ_p$, which contains the coefficients of $f$ and $g$ as well as $\alpha_f,\beta_f,\alpha_g,\beta_g$. Define $V:=V_f \otimes V_g$ to be the $L$-linear $G_{\QQ}$-representation associated to the convolution of $f$ and $g$. If $\cO$ denotes the ring of integers of $L$, we let $T$ denote a Galois-stable $\cO$-lattice in $V$, coming from certain integral sheaves on the modular curves of levels $N_f$ and $N_g$ (see \S\ref{sec:basis} below for more details). We set $A:= V / T$, and $\Lambda_n:= \cO[G_n]$.

For every $n \geq 0$ and $k_g+1 \leq j \leq k_f$, we construct  Theta elements \[ \Theta_{j,n}, \Theta^{\pm}_{j,n} \in L[G_n]\] by means of the geometric $p$-adic $L$-functions attached to $f$ and $g$ (cf. Definition \ref{analyticdeftheta}). These elements satisfy interpolation formulae similar to \eqref{eqn:mazur-tate1} for the Theta elements for elliptic curves. For instance, when $g$ is $p$-ordinary, $\theta$ is a non-trivial character on $G_n$ of conductor $p^m$ (so that $1\le m\le n+1$) and $\alpha_f\ne -\beta_f$, in Lemma~\ref{lem:evaluateTheta} we show that 
\[
\Theta_{j,n}(\theta )=\frac{\beta_f^{2n-2m+4}-\alpha_f^{2n-2m+4}}{\beta_f^2-\alpha_f^2}\cdot c_{j,m,\theta}\cdot L(f,g,\theta^{-1},j+1),
\]
where  $c_{j,m,\theta}$ is a constant appearing in the interpolation formulae of the $p$-adic $L$-functions $L_p(f_\alpha,g)$ and $L_p(f_\beta,g)$  at $\theta\chi_\cyc^j$, with $\chi_\cyc$ being the $p$-adic cyclotomic character and $L(f,g,\theta^{-1},s)$ denotes the complex Rankin $L$-function defined as in \cite[\S2.7]{KLZ2}. Note that $c_{j,m,\theta}$ is a ratio of an algebraic number and a complex period, allowing us to regard the right-hand side as an element of $L(\mu_{p^m})$ via the fixed embeddings $\iota_p$ and $\iota_\infty$.

\begin{remark}We make a couple of remarks on the construction of our Theta elements.
\begin{itemize}
    \item[i)] An alternative (but equivalent) construction of the Theta elements is discussed in Appendix \ref{arithmeticconstructionappendix}, where they are shown to be equal to images of the Beilinson--Flach classes considered in \cite{LLZ14,KLZ1,KLZ2,LZ1}  under certain  Perrin-Riou pairings over $K_n$.  As an application, we show that, when the weight of $g$ is one, the Theta elements are integral as long as a Fontaine-Laffaille condition holds for the representation $V$ (cf. Proposition \ref{cohomconstruction}). 

\item[ii)] The Theta elements have bounded denominator as $n$ varies (see Lemma~\ref{lem:integral}). However, in the range of cases pertaining to our main arithmetic results, they are actually elements of $\Lambda_n$.
\end{itemize}
\end{remark}

In our quest to study the Fitting ideals of Selmer groups over $K_n$ attached to $f$ and $g$, we consider two different settings. The first is when $f$ and $g$ are both $p$-ordinary and the second is when $f$ corresponds to a $p$-supersingular elliptic curve $E$ with $a_p(E)=0$ and $g$ is of weight one. We now discuss them separately.
 
\subsection{Main theorems: the ordinary setting}
We consider a big enough prime $p$ coprime with $N_f N_g$ and let $f$ and $g$ be both $p$-ordinary, with unit roots $\alpha_f$ and $\alpha_g$; our main object of interest is the Greenberg Selmer group $\textup{Sel}_{\Gr}(K_n,A(j+1))$ (see \S \ref{subsec:doublyordinaryselmer} for details), for $k_g+1 \leq j \leq k_f$. We allow $n$ to be $\infty$, by defining \[\textup{Sel}_{\Gr}(K_\infty,A(j+1)) := \varinjlim_{\QQ \subseteq K \subseteq K_{\infty}} \textup{Sel}_{\Gr}(K,A(j+1)).\]

 We study the growth of $\textup{Sel}_{\Gr}(K_n,A(j+1))$ as $n$ varies by analysing the Fitting ideal over $\Lambda_n$ of its Pontryagin dual, which we denote by $\cX_{j,n}^\textup{Gr}$.
 
 By mimicking the techniques used in \cite{KK}, we prove the following theorem, which corresponds to the setting (ord) in Theorem~\ref{mainthmkk}. 

\begin{lthm}[{Theorem \ref{mainordinary}}]\label{maintheorem1} 
 Suppose that the residual representations of $f$ and $g$ are irreducible, that
$f$ is non-CM, that $\epsilon_f \epsilon_g \neq 1$, that $(N_f,N_g)=1$, and one of the following conditions: \begin{itemize}
    \item $g$ is non-CM and has weight $\geq 2$;
    \item $g$ has weight $\geq 2$, is of CM-type, and $\epsilon_g$ is not 1 nor the quadratic character attached to the corresponding CM field of $g$;
    \item $g$ has weight 1.
\end{itemize}  
Furthermore, suppose that there exists a Dirichlet character $\theta$ factoring through $\Gal(K_\infty/\QQ)$ such that $L_p(f_\alpha,g)(\theta\chi_{\rm cyc}^j) \ne 0$, with $k_g+1 \leq j \leq k_f$. If $(f,g,j)$ is non-anomalous ordinary at $p$, then
\[ \left(\Theta_{j,n}, \nu_n(\Theta_{j,n-1}) \right)  \subseteq \Fitt_{\Lambda_n}\cX_{j,n}^\textup{Gr}. \]
\end{lthm}
 
 The non-anomalous ordinary condition at $p$ is introduced in Definition \ref{def:nonanomalous}. In Theorem \ref{mainordinary}, we further show that $\left(\Theta_{j,n}, \nu_n(\Theta_{j,n-1}) \right)$ is the principal ideal generated by $\Theta_{j,n}$.  We also note that the existence of such a character $\theta$ implies that our Theta elements are non-trivial.
 We invite the reader to consult \S \ref{ordinarysectionmaintheorems} for more details. 
 
 \begin{remark}
 In the case where $f$ corresponds to a non-CM elliptic curve $E/\QQ$ and $g$ to a two dimensional odd irreducible Artin representation of $G_\QQ$ with splitting field $F$, Theorem \ref{maintheorem1} gives, in the same spirit of  Corollary \ref{cor:KKonrank}, an upper bound to the dimension of  the $\rho$-isotypic component of the Mordell--Weil group of $E(FK_n)$ in terms of the order of vanishing of $\Theta_{j,n}$ (cf. \S \ref{intro:ellcurves} for further details). 
 \end{remark}

Similar to \cite[Theorem 1.14]{KK}, Theorem \ref{maintheorem1} is a consequence of the inclusion of the Iwasawa main conjecture 
\begin{equation*}
 \pi_\Delta \circ \Tw^{j}L_p(f_\alpha,g) \in  \textup{char}_\Lambda( \cX_{j,\infty}^\textup{Gr} )
\end{equation*}
(see \S\ref{sec:Iwalg} for the definitions of $\pi_\Delta$ and $\Tw$), which follows from \cite[Theorem 11.6.4]{KLZ2}, as well as the non-existence of non-trivial finite $\Lambda$-submodules of $ \cX_{j,\infty}^\textup{Gr}$ (which we prove as part of Proposition \ref{nofinitesubmod}), and a control theorem for $\cX_{j,n}^\textup{Gr}$ as $n$ varies (which we establish in Theorem~\ref{controlordinary}). The last two results may be of independent interest. We invite the reader to consult \S \ref{subsec:nonexistenceoflambdasubmoduleoffiniteindex} and \S \ref{sec:ordinarycontroltheorem} for more details and comments.

\subsection{Main theorems: the supersingular setting}
The second setting we consider is when  $f$ corresponds to an elliptic curve $E$ with supersingular reduction at $p$ and $a_p(E)=0$. Since the weight of $f$ is now two, it forces $g$ to have weight one and $j=0$. Using the Coleman map technique developed in \cite{BLLV}, we define four plus and minus Selmer groups $\Sel_{\pm\pm}(K_n,A(1))$ as well as four signed $p$-adic $L$-functions $L_p^{\pm\pm}(f,g)$.  Our Theta elements $\Theta^\pm_{0,n}$ are intimately linked to $L_p^{++}(f,g)$ and $L_p^{--}(f,g)$  (see Proposition~\ref{prop:pmTheta}).

Our assumption that $a_p(E)=0$ and $g$ is of weight one allows us to describe the local conditions at $p$ of the plus and minus Selmer groups explicitly using Kobayashi's "trace-jumping conditions" studied in \cite{kobayashi03}. This explicit description in particular allows us to generalize Kobayashi's work to prove a control theorem for two of these Selmer groups $\Sel_{++}(K_n,A(1))$ and $\Sel_{--}(K_n,A(1))$ (see Theorem~\ref{thm:controlpm}). As in the ordinary case, we also prove in Proposition~\ref{nofinitesubmod} that $\Sel_{\pm\pm}(K_\infty,A(1))^\vee$ admit no non-trivial finite $\Lambda$-submodule. These results allow us to apply the techniques of Kim--Kurihara in \cite{KK} to prove the following generalization of the (ss) part of Theorem~\ref{mainthmkk}.

 Suppose that  \textbf{(H.Kim)} holds and that the inclusion  

\begin{lthm}[Theorem~\ref{thm:final}]\label{intro:maintheoremss}
Assume that $L_p^{++}(f,g)$ and $L_p^{--}(f,g)$ are non-zero elements of $\Lambda$ and that they are contained inside the characteristic ideals of the Pontryagin duals of $\Sel_{++}(K_n,A(1))$ and $\Sel_{--}(K_n,A(1))$ respectively. Furthermore, suppose that the hypothesis \textbf{(H.Kim)} introduced in \S\ref{sec:controlpm} holds. Let $\cX_{n}^\textup{BK}$ denote the Pontryagin dual of the Bloch--Kato Selmer group of $A(1)$ over $K_n$.
 {If $n \geq 2$ is even, then \[
\left(p\Theta_{0,n}^+,\Theta_{0,n}^-\right)\subset \Fitt_{\Lambda_n}\cX_n^\BK.
\]}
{If $n \geq 1$ is odd, then \[
\left(\Theta_{0,n}^+,p\Theta_{0,n}^-\right)\subset \Fitt_{\Lambda_n}\cX_n^\BK.
\]}
\end{lthm}

\subsection{Selmer groups of elliptic curves}\label{intro:ellcurves}
We may combine Theorems~\ref{maintheorem1} and \ref{intro:maintheoremss} to use our Theta elements to study Selmer groups of elliptic curves over number fields in a unified way.
Let $E$ be an elliptic curve over $\QQ$ without CM, of conductor $N_E$. Let $\rho$ be a two dimensional odd irreducible Artin representation of $G_\QQ$ and denote by $F$ its splitting field and by $N_\rho$ its conductor. We assume that $p\nmid N_EN_\rho[F:\QQ]$ and that $p \geq 5$. When $E$ has supersingular reduction we further assume that $4 \nmid [F:\QQ]$. The representation $\rho$ takes values in a finite extension $L_\rho$ of $\QQ$. Fix $\mathfrak{P}$ a prime of $L_\rho$ above $p$ and let $L$ be a finite extension of $L_{\rho,\mathfrak{P}}$ with ring of integers $\cO$.

Let $F_{\infty}=F K_\infty$ be the cyclotomic $\Zp$-extension of $F$, with finite layers $F_{n}$ so that $[F_{n}:F]=p^n$ and $F_{n}$ is the compositum of $K_n$ and $F$. Consider the Selmer group $\textup{Sel}_{n}( E/F)$ of $E$ over $F_n$  and its $\rho$-isotypic component $\textup{Sel}_{n}( E/F)_{(\rho)}$ (cf. Definitions \ref{def:selmerellipticc} and \ref{def:selmerrhoisot}). The results described above allow us to study the Fitting ideal of the Pontryagin dual $\cX_n(E)_{(\rho)}$ of $\textup{Sel}_{n}( E/F)_{(\rho)}$.

As $E$ and $\rho^*$ correspond respectively to a weight 2 modular form $f$ and a weight 1 modular form $g$, we can consider the Theta elements associated to $f$ and $g$.  We adopt the following notation. If $E$ has ordinary reduction at $p$, denote by $\Theta_n(E,\rho)$ the Theta element $\Theta_{0,n}$, while,  if $E$ has supersingular reduction at $p$ with $a_p(E)=0$, we write $\Theta_n^{\pm}(E,\rho) $ for the Theta elements $\Theta_{0,n}^\pm$.
These elements are integral as proved in Proposition~\ref{cohomconstruction}.

Theorems \ref{maintheorem1} and \ref{intro:maintheoremss} together give the following theorem.
\begin{lthm}[Theorem \ref{thm:EC}]\label{introthm:EC}
Suppose that  that $L(E,\rho, \theta^{-1},1) \ne 0$ for some finite order character $\theta$ on $\Gal(K_\infty/\QQ)$. Under certain technical hypotheses on $E$ and $\rho$ (cf. hypotheses (1)--(6) of \S \ref{subsec:ellipticcurvesovernumberfields}), we have:
\begin{itemize}
    \item[(i)] If $E$ has ordinary reduction at $p$, then for all $n \geq 0$, we have 
\[\left(\Theta_n(E,\rho) \right) \subseteq {\rm Fitt}_{\Lambda_n}\cX_n(E)_{(\rho)}.\]
\item[(ii)]Suppose that $E$ has supersingular reduction at $p$ with $a_p(E)=0$ and let $p^m$ be the conductor of $\theta$. 
If $m$ is odd, then for all even $n\ge0$, we have
\[\left(\Theta^-_n(E,\rho) \right) \subseteq {\rm Fitt}_{\Lambda_n}\cX_n(E)_{(\rho)}.\]
If $m$ is even, then for all odd $n\ge0$, we have
\[\left(\Theta^+_n(E,\rho) \right) \subseteq {\rm Fitt}_{\Lambda_n}\cX_n(E)_{(\rho)}.\]
\end{itemize}
\end{lthm}

We conclude by stating a corollary to Theorem \ref{introthm:EC}, which is our analogous to Corollary~\ref{cor:KKonrank} in the current setting. 
Consider the $\rho$-isotypic component 
\[ E(F_{n})_{(\rho)}:= {\rm Hom}_{{\Gal(F_{n}/K_n)}}(\rho, E(F_{n})  \otimes L), \]
and, given $\chi:G_n \to \overline{\QQ}_p^\times$, \[ E(F_{n})_{(\rho)}^\chi:=\{P \in E(F_{n})_{(\rho)} \otimes \overline{\QQ}_p \;:\; \sigma \cdot P = \chi(\sigma)P\; \text{ for all }\sigma \in G_n \}.\]  

\begin{corl}[Corollary~\ref{cor:final}]\label{cor:introfinal}
We keep the hypotheses of Theorem C. Let $\chi$ be a character on $G_n$. 
\begin{itemize}
    \item[(i)] If $E$ has ordinary reduction at $p$, then $\dim_{\overline{\QQ}_p} E(F_{n})_{(\rho)}^\chi \le \ord_\chi\Theta_{n}(E,\rho)$. 
    \item[(ii)] If $E$ has supersingular reduction at $p$, then 
    \[
    \dim_{\overline{\QQ}_p} E(F_{n})_{(\rho)}^\chi\le\begin{cases}
    \ord_\chi\Theta_n^+(E,\rho)&\text{if $n$ is odd,}\\
    \ord_\chi\Theta_n^-(E,\rho)&\text{if $n$ is even.}
    \end{cases}
    \]
\end{itemize}

 \end{corl}
 
 The proof of  Corollary~\ref{cor:final} follows from Theorem \ref{introthm:EC} and \cite[Proposition 3]{MaTa}.
 
\subsection{Acknowledgements}
 We would like to thank Chan-Ho Kim for having taken the time to explain to us some of the calculations in \cite{KK} in details.
This work originated from discussions with Guhan Venkat at Universit\'e Laval. We thank him for his comments and suggestions. We also thank Joaquin Rodrigues Jacinto, David Loeffler and Victor Rotger for fruitful discussions.  We are indebted to the anonymous referees who have read very carefully and thoroughly an earlier version of the manuscript. Their very constructive comments and suggestions have led to many improvements of the  article.

A good part of this work was conducted while the first named author was a postdoctoral fellow at Universit\'e Laval, supported by Centre de Recherches Math\'ematiques and  the NSERC  Discovery Grants Program (RGPIN-2015-05710). He would like to thank Universit\'e Laval for their great hospitality and their support. The first named author's research is supported by the European Research Council (ERC) under the European Union's Horizon 2020 research and innovation programme (grant agreement No. 682152).
The second named author's research is supported by the NSERC  Discovery Grants Program (RGPIN-2020-04259 and RGPAS-2020-00096).

\section{Notation}\label{sec:review}
We fix some notation and  define objects that will be used throughout the paper. Throughout $p$ is  a fixed odd prime and  $L / \Qp$  is a fixed finite extension, which will be the coefficient field for all the representations we shall consider. The ring of integers of $L$ is denoted by $\cO$. We fix a uniformizer $\varpi$ of $\cO$.

\subsection{Iwasawa algebras and distribution algebras}\label{sec:Iwalg}

Let $\Gamma=\Gal(\QQ(\mu_{p^\infty})/\QQ)$. This group is isomorphic to a direct product $\Delta\times\Gamma_1$, where $\Delta$ is a finite group of order $p-1$ and $\Gamma_1=\Gal(\QQ(\mu_{p^\infty}) / \QQ(\mu_p))$, which we identify with $\Gal(K_\infty/\QQ)$. We fix a topological generator $\gamma$ of $\Gamma_1$, which determines an isomorphism of topological groups $\Gamma_1 \cong \Zp$.

We write $\Lambda=\cO\lb \Gamma_1\rb$, the Iwasawa algebra of $\Gamma_1$ over $\cO$. It can be identified with the formal power series ring $\cO\lb X \rb$, via the isomorphism sending $\gamma$ to $1 + X$.
We may consider $\Lambda$ as a subring of the ring $\cH$ of locally analytic $L$-valued distributions on $\Gamma_1$, which we identify with the ring of power series $F \in L\lb X \rb$ which converge on the open unit disc $|X| < 1$. Given a real number $r\ge0$, we write $\cH_r$ for the set of power series  $F=\sum_{n\ge0}c_nX^n\in\cH$ such that $\sup\frac{|c_n|_p}{n^r}<\infty$, where $|\ |_p$ denotes the $p$-adic norm of $L$ normalized by $|p|_p=p^{-1}$.

We write $\log_p\in \cH_1$ for the $p$-adic logarithm. Given an integer $m\ge1$, we define \[
\log_{p,m}=\prod_{i=0}^{m-1}\Tw^{-i}\log_p\in \cH_m.
\]
For $m=0$, we set $\log_{p,0}=1$.

We  write $\Lambda(\Gamma)=\Lambda[\Delta]$, $\cH(\Gamma)=\cH[\Delta]$ and $\cH_r(\Gamma)=\cH_r[\Delta]$. The projection maps $\Lambda(\Gamma)\rightarrow \Lambda$ and $\cH(\Gamma)\rightarrow\cH$ will be denoted by $\pi_\Delta$. We write $\Tw$ for the $\cO$-linear automorphism of $\cH(\Gamma)$ defined by $\sigma \mapsto \chi_\cyc(\sigma)\sigma$ for $\sigma\in \Gamma$, where $\chi_\cyc$ denotes the $p$-adic cyclotomic character.

For $n\ge0$, we write $\omega_{n}(X)$ for the polynomial $(1+X)^{p^{n}}-1$. We set $\Phi_0(X) = X$, and $\Phi_n(X) = \omega_n(X)/\omega_{n-1}(X)$ for $n \ge 1$. We also write $G_n:= \Gamma_1/\Gamma_1^{p^n}, \Lambda_n = \mathcal{O}[G_n]$ and denote by $\pi_n: \Lambda_n \to \Lambda_{n-1}$ the natural projection and by $ \nu_n: \Lambda_{n-1} \to \Lambda_n$ for trace map induced by $G_{n-1} \to G_n, \; \; \sigma \mapsto \sum_{ \pi_{n}(\tau)=\sigma} \tau$ for $n\ge1$. Notice that $\nu_n$ is given by multiplication by $\Phi_n(X)$.

In the rest of the paper, when no confusion arises, we simplify write $\omega_n$ and $\Phi_n$ for $\omega_n(X)$ and $\Phi_n(X)$  respectively.

\subsection{Galois representations and Dieudonn\'e modules}\label{notationmodularforms}\label{sec:basis}
Let $f$ and $g$ be two normalized cuspidal eigen-newforms of weights $k_f+2$ and $k_g+2$, levels $N_f$, $N_g$, and characters $\epsilon_f$, $\epsilon_g$ as in the introduction. Recall that we assume  $k_f > k_g \geq -1$ and  $p \nmid N_fN_g$.

Recall from the introduction that $L$  contains the coefficients of $f$ and $g$ as well as $\alpha_f \ne \beta_f$ and $\alpha_g\ne \beta_g$.
For $h\in\{f,g\}$, let $V_h$ be the Deligne $L$-linear $G_\QQ$-representation attached to $h$. It has Hodge--Tate weights $0$ and $-1-k_h$ (our normalization is that $\chi_\cyc$ has Hodge--Tate weight 1). Let $V_h^*=\Hom_L(V_h,L)$. There is a natural integral sheaf $\mathrm{TSym}^{k_h}\mathscr{H}_{\Zp}$ on the modular curve $Y_1(N_h)$ as defined in \cite[\S2.3]{KLZ2}, giving rise to a Galois-stable $\cO$-lattice $T_h^*$ inside $V_h^*$.  We define $T_h$ to be the $\cO$-linear dual of $T_h^*$. There is a natural $\cO$-basis $\omega_h$ for the $\cO$-submodule $\Fil^1\Dcris(T_h)$ (see for example \cite[\S6.1]{KLZ1}). We also have the differential  $\omega_{h^*}$ for the dual form $h^*$, as given in \cite[Definition~6.1.2]{KLZ1}, which we regard   as an element of $\Dcris(T_h^*)$. This in turn gives an $\cO$-basis of $\Fil^0\Dcris(T_h^*)$.

Let  $T$ denote the  Galois-stable $\cO$-lattice $T_f \otimes T_g$ inside the $L$-linear $G_{\QQ}$-representation $V:=V_f \otimes V_g$ associated to the convolution of $f$ and $g$. We let $T^*$,  $V^*$ denote the $\cO$-dual of $T$ and the $L$-dual of  $V$, respectively.

We now choose a basis of eigenvectors in the Dieudonn\'e modules attached to our modular forms $f$ and $g$. These elements allow us to relate Beilinson--Flach elements to $p$-adic $L$-functions and will play a role in the  definition of the Theta elements given in the next section. 

\begin{defn}\label{def:eigenvectors}
Let  $h\in\{f,g\}$ and $\lambda\in\{\alpha,\beta\}$. We define
the $\vp$-eigenvector in $\Dcris(V_h)$ (whose eigenvalue is $\lambda_h$):
\[v_{h,\lambda}=\frac{1}{\langle\vp(\omega_h),\omega_{h^*}\rangle}(\vp(\omega_h)-\lambda_h'\omega_h), \]
where $\lambda'$ is the unique element of $\{\alpha,\beta\}\setminus\{\lambda\}$ and  $\langle\sim,\sim\rangle$ denotes the natural pairing on $\Dcris(V_h)\times \Dcris(V_h^*)$. 

We define $\{v_{h,\alpha}^*,v_{h,\beta}^*\}$ to be the  basis of $\vp$-eigenvectors of $\Dcris(V_h^*)$ dual to $\{v_{h,\alpha},v_{h,\beta}\}$ under the pairing $\langle\sim,\sim\rangle$.
\end{defn}
\begin{remark}
Note that  \[
v_{h,\alpha}\equiv v_{h,\beta}\mod \Fil^1\Dcris(V_h).
\]
Dually, we have
\[
v_{h,\lambda}^*=\frac{\lambda_h}{(\lambda_h-\lambda'_h)}\left(\omega_{h^*}-\lambda'_h\vp(\omega_{h^*})\right),\]
with $v_{h,\alpha}^*+v_{h,\beta}^*\equiv 0\mod \Fil^0\Dcris(V_h^*)$. 
\end{remark}
\begin{defn}
For $\lambda,\mu\in\{\alpha,\beta\}$, we define $v_{\lambda,\mu}=v_{f,\lambda}\otimes v_{g,\mu}$ and $v_{\lambda,\mu}^*=v_{f,\lambda}^*\otimes v_{g,\mu}^*$.
\end{defn}

Let $\cT$ be a finite-rank free $\cO$-module with a continuous action of $G_F$, where $F$ is a finite  extension of $\Qp$ and $\cF$ a $p$-adic  Lie extension of $F$. Then the Iwasawa cohomology groups of $\cT$ over $\cF$ is defined by
\[  \tupH^i_{\Iw}(\cF, \cT) \coloneqq \varprojlim_{F\subset F'\subset \cF} \tupH^i(F', \cT),\]
where the inverse limit runs through finite extensions $F'/F$ such that $F'\subset \cF$ and the connection maps are corestrictions. If $\cV=\cT\otimes_{\Zp} \Qp$, we define $\tupH^i_{\Iw}(\cF, \cV)=\tupH^i_{\Iw}(\cF, \cT)\otimes_{\Zp} \Qp$. When $\cV$ is a crystalline representation of $G_{\Qp}$ with Hodge--Tate weights $\ge0$, we write 
\[
\cL_\cV:\tupH^1_{\Iw}(\Qp(\mu_{p^\infty}), \cV)\rightarrow \Dcris(\cV)\otimes \cH(\Gamma)
\]
for the Perrin-Riou map as defined in \cite[\S3.1]{LLZ2}.

Let $K$ be a number field and $\cK/K$ an infinite algebraic extension. If $\cT$ is equipped with a continuous $G_K$-action unramified outside a set of primes $\Sigma$, we define 
\[  \tupH^i_{\Iw}(\cK, \cT) \coloneqq \varprojlim_{K\subset K'\subset \cK} \tupH^i(G_{K',\Sigma}, \cT),\]
where $G_{K',\Sigma}$  denotes the Galois group of the maximal extension of $K'$ unramified outside $\Sigma$ and the inverse limit runs through finite extensions $K'/K$ such that $K'\subset \cK$ and the connection maps are corestrictions.

\section{Theta elements for Rankin--Selberg convolutions} \label{sec:thetaelements}

The main goal of this section is to define  Theta elements associated to the Rankin--Selberg convolution of $f$ and $g$ via the geometric $p$-adic $L$-functions  attached to $f$ and $g$. In Appendix \ref{arithmeticconstructionappendix}, we shall give a cohomological construction of these elements using Beilinson--Flach classes.

\subsection{Beillinson--Flach elements and $p$-adic $L$-functions}\label{S:BF}
We review the various $p$-adic $L$-functions that come up in our construction of Theta elements. We first introduce the notion of Beillinson--Flach elements.
\begin{defn}\label{defn:BF}
 For $\lambda,\mu\in\{\alpha,\beta\}$, $c > 1$ coprime to $6pN_f N_g$, $m \ge 1$ coprime to $pc$, and $a\in (\ZZ / mp^\infty \ZZ)^\times$, let
 \[
 {}_c\cBF^{\lambda,\mu}_{m, a} \in
 \cH_{\ord_p(\lambda_f\mu_g)}(\Gamma) \hat{\otimes}_{\Lambda} \HIw(\QQ(\mu_{mp^\infty}), T^*)
 \]
 be the Beillinson--Flach element as constructed in \cite[Theorem 5.4.2]{LZ1}.
\end{defn}
We will take $a = 1$ throughout.  If $\epsilon_f \epsilon_g$ is non-trivial, then we may remove the dependence on the auxiliary integer $c$. From now on, we fix a value of $c$ and drop it from the notation and, for simplicity, we  write  $\BF_{\lambda,\mu,m}$ for  $ {}_c\cBF^{\lambda,\mu}_{m, 1} $.

\begin{defn}\label{defn:geom}
Let us write $L_p(f_\alpha,g)$ and $L_p(f_\beta,g)$ for the two geometric $p$-adic $L$-functions attached to $f$ and $g$, with $f$ being the "dominant" form, by
\[
 L_p(f_\lambda,g)=\frac{\mu_g'-\mu_g}{\log_{p,1+k_g}}\langle \cL_{V^*}(\BF_{\lambda,\mu,1}),v_{\lambda,\mu'}\rangle,
\]
where we have identified $\BF_{\lambda,\mu,1}$ with its localization at $p$, the elements $\lambda,\mu$ are either $\alpha$ or $\beta$, the vector $v_{\lambda,\mu'}$ is the $\vp$-eigenvector as defined in \S\ref{sec:basis}, the element $\mu'$ denotes the unique element of $\{\alpha,\beta\}\setminus\{\mu\}$.
\end{defn}

\begin{remark}\label{rk:geo}
We recall the following properties satisfied by the geometric $p$-adic $L$-functions.
\begin{itemize}
    \item[(i)]  Since $\BF_{\lambda,\mu,1} \in
 \cH_{\ord_p(\lambda_f\mu_g)}(\Gamma) \hat{\otimes}_{\Lambda} \HIw(\QQ(\mu_{p^\infty}), T^*)$, the pairing $\langle \cL_{V^*}(\BF_{\lambda,\mu,1}),v_{\lambda,\mu'}\rangle$ gives an element in 
 \[
 \cH_{\ord_p(\lambda_f\mu_g)}(\Gamma) \hat{\otimes}_{\Lambda} \cH_{\ord_p(\lambda_f\mu_g')}(\Gamma)=\cH_{\ord_p(\lambda_f\mu_g)+\ord_p(\lambda_f\mu_g')}(\Gamma)=\cH_{2\ord_p(\lambda_f)+k_g+1}(\Gamma).
 \] Therefore, we have  $ L_p(f_\lambda,g)\in\cH_{2\ord_p(\lambda_f)}(\Gamma)$;
    \item[(ii)] Suppose that either $g$ is $p$-ordinary or $\lambda_f\in\cO^\times$. Let $\theta$ be a Dirichlet character of conductor $p^n$ and $k_g+1\le j\le k_f$
\begin{equation}
L_p(f_\lambda,g)(\theta\chi_\cyc^j)=
c_{j,n,\theta}\begin{cases}
\lambda_f^{-2n}L(f,g,\theta^{-1},j+1)&n>0,\\
\cE(\lambda,j)L(f,g,j+1)&n=0,
\end{cases}
    \label{eq:interpolation}
\end{equation}
where $c_{j,n,\theta}$ is independent of the choice of $\lambda$ and $\cE(\lambda,j)$ is given by
\[
\frac{\left(1-\frac{p^j}{\lambda_f\mu_g}\right)\left(1-\frac{p^j}{\lambda_f\mu_g'}\right)\left(1-\frac{\lambda_f'\mu_g}{p^{1+j}}\right)\left(1-\frac{\lambda_f'\mu_g'}{p^{1+j}}\right)}{\left(1-\frac{\lambda_f'}{p\lambda_f}\right)\left(1-\frac{\lambda_f'}{\lambda_f}\right)}
\]
(see \cite[Theorem~2.7.4]{KLZ2} and \cite[Theorem 6.3]{loeffler18}).
\item[(iii)]Our definition of geometric $p$-adic $L$-functions differs from the one given in \cite[Definition~3.6.4]{BLLV} by a scalar.
\end{itemize}
\end{remark}
To define our Theta elements, we will also utilize the following "extra" $p$-adic $L$-functions studied in \cite[\S3.7]{BLLV} :

\begin{defn}\label{defn:extra}
For $\lambda,\mu\in\{\alpha,\beta\}$, we define
\[
L_p^?(f_\lambda,g_\mu)=\frac{\mu_g'-\mu_g}{\log_{p,1+k_g}}\langle \cL_{V^*}(\BF_{\lambda,\mu,1}),v_{\lambda',\mu'}\rangle.
\]
\end{defn}
We finish this subsection with the following observations.
\begin{remark}\label{rk:?}
We make a number of remarks on these "extra" $p$-adic $L$-functions. 
\begin{itemize}
    \item[(i)]  Since $\BF_{\lambda,\mu,1} \in
 \cH_{\ord_p(\lambda_f\mu_g)}(\Gamma) \hat{\otimes}_{\Lambda} \HIw(\QQ(\mu_{p^\infty}), T^*)$, the pairing $\langle \cL_{V^*}(\BF_{\lambda,\mu,1}),v_{\lambda',\mu'}\rangle$ gives an element in 
 \[
 \cH_{\ord_p(\lambda_f\mu_g)}(\Gamma) \hat{\otimes}_{\Lambda} \cH_{\ord_p(\lambda_f'\mu_g')}(\Gamma)=\cH_{\ord_p(\lambda_f\mu_g)+\ord_p(\lambda_f'\mu_g')}(\Gamma)=\cH_{k_f+k_g+2}(\Gamma).
 \] Therefore, we have $L_p^?(f_\lambda,g_\mu)\in \cH_{k_f+1}(\Gamma)$;
    \item[(ii)]Our definition here differs from  the one given in \cite[\S3.7]{BLLV} by a non-zero scalar;
    \item[(iii)] Suppose either $g$ is $p$-ordinary or $\lambda_f'\in\cO^\times$. We have the following interpolation formulae.
    Let $\theta$ be a Dirichlet character of conductor $p^n$ and $k_g+1\le j\le k_f$.
Proposition~3.7.1 in \emph{op. cit.} says that
\begin{equation}
L_p^?(f_\lambda,g_\mu)(\theta\chi_\cyc^j)=R_{\lambda,\mu,j,n}\cdot L_p(f_{\lambda'},g)(\theta\chi_\cyc^j),
\label{eq:interpolate?}    
\end{equation}
where $R_{\lambda,\mu,j,n}$ is given by $(\lambda_f'/\lambda_f)^n$ when $n\ge1$. When $n=0$, it is given by
\[
R_{\lambda,\mu,j,0}=\frac{\left(1-\frac{\lambda_f'\mu_g}{p^{1+j}}\right)\left(1-\frac{p^j}{\lambda_f\mu_g}\right)}{\left(1-\frac{\lambda_f\mu_g}{p^{1+j}}\right)\left(1-\frac{p^j}{\lambda_f'\mu_g}\right)}.
\]
\end{itemize}
\end{remark}

\subsection{Definition and basic properties of Theta elements} \label{sec:Theta}

We first introduce the following projection of the geometric $p$-adic $L$-function as given in Definition~\ref{defn:geom}.
\begin{defn}
If $j$ is an integer such that $k_g+1\le j\le k_f$ and $\lambda\in\{\alpha,\beta\}$, we define 
\[
L_p(\lambda,j,n)\in L[G_n]=\cH/\omega_n(X)
\]
to be the natural image of $\pi_\Delta\circ\Tw^jL_p(f_\lambda,g)$ modulo $\omega_n(X)$.

Similarly, for $\lambda,\mu\in\{\alpha,\beta\}$, we define 
\[
L_p^?(\lambda,\mu,j,n)\in L[G_n]
\]
to be the natural image of $\pi_\Delta\circ \Tw^jL_p^?(f_\lambda,g_\mu)$ modulo $\omega_n(X)$. 
\end{defn}

The following lemma allows us to study the integrality of $L_p(f_\lambda,g)$. It is a well-known result, but we include a proof here due to our ignorance of a proper reference in the literature.
\begin{lemma}\label{lem:modulo}
Let $F\in\cH_r$, which we identify with a power series $\sum_{m\ge0}c_mX^m\in L[[X]]$.
Then, there exists an integer $s$ such that $F$ is congruent to a polynomial in $\varpi^{-s}p^{-rn}\cO[X]$ modulo $\omega_n$ for all $n\ge1$.
\end{lemma}
\begin{proof}
Let $m,n\ge1$ be integers and write $m=kp^n+t$, where $k$ and $t$ are integers with $k\ge 0$ and $0\le t<p^n$.
Consider 
\[
X^m=\left(\left((1+X)-1\right)^{p^n}\right)^kX^t.
\]
Note that $\left((1+X)-1\right)^{p^n}$ is congruent to a polynomial in $p\cO[X]$ modulo $\omega_n(X)$ by the binomial theorem. Thus, we may replace $X^m$ by a polynomial  $Q_m(X)\in p^k\cO[X]$.

By definition, there exists a constant $C$, independent of $m$ such that $|c_m|_p\le Cm^r$. Without loss of generality, we may assume that $C\in p^\ZZ$. Let $\ord_p$ denote the $p$-adic valuation on the integers. If $s$ is the smallest integer satisfying  $p^s\ge m^r$, then
\[
c_m\in p^{-\ord_p(C)-s}\cO=C^{-1}p^{-s}\cO.
\]

We now estimate $s$. Note that $m=kp^n+t<(k+1)p^n$ by the definition of $t$. Thus, $m <p^{\log(k+1)/\log p+n}$, where $\log$ denotes the natural logarithm on positive real numbers.  Hence, $m^r<p^{r(\log(k+1)/\log p+n)}$. The  minimality of $s$ then gives
$$s\le \lfloor r\log(k+1)/\log p\rfloor+rn,$$
where $\lfloor\ \rfloor$ is the floor function on the real numbers.
Therefore,
\[
c_m\in C^{-1}p^{-\lfloor r\log(k+1)/\log p\rfloor-rn}\cO,
\]
which implies that
\[
c_mQ_m(X)\in \left(C^{-1}p^{k-\lfloor r\log(k+1)/\log p\rfloor} \right)p^{-rn}\cO[X].
\]
Since $k-\lfloor r\log(k+1)/\log p\rfloor$ is bounded below as $k$ varies, the lemma follows.
\end{proof}

\begin{remark}\label{rk:integral}
It follows from Lemma~\ref{lem:modulo} and Remark~\ref{rk:geo}(i) that there exists a constant $s$ such that
\[
\lambda_f^{2n}L_p(\lambda,j,n)\in \varpi^{-s}\Lambda_n
\]
for both $\lambda\in\{\alpha,\beta\}$, all $k_f+1\le j\le k_f$ and $n\ge 0$. 

When $\alpha_f=-\beta_f$, we have $\ord_p(\alpha_f)=\ord_p(\beta_f)=\frac{k_f+1}{2}$. Thus, it follows from Lemma~\ref{lem:modulo} and Remark~\ref{rk:?}(i) that $\lambda_f^{2n} L_p^?(\lambda,\mu,j,n)\in \varpi^{-s}\Lambda_n$.
\end{remark}

We are now ready to define our Theta elements. Note that we have to consider the cases $\alpha_f\ne-\beta_f$ and $\alpha_f=-\beta_f$ separately. 
\begin{defn}\label{analyticdeftheta}
Let $n\ge 0$ and $k_g+1\le j\le k_f$. If $\alpha_f\ne-\beta_f$, we define 
\[
\Theta_{j,n}=
\frac{1}{\beta_f^2-\alpha_f^2} \left(\beta_f^{2n+4}L_p(\beta,j,n)-\alpha_f^{2n+4}L_p(\alpha,j,n) \right).
\]
In the case $\alpha_f=-\beta_f$, we define
\begin{align*}
\Theta_{j,n}^+&=\frac{\alpha_f^{2n+2}}{4}\left(L_p(\alpha,j,n)+L_p(\beta,j,n)+L_p^?(\beta,\beta,j,n)+L_p^?(\alpha,\beta,j,n)\right)\\
\Theta_{j,n}^-&=\frac{\alpha_f^{2n+2}}{4}\left(L_p(\alpha,j,n)+L_p(\beta,j,n)-L_p^?(\beta,\beta,j,n)-L_p^?(\alpha,\beta,j,n)\right).
\end{align*}
\end{defn}

\begin{lemma}\label{lem:integral}
If $\alpha_f\ne- \beta_f$, then there exists an integer $s$ such that $\Theta_{j,n}\in \varpi^{-s}\Lambda_n$ for all $j$ and $n$. Similarly, the same is true for $\Theta_{j,n}^\pm$ when $\alpha_f=-\beta_f$.
\end{lemma}
\begin{proof}
This follows immediately from Remark~\ref{rk:integral}.
\end{proof}

\begin{remark}We comment on the definition of $\Theta_{j,n}^\pm$, which might look unnatural to the reader at first sight.
\begin{itemize}
    \item[(i)] In Corollary~\ref{cor:generalizing-pollack} below, we show that these elements are closely related to certain plus and minus $p$-adic $L$-functions, which can be regarded as a generalization of \cite[Proposition~6.18]{pollack03} on the Mazur--Tate elements attached to an elliptic curve.
    \item[(ii)] It turns out that these two elements are sufficient to study the Selmer group over $K_n$ when $f$ corresponds to a $p$-supersingular elliptic curve and $g$ is of weight one (see Theorem~\ref{thm:final} below).
    \item[(iii)] In our quest to generalize the work of Kim--Kurihara \cite{KK}, we consider two different settings. The first is when both $f$ and $g$ are $p$-ordinary and the second is the setting discussed in (ii) above. In the ordinary setting, it turns out that we may study the Selmer groups using one single $p$-adic $L$-function and the Theta elements $\Theta_{j,n}$ are  sufficient. In the setting of (ii), we have a fairly straightforward control theorem for the plus and minus Selmer groups (see Theorem~\ref{thm:controlpm}), generalizing the work of Kobayashi \cite[Theorem~9.3]{kobayashi03}. This in turn allows us to apply the techniques of Kim--Kurihara in \cite{KK} to prove Theorem~\ref{thm:final}. In all other settings, the lack of control theorem stops us from studying Fitting ideals of Selmer groups using the techniques of \emph{loc. cit.} 
    \item[(iv)] Our calculations suggest that when the "dominant" form $f$ is $p$-non-ordinary and the "non-dominant" form $g$ is $p$-ordinary, two Theta elements are required to study the Fitting ideal of the Selmer group over $K_n$. It would seem reasonable to expect that if  both $f$ and $g$ are non-ordinary at $p$, one would need to consider four linearly independent combinations of $p$-adic $L$-functions to define the appropriate Theta elements in order to study the Fitting ideals of Selmer groups over $K_n$. In the case $\alpha_f=-\beta_f$, we might consider defining two extra Theta elements of the form
    \begin{align*}
&\frac{\alpha_f^{2n+2}}{4}\left(L_p(\alpha,j,n)-L_p(\beta,j,n)-L_p^?(\beta,\beta,j,n)+L_p^?(\alpha,\beta,j,n)\right),\\
&\frac{\alpha_f^{2n+2}}{4}\left(L_p(\alpha,j,n)-L_p(\beta,j,n)+L_p^?(\beta,\beta,j,n)-L_p^?(\alpha,\beta,j,n)\right).
\end{align*}
 When $\alpha_f\ne-\beta_f$, the linear combinations might involve the logarithmic matrices defined in \cite[\S5.1]{BLLV}.
\end{itemize}
\end{remark}

The Theta elements defined in Definition~\ref{analyticdeftheta} satisfy the following interpolation formulae, which can be regarded as a generalization of \cite[(1)]{MaTa}.

\begin{lemma}\label{lem:evaluateTheta}
Suppose that $g$ is $p$-ordinary. Let $\theta$ be a non-trivial character on $G_n$ of conductor $p^m$ (so that $1\le m\le n+1$). If $\alpha_f\ne -\beta_f$, then 
\[
\Theta_{j,n}(\theta)=\frac{\beta_f^{2n-2m+4}-\alpha_f^{2n-2m+4}}{\beta_f^2-\alpha_f^2}\cdot c_{j,m,\theta}\cdot L(f,g,\theta^{-1},j+1).
\]
If $\alpha_f=-\beta_f$, then
\[
    \Theta_{j,n}^\pm(\theta)=\frac{\alpha_f^{2n-2m+2}(1\pm(-1)^m)}{2}\cdot c_{j,m,\theta}\cdot L(f,g,\theta^{-1},j+1).
\]
\end{lemma}
\begin{proof}
This follows from \eqref{eq:interpolation} and \eqref{eq:interpolate?}.
\end{proof}

It is of course possible to write down the values of the Theta elements evaluated at the trivial character using  \eqref{eq:interpolation} and \eqref{eq:interpolate?}. But since the formulae are much more tedious and we will only need the one for the Theta element $\Theta_{j,n}$ in \S \ref{ordinarysectionmaintheorems}, we have decided not to write them down here and to refer the interested reader to Lemma \ref{auxiliarylemma}.

While we have defined our Theta elements in terms of geometric  $p$-adic $L$-functions, we may go the other way to describe the $p$-adic $L$-functions in terms of these Theta elements. This is analogous to the relation  atisfied by the Mazur--Tate elements for modular forms as given in \cite[(10.2)]{MTT} (at least in the case where $\alpha_f\ne\beta_f$).

\begin{proposition}\label{prop:Theta-L}
Suppose that $g$ is $p$-ordinary. Let $n\ge 1$ be an integer. If $\alpha_f\ne -\beta_f$, we have
\[
L_p(\lambda,j,n)=\frac{1}{\lambda_f^{2n+2}}\left(\Theta_{j,n}-\frac{(\lambda_f')^2}{p}\nu_n(\Theta_{j,n-1})\right).
\]
In the case where $\alpha_f=-\beta_f$, we have 
\[
L_p(\alpha,j,n)+L_p(\beta,j,n)=\frac{2}{\lambda_f^{2n+2}}\left(\Theta^+_{j,n}+\Theta^{-}_{j,n}\right).
\]
\end{proposition}
\begin{proof}
The case $\alpha_f=-\beta_f$ is clear. We  assume that $\alpha_f\ne -\beta_f$ in the rest of the proof.

 Since we are comparing polynomials of degree $<p^n$. It is enough to show that these polynomials give the same values when evaluated at all characters of $G_n$. Moreover, recall that the trace map $\nu_n: \Lambda_{n-1} \to \Lambda_n$ (defined in \S \ref{sec:Iwalg}) is given by multiplication by $\Phi_n(X)$, thus evaluating $\nu_n(\Theta_{j,n-1})$ at a character $\theta$ of $G_n$ is the same as evaluating $\Phi_n(X)\Theta_{j,n-1}$ at $\theta$.  Without loss of generality, we may assume that $\lambda=\alpha$.

 Let $\theta$ be a character on $G_n$ of conductor $p^m$ (so that $m\le n+1$). Notice that $\theta$ sends $1+X$ to a primitive  $p^{m-1}$-th root of unity. If $m=n+1$, then $\theta$ vanishes at $\Phi_n(X)=((1+X)^{p^n} - 1) /((1+X)^{p^{n-1}} -1)$ as $\theta(1+X)$ is a primitive  $p^{n}$-th root of unity. Lemma~\ref{lem:evaluateTheta} tells us that when we evaluate the right-hand side at $\theta$, we have
\[
\alpha_f^{-2n-2}\Theta_{j,n}(\theta)=\alpha_f^{-2n-2}c_{j,n+1,\theta}L(f,g,\theta^{-1},j+1).
\]
This agrees with the value on the left-hand side by \eqref{eq:interpolation}.

Suppose now that $m<n+1$. In this case, as $\theta$ sends $1+X$ to a primitive  $p^{m-1}$-th root of unity, the value of $\Phi_n(X)$ at $\theta$ is $p$. We deduce that \[\frac{1}{\alpha_f^{2n+2}}\left(\Theta_{j,n}-\frac{\beta_f^2\Phi_n(X)}{p}\Theta_{j,n-1}\right) \] when evaluated at $\theta$ gives
\begin{align*}
&\ \frac{1}{\alpha_f^{2n+2}}\left(\Theta_{j,n}(\theta)-\beta_f^2\Theta_{j,n-1}(\theta)\right)\\
=&\ \frac{\left(\beta_f^{2n+4}L_p(f_\beta,g)-\alpha_f^{2n+4}L_p(f_\alpha,g)-\beta_f^{2n+4}L_p(f_\beta,g)+\alpha_f^{2n+2}\beta_f^2L_p(f_\alpha,g)\right)(\theta\chi_\cyc^j)}{\alpha_f^{2n+2}(\beta_f^2-\alpha_f^2)} \\
=&\ L_p(f_\alpha,g)(\theta\chi_\cyc^j)\\
=&\ L_p(\alpha,j,n)(\theta)
\end{align*}
as required.
\end{proof}

\section{Selmer groups and their structure} \label{sec:selmer}
The goal of this section is to first review the definitions of the various Selmer groups we are interested in. We will then show that the Selmer groups over $K_\infty$ contain no non-trivial $\Lambda$-submodule of finite index. We will treat the cases where $f$ is $p$-ordinary and $p$-non-ordinary  distinctly and separately. 
 
\subsection{Definitions of Selmer groups in the ordinary setting} \label{subsec:doublyordinaryselmer}
In this section, we assume that  the eigenform $f$ is $p$-ordinary, with $\alpha_f$ being the unit root of the Hecke polynomial at $p$. Note that we do not have to assume that $g$ is $p$-ordinary here. For $h \in \{ f,g \}$, we denote $A_h:= V_h / T_h$; similarly we set $A:=V/T$. As $f$ is $p$-ordinary, there is a $G_{\QQ_{p}}$-stable filtration (cf. \cite[(4.2)]{EPW})
\begin{equation*}
0 \rightarrow \mathcal{F}^{+}A_f \rightarrow A_f \rightarrow A_f/\mathcal{F}^{+}A_f \rightarrow 0,
\end{equation*}
where $\mathcal{F}^{+}A_{f} \subset A_{f}$ is of corank 1 over $\cO$ and $A_{f} / \mathcal{F}^{+}A_{f}$ is unramified.
\begin{notation}
For any integer $k_g + 1 \leq j \leq k_f$, denote $A_j:= A(1+j)$ and set
\[ \mathcal{F}^{+}A_j :=  \mathcal{F}^{+}A_{f}\otimes A_{g}(1+j)\subset A_j.  \]
\end{notation} 

Let $K$ be any finite extension of $\QQ$ contained in the cyclotomic $\Zp$-extension $K_{\infty}$. For any place $\nu$ of $K$, we define
\begin{align*}
      \tupH^{1}_{f}(K_{\nu}, A_j):=& \textup{ker}\big(\tupH^{1}(K_{\nu}, A_j) \rightarrow \tupH^{1}(I_{\nu}, A_j)\big)_{\rm div}  \text{ if } \nu \nmid p ,\\
\tupH^{1}_{\Gr}(K_{\nu}, A_j):=&\tupH^{1}(K_{\nu}, \mathcal{F}^{+}A_j)_{\rm div}  \; \; \; \; \; \;\; \; \; \; \; \;\; \; \; \; \; \; \;\; \; \;\;\text{ if } \nu \mid p.
\end{align*}
Here $I_{\nu}$ denotes the inertia group at the place $\nu$, while the subscript div stands for the maximal divisible subgroup. For primes $\nu \mid p$, the local condition is called the \emph{Greenberg condition}.
Further, we denote by 
\[ \tupH^{1}_{/\mathcal{L}_{\nu}}(K_{\nu}, A_j) := \frac{\tupH^{1}(K_{\nu}, A_j)}{\tupH^{1}_{\mathcal{L}_{\nu}}(K_{\nu}, A_j)}, \]
for $\mathcal{L}_{\nu}$ equal to $f$ if $\nu \nmid p$ or to $\Gr$ otherwise.

\begin{defn}\label{defn:selmergroupsordinary} Let $\Sigma$ be a finite set containing $p,\infty$, and the primes dividing $N_fN_g$ and denote by $\QQ_\Sigma$ the maximal extension of $\QQ$ unramified outside $\Sigma$. Then, we define
\begin{itemize}
\item[(i)]    the Greenberg Selmer group of $A_j$ over $K$ by 
\[ \textup{Sel}_{\Gr}(K,A_j) := \textup{ker} \Bigg( \tupH^{1}(\QQ_{\Sigma}/K, A_j) \rightarrow \prod\limits_{\nu\nmid p} \tupH^{1}_{/ f}(K_{\nu}, A_j) \times \prod\limits_{\nu \mid p}\tupH^{1}_{/\Gr}(K_{\nu}, A_j)   \Bigg), \]
where the first product runs through all primes of $K$ above $\Sigma\setminus\{p\}$. 
\item[(ii)]  the Greenberg Selmer group over $K_\infty $ by
\[\textup{Sel}_{\Gr}(K_\infty,A_j) := \varinjlim_{\QQ \subseteq K \subseteq K_{\infty}} \textup{Sel}_{\Gr}(K,A_j)\]
where the limits are taken over all finite subextensions of $K_\infty$ with respect to the restriction maps.
\end{itemize}
\end{defn}

\begin{remark}
For a prime $w$ of $K_{\infty}$, we can describe the local conditions defining the Selmer group $\textup{Sel}_{\Gr}(K_\infty,A_j)$ explicitly. At $w \nmid p$, it is given by
\[
\tupH^{1}_{f}(K_{\infty,w}, A_j) = 
\textup{ker}(\tupH^{1}(K_{\infty,w}, A_j) \rightarrow \tupH^{1}(I_{w}, A_j))_{\rm div}.\]
At $ w \mid p$, the limits of $\tupH^{1}(K_{n,w},\cF^+ A_j)_{\rm div}$ and  $\tupH^{1}(K_{n,w}, \cF^+A_j)$ as $n\to \infty$ are the same, which shows that 
\[ \tupH^{1}_{\Gr}(K_{\infty,w}, A_j) = \tupH^{1}(K_{\infty,w}, \mathcal{F}^{+}A_j).\]
\end{remark}

We conclude by recalling the Iwasawa main conjecture for $A_j$.

\begin{conjecture}
The Pontryagin dual of the Selmer group $\textup{Sel}_{\Gr}(K_\infty,A_j)$ is torsion over $\Lambda$ and its characteristic ideal is generated by the $p$-adic $L$-function $\pi_\Delta\circ{\rm Tw}^{j}L_p(f_\alpha,g)$.
\end{conjecture}

\begin{remark}\label{cnjordoneineq}
 If, in addition to our running hypotheses, we further assume that $g$ is $p$-ordinary, $f$ and $g$ are non-Eisenstein at $p$, that a big image assumption holds (cf. \cite[Hypothesis 11.1.1 and Remark 11.1.3]{KLZ2}), and that the $p$-adic $L$-function $\pi_\Delta\circ{\rm Tw}^{j}L_p(f_\alpha,g)$ is non-trivial, we can apply  \cite[Theorem 11.6.4]{KLZ2}, to deduce the cotorsionness of 
$\textup{Sel}_{\Gr}(K_\infty,A_j)$ and the inclusion of the main conjecture \[\pi_\Delta\circ {\rm Tw}^{j}L_p(f_\alpha,g) \in \textup{char}_\Lambda \textup{Sel}_{\Gr}(K_\infty,A_j)^\vee. \]
We conclude with the following remarks. \begin{itemize}
    \item The ``no exceptional zero'' hypothesis (NEZ) in Theorem 11.6.4 in \emph{op. cit.} is automatically satisfied in our setting as the weight of $f$ is strictly larger than the one of $g$.
    \item  The Selmer group appearing in Theorem 11.6.4 in \emph{op. cit.} is defined by using  the Selmer complex (introduced by Nekov\'a\v{r} in \cite{SelmerComplexes}) with the desired Greenberg local condition at $p$; it is isomorphic to $\textup{Sel}_{\Gr}(K_\infty,A_j)^\vee$ by \cite[Proposition 11.2.8]{KLZ2}.
\end{itemize} 
\end{remark}

\subsection{Plus and minus Selmer groups in the supersingular setting}\label{sec:setupSS}

In this section, we assume that $f$ corresponds to an elliptic curve $E/\QQ$ which has good supersingular reduction at $p$ with $a_p(E)=0$ and that $g$ is a weight-one form (so that $k_f=0$ and $k_g=-1$). Thus, the only integer $j$ satisfying $k_g+1\le j\le k_f$ is $j=0$. The Theta elements we defined in \S\ref{sec:Theta}  interpolate the $L$-values $L(E,g,\theta^{-1},1)$. Furthermore, $\beta_f=-\alpha_f$ with $\alpha_f^2=\beta_f^2=-p$. 

\subsubsection{Definitions of plus and minus Selmer groups}

Since $g$ is of weight one, the local representation  $T_g^*|_{G_{\Qp}}$ decomposes into a direct sum of two characters. This gives
\begin{equation}
    T^*|_{G_{\Qp}}\cong T_{f}^*|_{G_{\Qp}}(\theta_\alpha)\oplus T_{f}^*|_{G_{\Qp}}(\theta_\beta),
\label{eq:localdecompowt1}
\end{equation}
where $\theta_\mu$ denotes the finite unramified $L$-valued characters of $G_{\Qp}$ such that $\vp$ acts on $\Dcris(L(\theta_\mu))$ by multiplication by $\mu_g^{-1}$. For notational simplicity, we write $T_{f,\mu}^*$ for the $G_{\Qp}$-representation $T_{f}^*|_{G_{\Qp}}(\theta_\mu)$ and let $V_{f,\mu}^*$ denote $T_{f,\mu}^*\otimes\Qp$. It can be verified that the characteristic polynomial of $\vp|_{\Dcris(V_{f,\mu}^*)}$ is given by
\begin{equation}
    (X-\alpha_f^{-1}\mu_g^{-1})(X-\beta_f^{-1}\mu_g^{-1})=X^2+\left(\mu_g^2\epsilon_f(p)p^{k_f+1}\right)^{-1}.
\label{eq:charpoly}
\end{equation}

For $\lambda,\mu\in\{\alpha,\beta\}$, let 
\[
\cL_{\lambda,\mu}:\HIw(\Qp(\mu_{p^\infty}),T^*)\rightarrow \cH_{\ord_p(\lambda_f\mu_g)}(\Gamma)
\]
denote the map given by $(\mu'_g-\mu_g)\langle\cL_{T^*}(-),v_{\lambda,\mu}\rangle$.
By an abuse of notation, we shall write $\cL_{\lambda,\mu}$ for the composition
\[
\HIw(\Qp(\mu_{p^\infty}),T_{f,\mu}^*)\hookrightarrow \HIw(\Qp(\mu_{p^\infty}),V)\stackrel{\cL_{\lambda,\mu}}{\longrightarrow} \cH_{\ord_p(\lambda_f\mu_g)}(\Gamma).
\]
In \cite[\S3.4]{Lei2011}, it has been shown that there exist Coleman maps
\[
\col^\pm_\mu:\HIw(\Qp(\mu_{p^\infty}),T_{f,\mu}^*)\rightarrow \Lambda(\Gamma)
\]
such that
\begin{equation}
    \cL_{\lambda,\mu}=\log_p^+\col^+_\mu+\lambda_f\log_p^-\col_{\mu}^-,
\label{eq:decompose}
\end{equation}
where  $\log_p^\pm$ are Pollack's plus and minus logarithms defined in \cite{pollack03}.

\begin{remark}
We make a couple of remarks on the construction of the maps $\col^\pm_\mu$.
\begin{itemize}
    \item[(i)] In \cite{Lei2011}, we only studied the $p$-adic representation attached to a modular form. But the construction of Coleman maps is purely local and applies to any 2-dimensional crystalline representation whose trace of the Frobenius on the Dieudonn\'e module is zero. In particular, it applies to $T_{f,\mu}^*$ thanks to \eqref{eq:charpoly}.
    \item[(ii)]Note that a priori, the Coleman maps in \cite{Lei2011} take values in $\Lambda(\Gamma)\otimes L$. But we can show that they land inside $\Lambda(\Gamma)$ using the theory of Wach modules (see for example \cite[\S3.1]{LLZ1}). The link between $\col_{\mu}^\pm$ and $\cL_{\lambda,\mu}$ stated in \eqref{eq:decompose} follows from combining \cite[equation (2)]{LLZ2} and \cite[Corollary~5.11]{LLZ1}. 
\end{itemize}
\end{remark}

We may describe explicitly  $\ker\col^\pm_{\mu}$. Let us first introduce some notation and prove a preliminary lemma. 

\begin{defn}
Let $T_{f,\mu}=T_f|_{G_{\Qp}}(\theta_\mu^{-1})$ and write $A_{f,\mu}=T_{f,\mu}\otimes L/\cO$ (so that $A_{f,\mu}(1)=E[p^\infty](\theta_\mu^{-1})$ and $A_{f,\mu}$ is the Pontryagin dual of $T_{f,\mu}^*$).
\end{defn}
\begin{lemma}\label{lem:restriction}
The restriction maps
\begin{align*}
\tupH^1(\Qp(\mu_{p^m}),T_{f,\mu}(1))&\rightarrow \tupH^1(\Qp(\mu_{p^n}),T_{f,\mu}(1)),\\
\tupH^1(\Qp(\mu_{p^m}),A_{f,\mu}(1))&\rightarrow \tupH^1(\Qp(\mu_{p^n}),A_{f,\mu}(1))
\end{align*}
are injective for all $m\le n$.
\end{lemma}
\begin{proof}
Recall that $\theta_\mu$ is an unramified character. Thus, $\tupH^0(\Qp(\mu_{p^n}),A_{f,\mu}(1))=0$ since $E$ is supersingular at $p$, which means that it admits no $p$-torsion over $K(\mu_{p^n})$ for any unramified extension $K/\Qp$ (see \cite[Proposition~8.7] {kobayashi03} and \cite[Proposition~3.1]{KO}). It then follows that $\tupH^0(\Qp(\mu_{p^n}),T_{f,\mu}(1))=0$ as well. The injectivity of the restriction maps is now a consequence of the inflation-restriction exact sequence.
\end{proof}

Via Lemma~\ref{lem:restriction}, we may identify $\tupH^1(\Qp(\mu_{p^m}),T_{f,\mu}(1))$ and $\tupH^1(\Qp(\mu_{p^m}),A_{f,\mu}(1))$ as sub-modules of \\ $ \tupH^1(\Qp(\mu_{p^n}),T_{f,\mu}(1))$ and $\tupH^1(\Qp(\mu_{p^m}),A_{f,\mu}(1))$ respectively for all $m\le n$.

\begin{defn}
Let $S_n^+=[0,n-1]\cap 2\ZZ$ and $S_n^-=[0,n-1]\cap (2\ZZ+1)$. Define $\tupH^1_\pm(\Qp(\mu_{p^n}),T_{f,\mu}(1))$ to be
\[
\left\{x\in \tupH^1_f(\Qp(\mu_{p^n}),T_{f,\mu}(1)):\cor_{n/m+1}(x)\in \tupH^1_f(\Qp(\mu_{p^m}),T_{f,\mu}(1))\ \forall m\in S_n^\pm\right\}.
\]
Here, $\tupH^1_f$ denotes Bloch-Kato's subgroup in $\tupH^1$ defined in \cite[(3.7.2)]{BK}.
\end{defn}
\begin{remark}\label{rk:duality}
In \cite[\S4.4 and \S4.5]{Lei2011}, we have proved that the annihilator of $\ker\col^\pm_{\mu}$ under the  local Tate duality 
\[
\HIw(\Qp,T_{f,\mu}^*)\times \tupH^1(\Qp(\mu_{p^{\infty}}),A_{f,\mu}(1))\rightarrow L/\cO
\]
is given by the image of 
\[
\varinjlim \tupH^1_\pm(\Qp(\mu_{p^n}), T_{f,\mu}(1))\otimes L/\cO
\]
inside $\tupH^1(\Qp(\mu_{p^{\infty}}),A_{f,\mu}(1))$.
\end{remark}

In what follows, we give an alternative description of these plus and minus groups in terms of points on the elliptic curve $E$, following closely the work of Kobayashi \cite{kobayashi03}.

\begin{defn}
Let $K/\Qp$ be a finite unramified extension. We define
\begin{equation}\label{eq:Epm}
   E^\pm(K(\mu_{p^n}))=\left\{P\in \hat E(K(\mu_{p^n})):\Tr_{n/m+1}P\in \hat E(K(\mu_{p^m}))\ \forall m\in S_n^\pm\right\},
\end{equation}
where $\hat{E}$ denotes the formal group of $E$ at $p$, $\hat E(K(\mu_{p^m}))$ denotes the points on $\hat E$ defined over the maximal ideal of $K(\mu_{p^m})$ and $\Tr_{n/m+1}:\hat{E}(K(\mu_{p^n}))\rightarrow\hat{E} (K(\mu_{p^{m+1}}))$ is the trace map on the formal group.
\end{defn}

\begin{conv}\label{conv:pm}
 Note that our choice of signs is the same as the one given in \cite{Lei2011}, which corresponds to the choice of signs for the plus and minus $p$-adic $L$-functions  given in \cite{pollack03}. But it is opposite to the one given in \cite{kobayashi03,Kim,KO}. In particular, we emphasize that our groups $E^\pm(K(\mu_{p^n}))$ correspond to $E^\mp(K(\mu_{p^n}))$ given in \cite[\S3.1.2]{KK} when $K=\Qp$. 
\end{conv}

It has been proved in \cite[Proposition~8.12]{kobayashi03} and \cite[Proposition~3.16]{KO} that there is a short exact sequence
\begin{equation}\label{eq:pmSES}
    0\rightarrow \hat{E}(K)\rightarrow E^+(K(\mu_{p^n}))\oplus E^-(K(\mu_{p^n}))\rightarrow \hat E(K(\mu_{p^n}))\rightarrow 0,
\end{equation}
where the first map is the diagonal embedding and the second map is given by $(x, y)\mapsto x-y$.

\begin{lemma}\label{lem:projectlambda}
Let $\mu\in\{\alpha,\beta\}$ and $K$ be the finite unramified extension of $\Qp$ given by $(\overline{\Qp})^{\ker\theta_\mu }$. Let $\Phi_\mu$ be the composition
\[
\tupH^1(K(\mu_{p^n}),T_f(1))\cong \tupH^1\left(\Qp(\mu_{p^n}),T_f(1)\otimes\Ind_K^{\Qp} \mathbf{1}\right)\rightarrow \tupH^1\left(\Qp(\mu_{p^n}),T_{f,\mu}(1)\right),
\]
where the first isomorphism is given by Shapiro's lemma and the second map is induced by the natural projection $\Ind_K^{\Qp} \mathbf{1}\rightarrow \theta_\mu$. Then,
\[
\Phi_{\mu}\left(E^\pm(K(\mu_{p^n}))\right)=\tupH^1_\pm(\Qp(\mu_{p^n}),T_{f,\mu}(1)),
\]
where $\hat{E}(K(\mu_{p^n}))$ is identified with $\tupH^1_f(K(\mu_{p^n}),T_f(1))$ via the Kummer map.
\end{lemma}
\begin{proof}
Note that $T_f(1)$ is the $p$-adic Tate module of $E$. We recall \cite[Example 3.10.1]{BK} that $\hat{E}(K(\mu_{p^n}))$ can be identified with $\tupH^1_f(K(\mu_{p^n}),T_f(1))$ under the Kummer map. It follows from the definition of $\tupH^1_f$ that $\tupH^1_f(\Qp,M_1\oplus M_2)=\tupH^1_f(\Qp,M_1)\oplus \tupH^1_f(\Qp,M_2)$ for any $G_{\Qp}$-modules $M_1$ and $M_2$. Thus, the map $\Phi_{\mu}$ sends $\tupH^1_f(K(\mu_{p^n}),T_f(1))$ onto $\tupH^1_f(\Qp(\mu_{p^n}),T_{f,\mu}(1))$. Furthermore, $\Phi_\mu$ commutes with the trace maps. Hence our result follows.
\end{proof}

\begin{defn}
Let $\bullet,\circ\in\{+,-\}$. We define
\[
\tupH^1_{\bullet\circ}(\Qp(\mu_{p^n}),T(1))=\tupH^1_\bullet(\Qp(\mu_{p^n}),T_{f,\alpha}(1))\oplus \tupH^1_{\circ}(\Qp(\mu_{p^n}),T_{f,\beta}(1)),
\]
which can be realized as subgroups of $\tupH^1_f(\Qp(\mu_{p^n}),T(1))$ via \eqref{eq:localdecompowt1}. We define $\tupH^1_{\bullet\circ}(\Qp(\mu_{p^n}),A(1))$ to be the image of $\tupH^1_{\bullet\circ}(\Qp(\mu_{p^n}),T(1))\otimes\Qp/\Zp$ inside $\tupH^1_f(\Qp(\mu_{p^n}),A(1))$.
\end{defn}

\begin{remark}
Suppose $\bullet=\circ$, then  $\tupH^1_{\bullet\bullet}(\Qp(\mu_{p^n}),T(1))$ is given by
\[
\left\{x\in \tupH^1_f(\Qp(\mu_{p^n}),T(1)):\cor_{n/m+1}(x)\in \tupH^1_f(\Qp(\mu_{p^m}),T(1))\ \forall m\in S_n^\bullet\right\}.
\]
\end{remark}

\begin{corollary}\label{cor:pmSES}
Let $\bullet,\circ\in\{+,-\}$ and write $\bullet'$ (resp. $\circ'$) for the unique element of $\{+,-\}\setminus \{\bullet\}$ (resp. $\{+,-\}\setminus \{\circ\}$)  We have a short exact sequence
\[
0\rightarrow \tupH^1_f(\Qp,T(1))\rightarrow \tupH^1_{\bullet\circ}(\Qp(\mu_{p^n}),T(1))\oplus \tupH^1_{\bullet'\circ'}(\Qp(\mu_{p^n}),T(1))\rightarrow \tupH^1_f(\Qp(\mu_{p^n}),T(1))\rightarrow 0.
\]
\end{corollary}
\begin{proof}
This follows from combining \eqref{eq:pmSES} and Lemma~\ref{lem:projectlambda}.
\end{proof}

\begin{defn}Let $\bullet,\circ\in\{+,-\}$.\label{defn:pmSel}
\begin{itemize}
    \item[(i)] Let $\Sigma$ be the set of primes given in Definition~\ref{defn:selmergroupsordinary}. We define the plus and minus Selmer groups $\Sel_{\bullet\circ}(\QQ(\mu_{p^n}),A(1))$ by 
\[
\ker\left(\tupH^1(\QQ_\Sigma/\QQ(\mu_{p^n}),A(1))\rightarrow\prod_{\nu\nmid p}\tupH^1_{/f}(\QQ(\mu_{p^n})_\nu,A(1))\times \frac{\tupH^1(\Qp(\mu_{p^n}),A(1))}{\tupH^1_{\bullet\circ}(\Qp(\mu_{p^n}),A(1))}\right).
\]
and we define $\Sel_{\bullet\circ}(\QQ(\mu_{p^\infty}),A(1))=\varinjlim_n\Sel_{\bullet\circ}(\QQ(\mu_{p^n}),A(1))$.

\item[(ii)] For all $0\le n\le \infty$, we define $\Sel_{\bullet\circ}(K_n,A(1))$ to be  the $\mathbf{1}$-isotypic component of $\Sel_{\bullet\circ}(\QQ(\mu_{p^{n+1}}),A(1))$, where $\mathbf{1}$ is the trivial character on $\Delta$.
\end{itemize}

\end{defn}
The following lemma about the local quotient in the definition of the plus and minus Selmer groups will be used in \S\ref{subsec:nonexistenceoflambdasubmoduleoffiniteindex} below.

\begin{lemma}\label{lem:corankpm}
For $\bullet,\circ\in\{+,-\}$, the Pontryagin dual of $\displaystyle\left(\frac{\tupH^1(\Qp(\mu_{p^\infty}),A(1))}{\tupH^1_{\bullet\circ}(\Qp(\mu_{p^\infty}),A(1))}\right)^\Delta$ is of rank 2 over $\Lambda$.
\end{lemma}
\begin{proof}
It is enough to show that the Pontryagin dual of the summand $\displaystyle\left(\frac{\tupH^1(\Qp(\mu_{p^\infty}),A_{f,\mu}(1))}{\tupH^1_\pm(\Qp(\mu_{p^\infty}),A_{f,\mu}(1))}\right)^\Delta$ is  of rank one over $\Lambda$. By Remark~\ref{rk:duality}, it is isomorphic to $(\ker\col_\mu^\pm)^\Delta$. Since $\HIw(\Qp,T_{f,\mu}^*)^\Delta$ is of rank two over $\Lambda$, it is enough to show that 
\[
\frac{\HIw(\Qp,T_{f,\mu}^*)^\Delta}{(\ker\col_\mu^\pm)^\Delta}\cong(\image\col_\mu^\pm)^\Delta
\]
is of rank one over $\Lambda$. As the right-hand side is a non-zero submodule of $\Lambda$  (see for example \cite[Appendix~A]{Harron} for an explicit description). In particular, it is a rank-one $\Lambda$-module.
\end{proof}

\subsubsection{Plus and minus $p$-adic $L$-functions and signed main conjectures}

We define plus and minus $p$-adic $L$-functions in terms of the Coleman maps  and signed Beilinson--Flach elements given below.
\begin{theorem}\label{thm:signedBF}
Let $m$ be an integer as given in Definition~\ref{defn:BF}. For $\lambda,\mu\in\{\alpha,\beta\}$, there exist bounded elements $\BF_{\pm,\mu,m}\in \varpi^{-s}\HIw(\QQ(\mu_{mp^\infty}),T^*)$ such that
 \[
\BF_{\lambda,\mu,m}=\log_{p}^+\BF_{+,\mu, m}+\lambda\log_p^-\BF_{-,\mu,m},
 \]
where $s$ is an integer independent of $m$.
\end{theorem}
\begin{proof}
This follows from the same proof as \cite[Theorem~3.7]{BFSuper}.
\end{proof}

\begin{defn}
Let  $\bullet,\circ\in\{+,-\}$. We define the plus and minus $p$-adic $L$-functions $L_{p}^{\bullet\circ}(f,g)\in\varpi^{-s}\Lambda(\Gamma)$ to be  $\col^\bullet_{\alpha}(\BF_{\circ,\beta,1})$.
\end{defn}
\begin{remark}
 We may switch $\alpha$ and $\beta$ in the definition. We expect that the resulting $p$-adic $L$-functions would differ by $-1$ (see \cite[hypothesis {\bf(A-Sym)}, on P.926]{BLLV}).
   \end{remark}
   
   \begin{conv}\label{conv:mp}
     Our choice of signs for the $p$-adic $L$-functions correspond to the one given in \cite{Lei2011,pollack03}. But it is opposite to the one in \cite{kobayashi03}. Unlike the plus and minus conditions, our choice of signs for the $p$-adic $L$-functions actually agrees with the one made in \cite{KK} (see in particular Remark~3.1).
   \end{conv}

\begin{conjecture}\label{conj:pmIMC}
Let $\bullet,\circ\in\{+,-\}$. The Pontryagin dual of the Selmer group $\Sel_{\bullet,\circ}(\QQ(\mu_{p^\infty}),A(1))$ is torsion over $\Lambda$. Furthermore, its characteristic ideal is generated by $\pi_\Delta L_p^{\bullet\circ}(f,g)$.
\end{conjecture}

\begin{remark}\label{conjssknownres}
It is explained in  the proof of \cite[Theorem~6.2.4]{BLLV} that under certain hypothesis, the existence of signed Euler systems in Theorem~\ref{thm:signedBF} allows us to show that $\Sel_{\bullet\circ}(K_\infty,A(1))$ is $\Lambda$-cotorsion and prove one inclusion of the main conjecture, namely,
\begin{equation}
    \pi_\Delta L_p^{\bullet\circ}(f,g)\in\Char_\Lambda\Sel_{\bullet\circ}(K_\infty,A(1))^\vee,
\label{eq:halfpmIMC}
\end{equation}
up to powers of $\varpi$. 
\end{remark}

 We conclude this section by showing that the Theta elements defined in Section~\ref{sec:Theta} are related to the signed $p$-adic $L$-functions in an explicit manner.

\begin{proposition}\label{prop:pmTheta}
For all $n\ge1$,
\begin{align*}
(-p)^{n+1}(\log_p^+)^2\pi_\Delta L_p^{++}(f,g)&\equiv \Theta_{0,n}^+\mod \omega_n,\\
(-p)^{n+2}(\log_p^-)^2\pi_\Delta L_p^{--}(f,g)&\equiv \Theta_{0,n}^-\mod \omega_n.
\end{align*}
\end{proposition}
\begin{proof}
Since $\log_p^+\col^+_{\alpha}=\frac{1}{2}\left(\cL_{\alpha,\alpha}+\cL_{\beta,\alpha}\right)$ and $\log_p^+\BF_{+,\beta,1}=\frac{1}{2}\left(\BF_{\alpha,\beta,1}+\BF_{\beta,\beta,1}\right)$ by \eqref{eq:decompose} and Theorem~\ref{thm:signedBF} respectively, we have
\begin{equation}
    (\log_p^+)^2 L_p^{++}(f,g)=\frac{1}{4}\left(L_p(f_\alpha,g)+L_p(f_\beta,g)+L_p^?(f_\alpha,g_\beta)+L_p^?(f_\beta,g_\beta)\right).
\label{eq:Lp++}
\end{equation}
Thus the result follows from the definition of $\Theta_{0,n}^+$ and the fact that $\alpha_f^2=-p$.  The proof for $L_p^{--}(f,g)$ is similar.
\end{proof}

\subsection{Non-existence of proper $\Lambda$-submodules of finite index}\label{subsec:nonexistenceoflambdasubmoduleoffiniteindex}

The aim of this section is to show that the Selmer groups over $K_\infty$ introduced above have no non-trivial $\Lambda$-submodules of finite index. We show this by using results of \cite{greenbergselmergroups}, namely \cite[Proposition 4.1.1]{greenbergselmergroups}. We recall it together with the relevant notation in \S \ref{greenbergres} below. In \S \ref{subsubsec:localcohomologygroups}, we discuss a result on the structure of the local Galois cohomology groups defining the Selmer groups over $K_\infty$, which is employed later in \S \ref{sec:greenbergstructure}.

\subsubsection{Local cohomology groups} \label{subsubsec:localcohomologygroups} Let us recall that we have denoted $A_j:=A(1+j)$, where $k_g +1 \leq j \leq k_f$. Given a Selmer structure $\mathcal{L} \in \{ \textup{Gr}, \bullet\circ \}$, we define, for any prime $\ell \neq p$ of $\QQ$
\[ \mathcal{H}_{\ell,j} := \varinjlim\prod\limits_{w \mid \ell}\tupH^{1}_{/f}(F_{w}, A_j) \]
and
\[ \mathcal{H}_{p,j}^\mathcal{L} := \varinjlim \tupH^{1}_{/\mathcal{L}_\mathfrak{p}}(F_{\mathfrak{p}}, A_j), \]
 where $F$ runs over all finite extension of $\QQ$ contained in $K_{\infty}$ and $\mathfrak{p}$ denotes the only one prime of $F$ above $p$. We now recall the following result that concerns the structure of these Galois cohomology groups. 

\begin{lemma} \label{lemma:localdimensions}
Let $\ell$ be a prime of $\QQ$. Then, we have
\begin{enumerate}
\item If $\ell \neq p$, then
\[\mathcal{H}_{\ell,j} \cong \prod\limits_{w \mid \ell} \tupH^{1}(K_{\infty,w},A_j)\]
as $\Lambda$-modules. Furthermore $\mathcal{H}_{\ell,j}$ is a co-finitely generated, cotorsion $\Lambda$-module.
\item For $\ell = p$ and $\mathcal{L}=\textup{Gr}$,
\[\mathcal{H}_{p,j}^\textup{Gr} \cong \frac{\tupH^{1}(K_{\infty,\mathfrak{p}},A_j)}{\tupH^{1}(K_{\infty,\mathfrak{p}},\mathcal{F}^{+}A_j)}\]
and $\mathcal{H}_{p,j}^\textup{Gr}$ is a co-finitely generated $\Lambda$-module of corank $2$. Here $\mathfrak{p}$ is the unique prime above $p$ in $K_{\infty}$.
\item For $\ell = p$, $\mathcal{L}=\bullet\circ$ where $\bullet,\circ\in\{+,-\}$, and $k_f=k_g+1=0$,
\[\mathcal{H}_{p,0}^{\bullet\circ} \cong \frac{\tupH^{1}(K_{\infty,\mathfrak{p}},A(1))}{\tupH^{1}_{\bullet\circ}(K_{\infty,\mathfrak{p}},A(1))}\]
and $\mathcal{H}_{p,0}^{\bullet\circ}$ is a co-finitely generated $\Lambda$-module of corank $2$.
\end{enumerate}
\end{lemma}
\begin{proof}
For part (1), note that, as $\ell$ is finitely decomposed in $K_\infty$, there exists a sufficiently large finite extension $\QQ \subset F \subset K_{\infty}$ such that the size of the set $s_{F,\ell}$ of places of $F$ above $\ell$ is constant and equal to the size of $s_{K_{\infty},\ell}$. Thus, 
\[ \mathcal{H}_{\ell,j} \cong \prod_{v \in s_{F,\ell}}  \varinjlim\limits\tupH^{1}_{/f}(K_{w}, A_j) =\prod \limits_{w \in  s_{K_{\infty},\ell}} \tupH^{1}_{/f}(K_{\infty,w},A_j),\]
where the limit runs through $F \subset K \subset K_{\infty}$ and $w$ denotes the unique prime of $K$ above $v$. 
Note that $\tupH^{1}_{f}(K_{\infty,w},A_j)=0$ for any $w \in s_{K_{\infty},\ell}$. Indeed, as
$\ell$ does not split completely in $K_\infty$, the quotient $G_{K_{\infty,w}}/I_w$ has pro-order prime to $p$ and thus the restriction map \[\tupH^{1}(K_{\infty,w},A_j) \to \tupH^{1}(I_w,A_j) \] 
is injective. Hence we obtain the desired isomorphism and $\mathcal{H}_{\ell,j}$ is a co-finitely generated, cotorsion $\Lambda$-module because each $\tupH^{1}(K_{\infty,w},A_j)$ is such one (cf. \cite[Proposition 2]{greenberg1989iwasawa}).

Part (2) follows similarly from \cite[Proposition 1, Corollary 1]{greenberg1989iwasawa}. Indeed, note that, by \cite[Proposition 1]{greenberg1989iwasawa} we have that the $\Lambda$-corank of $\tupH^{1}(K_{\infty,\mathfrak{p}},A_j)$ (resp. $\tupH^{1}(K_{\infty,\mathfrak{p}},\mathcal{F}^+ A_j)$) is equal to the $(L/\mathcal{O})$-dimension of the $G_{\QQ_{p}}$-module $A_j$ (resp. $\mathcal{F}^+ A_j$).

Analogously, part (3) is given by Lemma~\ref{lem:corankpm}.
\end{proof}

\subsubsection{A result of Greenberg}\label{greenbergres}
 The proof of the main result of this section (cf. Proposition \ref{nofinitesubmod}) relies on the application of \cite[Proposition 4.1.1]{greenbergselmergroups} to our setting. For the convenience of the reader, we set here the necessary notation from \cite{greenbergselmergroups} and state \cite[Proposition 4.1.1]{greenbergselmergroups}.
 
Let $\TT_j := T(1+j) \otimes \Lambda(\Psi^{-1})$, where $\Psi$ is the canonical character $G_\QQ \twoheadrightarrow \Gamma_1 \hookrightarrow \Lambda^\times$, and let $\aA_j: = \TT_j \otimes_\Lambda {\rm Hom}_{\Zp}( \Lambda, \Qp /\Zp)$, where $G_\QQ$ acts diagonally on $\TT_j$.  Recall that by $\mathcal{L}$ we denote one of the Selmer structures in $\{ \textup{Gr}, \bullet\circ \}$.  In a similar way to what was done in \S \ref{subsec:doublyordinaryselmer} and \S \ref{sec:setupSS}, we may define $\textup{Sel}_{\mathcal{L}}(\QQ, \aA_j)$, which, by Shapiro's lemma, are isomorphic to $\textup{Sel}_{\mathcal{L}}(K_\infty, A_j)$. In other words, if  $Q_{\mathcal{L}}(\QQ, \aA_j)$ denotes the target of the map defining the Selmer group $\textup{Sel}_{\mathcal{L}}(\QQ, \aA_j)$, then
\[  Q_{\mathcal{L}}(\QQ, \aA_j) \simeq \prod\limits_{\ell \neq p} \mathcal{H}_{\ell,j} \times \mathcal{H}_{p,j}^\mathcal{L} \]
and 
\[ \textup{Sel}_{\mathcal{L}}(K_\infty, A_j) \simeq \textup{ker}\Big( \tupH^{1}(\mathbb{Q}_{\Sigma}/\QQ, \aA_j) \rightarrow Q_{\mathcal{L}}(\QQ, \aA_j)\Big). \]

Finally, denote $\widetilde{\TT}_j := \textup{Hom}(\aA_j,\mu_{p^{\infty}})$ and let $\mathfrak{m}$ be the maximal ideal of the Iwasawa algebra $\Lambda$. 
Greenberg introduces the following list of hypotheses.
\begin{itemize}
\item ${\rm RFX}(\aA_j)$: $\TT_j$ is a reflexive $\Lambda$-module;
\item ${\rm LOC}_v^{(1)}(\aA_j)$: $(\widetilde{\TT}_j)^{\mathcal{G}_{\QQ_v}}=0$ for $v \in \Sigma$;
\item ${\rm LOC}_v^{(2)}(\aA_j)$: $(\widetilde{\TT}_j) / (\widetilde{\TT}_j)^{\mathcal{G}_{\QQ_v}}$ is reflexive for $v \in \Sigma$;
\item ${\rm LEO}(\aA_j)$: the $\Lambda$-module \[ \Sha^{2}(\QQ,\Sigma, \aA_j) = {\rm ker} \left ( \tupH^2(\QQ_\Sigma/\QQ, \aA_j) \longrightarrow \prod_{v \in \Sigma} \tupH^2(\QQ_v, \aA_j)  \right )  \] is cotorsion.
\item ${\rm CRK}(\aA_j,\mathcal{L})$: We have an equality of coranks:\\
$\textup{corank}_{\Lambda}(\tupH^{1}(\QQ_{\Sigma}/\QQ, \aA_j)) = \textup{corank}_{\Lambda}(\textup{Sel}_{\mathcal{L}}(K_\infty,A_j)) + \textup{corank}_{\Lambda}(Q_{\mathcal{L}}(\QQ, \aA_j))$.
\end{itemize}

Recall the following (cf. \cite[\S 2]{greenbergstructuredoc}).

\begin{defn} \label{defn:almostdivisible}
A discrete $\Lambda$-module $M$ is said to be almost divisible if for almost all height one prime ideals $P \in \textup{Spec}(\Lambda)$, $PM = M$. 
\end{defn}
\begin{remark}
This definition is equivalent to asking that the discrete $\Lambda$-module has no proper $\Lambda$-sub-modules of finite index. We refer the reader to \cite[\S 2]{greenbergstructuredoc} for further discussion on almost divisible modules and to Proposition 2.4 of  \emph{op. cit.} for a proof of this equivalence.  
\end{remark}

\begin{remark}
We say that a Selmer structure $\mathcal{L}=\{\mathcal{L}_{\nu}\}$ is almost divisible if the corresponding local cohomology groups $\tupH^{1}_{/\mathcal{L}_{\nu}}(\QQ_{\nu},\aA_j)$ is an almost divisible $\Lambda$-module for all places $\nu$. 
\end{remark}

We can now state \cite[Proposition 4.1.1]{greenbergselmergroups}.

\begin{proposition}\label{greenbergprop}
Suppose that ${\rm RFX}(\aA_j)$, ${\rm LEO}(\aA_j)$ are both satisfied, that ${\rm LOC}_v^{(2)}(\aA_j)$ is satisfied for all $v \in \Sigma$, and that there exists a non-archimedean $v \in \Sigma$ such that ${\rm LOC}_v^{(1)}(\aA_j)$ is satisfied. If $\mathcal{L}$ is almost divisible, ${\rm CRK}(\aA_j,\mathcal{L})$ holds, and $\aA_j[\mathfrak{m}]$ has no Galois sub-quotients isomorphic to $\mu_p$, then 
$\textup{Sel}_{\mathcal{L}}(K_\infty, A_j)$ is an almost divisible $\Lambda$-module.
\end{proposition}

\subsubsection{Almost divisibility} \label{sec:greenbergstructure}

We now discuss the validity of the hypotheses of Proposition \ref{greenbergprop} in our setting, which we then apply to deduce that the Selmer groups $\textup{Sel}_{\mathcal{L}}(K_\infty, A_j)$, with $\mathcal{L} \in \{ \textup{Gr}, \bullet\circ \}$, have no non-trivial $\Lambda$-modules of finite index. We will work under the following assumption.

\begin{itemize}
\item[]\textbf{(Co-tor)} $\textup{Sel}_{\mathcal{L}}(K_\infty, A_j)$ is $\Lambda$-cotorsion for $\mathcal{L} \in \{ \textup{Gr}, \bullet\circ \}$ and $k_g+1 \leq j \leq k_f$.
\end{itemize}  

The validity of \textbf{(Co-tor)} for $\mathcal{L} \in \{ \textup{Gr}, \bullet\circ \}$ has been discussed in Remarks \ref{cnjordoneineq} and \ref{conjssknownres}.
\begin{lemma} \label{lem:corank}
Assume \textbf{(Co-tor)}. \begin{enumerate}
    \item For any $k_g+1 \leq j \leq k_f$, ${\rm CRK}(\aA_j,\textup{Gr})$ and ${\rm LEO}(\aA_j)$ are satisfied.
    \item If $k_f=j=k_g+1=0$, ${\rm CRK}(\aA_0, \bullet\circ )$ and ${\rm LEO}(\aA_0)$ hold.
\end{enumerate}
\end{lemma}
\begin{proof}
The proof is based on an application of the Euler-Poincar\'e characteristic formula. Before going into details, in the case of $\mathcal{L}= \bullet\circ$, we set $k_g+1=j=k_f=0$. 

Note that by Hypothesis \textbf{(Co-tor)}, we have $\textup{corank}_{\Lambda}(\textup{Sel}_{\mathcal{L}}(K_\infty,A_j)) = 0$ for $\mathcal{L} \in \{ \textup{Gr}, \bullet\circ \}$. By Lemma~\ref{lemma:localdimensions}, we know that $\textup{corank}_{\Lambda}(Q_\mathcal{L}(\QQ,\aA_j)) = 2$. Hence to verify that ${\rm CRK}(\aA_j,\mathcal{L})$ holds, it suffices to show that \[ \textup{corank}_{\Lambda}(\tupH^{1}(\QQ_{\Sigma}/\QQ, \aA_j)) = 2. \]
Notice that from the exact sequence 
\[ 0 \to \textup{Sel}_{\mathcal{L}}(K_\infty,A_j) \to \tupH^{1}(\QQ_{\Sigma}/\QQ, \aA_j) \xrightarrow{\phi_\mathcal{L}}  Q_\mathcal{L}(\QQ,\aA_j), \]
we obtain the equality \[ 0 = \textup{corank}_{\Lambda}(\tupH^{1}(\QQ_{\Sigma}/\QQ, \aA_j)) - 2 + \textup{corank}_{\Lambda}(\textup{coker}(\phi_\mathcal{L})).\] 

However, from the Euler-Poincar\'e characteristic formula (c.f. \cite[Proposition 3]{greenberg1989iwasawa}) we have that \[ \textup{corank}_{\Lambda}(\tupH^{1}(\QQ_{\Sigma}/\QQ, \aA_j)) - \textup{corank}_{\Lambda}(\tupH^{2}(\QQ_{\Sigma}/\QQ, \aA_j)) = 2.\] The two equalities  give our claim. The same argument shows ${\rm LEO}(\aA_j)$.
\end{proof}

\begin{proposition}\label{nofinitesubmod}
Assume \textbf{(Co-tor)} and $j$ is an integer such that $k_g+1 \leq j \leq k_f$ and that the residual representation $\overline{T}(1+j)$ has no Galois sub-quotient isomorphic to $\mu_p$. Then, the Selmer groups $\textup{Sel}_{\textup{Gr}}(K_\infty, A_j)$ and $\textup{Sel}_{\bullet\circ}(K_\infty, A(1))$ (with $k_f=j=k_g+1=0$) have no proper $\Lambda$-sub-module of finite index.
\end{proposition}

\begin{proof}
As it was mentioned above, the result follows from applying Proposition \ref{greenbergprop}. Thus, we have to verify that all its hypotheses hold. Some of them have been shown in Lemma \ref{lem:corank}. We now discuss the others, namely ${\rm RFX}(\aA_j)$, ${\rm LOC}_v^{(1)}(\aA_j)$, ${\rm LOC}_v^{(2)}(\aA_j)$, and the almost divisibility of $\mathcal{L} \in \{ \textup{Gr}, \bullet\circ \}$. 

Notice that ${\rm RFX}(\aA_j)$ holds as $\TT_j$ is free of rank 4. Moreover, by \cite[Lemma 5.2.2]{greenberg2010surjectivity}, we have that  $(\widetilde{\TT}_j)^{\mathcal{G}_{\QQ_v}}=0$ for every $v \in \Sigma$, thus ${\rm LOC}_v^{(1)}(\aA_j)$ holds and implies ${\rm LOC}_v^{(2)}(\aA_j)$ for every $v \in \Sigma$. We are left with proving that the Selmer structures $\textup{Gr}$ and $\bullet\circ$ are almost divisible. The statement follows from showing that the modules $\cH_{\ell,j}, \cH^{\textup{Gr}}_{p,j}$ and $\cH^{\bullet\circ}_{p,0},$ introduced in \S \ref{subsubsec:localcohomologygroups}, are almost divisible $\Lambda$-modules. From \cite[Proposition 5.4]{greenbergstructuredoc} (which we can apply as ${\rm LOC}_v^{(2)}(\aA_j)$ is valid), we have that $H^1(K_{\infty,w},A_j)$ is almost divisible for any  place $w$, which implies the almost divisibility of $\cH_{\ell,j}$ because of Lemma \ref{lemma:localdimensions}(1). Finally, from Lemma \ref{lemma:localdimensions}(2)-(3), the Pontryagin duals of $\cH^{\textup{Gr}}_{p,j}$ and $\cH^{\bullet\circ}_{p,0}$ are submodules of $\Lambda^{\oplus 2}$. Thus, $\cH^{\textup{Gr}}_{p,j}$ and $\cH^{\bullet\circ}_{p,0}$ are almost divisible. 
\end{proof}

\begin{remark}
We make a couple of comments on the hypothesis that $\overline{T}(1+j)$  has no Galois sub-quotient isomorphic to $\mu_p$.
\begin{itemize}
    \item[(i)] If $\bar{T}(1+j)$ is irreducible, then the module $\aA_j[\mathfrak{m}]$ has no sub-quotient isomorphic to $\mu_{p}$ as a $G_{\QQ}$-module (cf. \cite[\S 4.3.3]{greenbergselmergroups}). 
    \item[(ii)] Suppose that $f$ corresponds to an elliptic curve $E$ with supersingular reduction at $p$. Then, $E[p]|_I$ is the direct sum of two fundamental characters of level 2 (see \cite[Theorem~2.6]{ed}), where $I$ is the inertial group of $G_{\Qp}$. In particular, there is no Galois sub-quotient isomorphic to $\mu_p$. Since $g$ is of weight one and $T_g$ decomposes into the direct sum of two unramified characters, the same is true for $\bar{T}(1)$.
\end{itemize}
\end{remark}

Recall that, given a ring $R$ and a finitely presented $R$-module $M$, the Fitting ideal $\textup{Fitt}_{R}(M)$ is defined as follows. Take a presentation of $M$ \[ \xymatrix{R^r \ar[r]^g & R^t \ar[r] & M \ar[r] & 0;} \]
$ \textup{Fitt}_{R}(M)$ is defined to be the ideal generated by the determinants of the $t\times t$-minors of $g$.
Proposition \ref{nofinitesubmod} has the following  consequence.

\begin{corollary}\label{cor:Fittchar}
Keep the hypotheses of Proposition \ref{nofinitesubmod}. For any $k_g+1 \leq j \leq k_f$, we have \[ {\rm char}_\Lambda (\textup{Sel}_{\textup{Gr}}(K_\infty, A_j)^\vee) = {\rm Fitt}_\Lambda (\textup{Sel}_{\textup{Gr}}(K_\infty, A_j)^\vee); \]
Similarly, if $k_g+1 =j = k_f =0$, we have \[ {\rm char}_\Lambda (\textup{Sel}_{\bullet\circ}(K_\infty, A(1))^\vee) = {\rm Fitt}_\Lambda (\textup{Sel}_{\bullet\circ}(K_\infty, A(1))^\vee). \]
\end{corollary}
\begin{proof}
It directly follows from \cite[Lemma A.7]{KK}.
\end{proof}

\begin{remark}
If $f$ and $g$ are both non-ordinary at $p$, we have defined in \cite[Definition~6.1.2]{BLLV} six Selmer groups over $\QQ(\mu_{p^\infty})$ using pairs Coleman maps on $\HIw(\Qp,T_{f,g}^*)$. The images of these Coleman maps are described in the appendix of \emph{op. cit.} In particular, we may verify the hypothesis CRK as in the proof of Lemma~\ref{lem:corank} and the almost divisiblity of the local conditions at $p$, allowing us to obtain a generalization of Proposition~\ref{nofinitesubmod} and Corollary~\ref{cor:Fittchar}.
\end{remark}

\section{Control theorems}

We prove control theorems for the Selmer groups studied in the previous section.  This is one of the key ingredients in our quest to generalize results in \cite{KK}. Once again, we consider the cases where $f$ is $p$-ordinary and $p$-non-ordinary separately.

\subsection{The ordinary case}\label{sec:ordinarycontroltheorem}
We study a control theorem for the Greenberg Selmer groups of Rankin--Selberg convolutions defined in Definition~\ref{defn:selmergroupsordinary} using results of \cite{ochiaiBK} and \cite{ochiai}. We assume that both $f$ and $g$ are $p$-ordinary; we let $j$ be an integer satisfying $k_g+1 \leq j \leq k_f$ as before.
To ease the notation, recall that we denote by $A_j$ the module $A(1 + j)$.
Finally, recall that, for $n \ge 0$, we have denoted by $\omega_{n}(X)$ the polynomial $(1+X)^{p^{n}}-1$.

\begin{defn}
For any $0 \leq n\leq \infty$ denote by $\cX_{j,n}^\textup{Gr}$ to be the Pontryagin dual
\[
\Sel_\textup{Gr}(K_n, A_j)^\vee.
\]
Moreover, denote by
    $ r_n: \cX_{j,\infty}^\textup{Gr}/\omega_n \rightarrow \cX_{j,n}^\textup{Gr}$  the maps induced by the restriction maps
\[ \Sel_\textup{Gr}(K_n,A_j)\rightarrow\Sel_\textup{Gr}(K_\infty,A_j)^{\Gamma_1^{p^n}}. \]
\end{defn}

We can now state the main theorem of the section.

\begin{theorem}\label{controlordinary} Let us suppose that, for all $n$, $\tupH^1(\Gamma_1^{p^n}, A_j^{\textup{Gal}(\QQ_\Sigma/K_\infty)})=0$.
Then, for all $n$, the maps $r_n: \cX_{j,\infty}^\textup{Gr}/\omega_n\rightarrow \cX_{j,n}^\textup{Gr}
$ are surjective with finite kernel of bounded orders as $n$ varies.
\end{theorem}

\begin{proof}

Consider the following commutative diagram:
\[ \xymatrix{ 0 \ar[r] & \Sel_\textup{Gr}(K_n,A_j) \ar[r] \ar[d]^{\alpha_n} & \tupH^{1}(\QQ_{\Sigma}/K_n, A_j) \ar[r] \ar[d]^{\beta_n} & \prod\limits_{\mathfrak{p} \neq v \in \Sigma_n} \tupH^{1}_{/f}(K_{n,v}, A_j)  \times \tupH^{1}_{/\textup{Gr}}(K_{n,\mathfrak{p}}, A_j) \ar[d]^{\gamma_n=\prod \gamma_{n,v}} \\
0 \ar[r] & \Sel_\textup{Gr}(K_\infty,A_j)^{\Gamma_1^{p^n}} \ar[r] & \tupH^{1}(\QQ_{\Sigma}/K_\infty, A_j)^{\Gamma_1^{p^n}} \ar[r] & \prod\limits_{\mathfrak{p} \neq w  \in \Sigma_\infty} \tupH^{1}_{/f}(K_{\infty,v},A_j)^{\Gamma_{1,v}^{p^n}} \times (\mathcal{H}_{p,j}^{\textup{Gr}})^{\Gamma_{1,\mathfrak{p}}^{p^n}} ,}\]
where $\Sigma_n$ denotes the primes of $K_n$ above primes in $\Sigma$. In order to bound kernel and cokernel of $\alpha_n$, we reduce to study the maps $\beta_n$ and $\gamma_n$.

From the diagram above, we have that  \[ \ker(\alpha_n)  \hookrightarrow \ker(\beta_n) = \tupH^1(\Gamma_1^{p^n}, A_j^{\textup{Gal}(\QQ_\Sigma/K_\infty)}).
\]
By assumption, the latter is trivial, thus $\ker(\alpha_n)=0$.

 We now turn our attention to the study of $\textup{Coker}(\alpha_n)$. We start by noticing that, by the inflation-restriction exact sequence, we have \[ \textup{Coker}(\beta_n) \hookrightarrow H^2(\Gamma_1^{p^n},A_j^{\textup{Gal}(\QQ_\Sigma/K_\infty)}). \] However, the latter group is $0$, as $\Gamma_1^{p^n}$ is a free pro-p group of cohomological dimension 1. Thus, by the Snake Lemma,  $\textup{Coker}(\alpha_n)$ is a sub-quotient of $\ker(\gamma_n)$.

 Let $v \in \Sigma_n$ be a prime which does not lie above $p$; then, by \cite[Lemma 2.8]{ochiaiBK}, the order of $\ker(\gamma_{n,v})$ is bounded independently of $n$. Finally, this leaves us with studying the kernel of \[\gamma_{n,\mathfrak{p}}: \tupH^{1}_{/\textup{Gr}}(K_{n,\mathfrak{p}}, A_j) \to (\mathcal{H}_{p,j}^{\textup{Gr}})^{\Gamma^{p^n}_{1,\mathfrak{p}}}. \] 
In order to prove the assertion that the order of $\ker(\gamma_{n,\mathfrak{p}})$ is bounded independently by $n$, we need to check the hypotheses of \cite[Theorem 3.5(2)]{ochiai} for $V(1+j)$. These in turn follow from the $p$-ordinarity of $f$ and $g$ and the fact that the eigenvalues of $\textup{Frob}_p$ on $V$ are not roots of unity, as they all have absolute value $p^{(k_f+k_g)/2+1}$ and $k_f > k_g \geq -1$.
\end{proof}

We now discuss cases where the hypothesis of Theorem \ref{controlordinary} holds.

\begin{proposition}\label{invariantsforalln}
Suppose that $f$ and $g$ have weights $\geq 2$ and that $f$ is non-CM. Then, for all but finitely many $p$ which do not divide $N_fN_g$ and for all $n$, we have
\[ \tupH^1(\Gamma_1^{p^n}, A_j^{\textup{Gal}(\QQ_\Sigma/K_\infty)})=0.\] 
\end{proposition}

\begin{proof}
The strategy for proving the statement uses ideas employed in \cite[Proposition~7.2.18]{LLZ14} and \cite[Proposition 4.2.1]{loeffler2017}; we now sketch it and refer to \cite{loeffler2017} for further details.
 Let $H$ denote the subgroup of $G_\QQ$ cut out by the Dirichlet characters corresponding to inner twists of $f$ and $g$; we denote by $\rho_{h,p}$ the Galois representation of $h \in \{f,g\}$ taking values in $\mathbf{GL}_2(\mathcal{O}_h)$, for $\mathcal{O}_h$ the ring of integral elements of a finite extension  $\Qp \subset F_h \subset L$.
 
 Suppose that $g$ is non-CM; then, for all but finitely many $p$, the image of the representation $(\rho_{h,p})_{|_H}$, where $h \in \{f,g\}$, equals to \[\{ x \in \mathbf{GL}_2(\mathcal{O}_h) \; : \; {\rm det}(x) \in \Zp^\times \}. \]

By our running hypotheses on their weights, $f$ and $g$ are not  twists of each other. Thus, by \cite[Proposition 4.2.1 and Remark 4.2.2]{loeffler2017}, for all but finitely many primes $p$ coprime with $N_fN_g$, the image of  $(\rho_{f,p}\times \rho_{g,p})_{|_H}$ equals \[\{ (x,y) \in \mathbf{GL}_2(\mathcal{O}_f) \times \mathbf{GL}_2(\mathcal{O}_g) \; : \; {\rm det}(x) = {\rm det}(y) \in \Zp^\times \}. \]
Let us now fix such a prime $p$.  The image  $(\rho_{f,p}\times \rho_{g,p})(H \cap G_{\QQ(\mu_{p^\infty})})$ contains the element 
\[ M:= \left ( \left( \begin{smallmatrix} x & \\ & x^{-1} \end{smallmatrix} \right) , \left ( \begin{smallmatrix} y & \\ & y^{-1}  \end{smallmatrix}\right) \right), \]
where $x,y \in \Zp^\times$. The image of $M$ under the tensor product homomorphism $\mathbf{GL}_2 \times \mathbf{GL}_2  \to \mathbf{GL}_4 $ is the diagonal matrix with eigenvalues $\{xy, x^{-1}y, xy^{-1},x^{-1}y^{-1} \}$. We can easily choose $x,y$ so that none of the eigenvalues equals to $1$, thus constructing an element in $G_{\QQ(\mu_{p^\infty})}$ which acts without fixed points on $V(1+j)$. Thus,  $\tupH^0(K_n,V(1+j))=0$ for all $n$.
 
As explained in the proof of \cite[Theorem 2.4(i)]{ochiaiBK}, this shows that $\tupH^1(\Gamma_1^{p^n}, A_j^{\textup{Gal}(\QQ_\Sigma/K_\infty)})$ is finite. We sketch the argument here and refer to \cite{ochiaiBK} for further details.
The map given by evaluation of cocycles at the topological generator $1+\omega_n$ of $\Gamma_1^{p^n}$ gives an isomorphism \[ \tupH^1(\Gamma_1^{p^n}, A_j^{\textup{Gal}(\QQ_\Sigma/K_\infty)}) \stackrel{\sim}{\longrightarrow} (A_j^{\textup{Gal}(\QQ_\Sigma/K_\infty)})_{\Gamma_1^{p^n}}:= A_j^{\textup{Gal}(\QQ_\Sigma/K_\infty)}/  \omega_n A_j^{\textup{Gal}(\QQ_\Sigma/K_\infty)}.
\]
As $A_j^{\textup{Gal}(\QQ_\Sigma/K_\infty)}$ is of finite type, it sits in the exact sequence \[ 0 \to D \to A_j^{\textup{Gal}(\QQ_\Sigma/K_\infty)} \to E \to 0, \]
where $D$ is the maximal divisible subgroup of $A_j^{\textup{Gal}(\QQ_\Sigma/K_\infty)}$ and $E$ is finite.
As $E$ is finite, we are left to show that $D_{\Gamma_1^{p^n}}$ is bounded. Notice that $D_{\Gamma_1^{p^n}}$ is either zero or infinite. If it is infinite, then ${\rm corank}_{\Zp}D_{\Gamma_1^{p^n}} ={\rm corank}_{\Zp}D^{\Gamma_1^{p^n}} \ne 0$, which implies that $A_j^{\Gamma_1^{p^n}}$ is infinite. However, this is impossible as we have shown that  $\tupH^0(K_n,V(1+j))=0$ for all $n$. 

As $\tupH^1(\Gamma_1^{p^n}, A_j^{\textup{Gal}(\QQ_\Sigma/K_\infty)})$ is finite, we are left to show that \[\tupH^1(\Gamma_1^{p^n}, A_j^{\textup{Gal}(\QQ_\Sigma/K_\infty)})[\varpi]=0. \]
 To do so, notice that we have a surjection \[\tupH^1(\Gamma_1^{p^n}, (A_j[\varpi])^{\textup{Gal}(\QQ_\Sigma/K_\infty)}) \twoheadrightarrow \tupH^1(\Gamma_1^{p^n}, A_j^{\textup{Gal}(\QQ_\Sigma/K_\infty)})[\varpi], \]
and that $A_j[\varpi] \simeq \overline{T}(1+j)$. 
By \cite[Proposition 4.2.1]{loeffler2017}, $\overline{T}$ is irreducible as a $G_{K_\infty}$-module, hence \[(\overline{T}(1+j))^{G_{K_\infty}}=0,\]
which implies that \[ \tupH^1(\Gamma_1^{p^n}, (A_j[\varpi])^{\textup{Gal}(\QQ_\Sigma/K_\infty)}) =0.\]

When $g$ is CM, for all but finitely many primes $p$ coprime to $N_fN_g$, the image of $H \cap G_{\QQ(\mu_{p^{\infty}})}$ of $\rho_{f,p} \times \rho_{g,p}$ contains the element $M$ considered above (cf. \cite[Proposition 4.3.1]{loeffler2017}). Thus,  we can construct an element $\tau \in G_{\QQ(\mu_{p^{\infty}})}$ which acts without fixed points on $\overline{T}(1+j)$ and on $V(1+j)$. This implies the result. 
\end{proof}

Of particular interest for us is the case when the weight of $g$ is one, which we now discuss. 

\begin{proposition}\label{invariantsforallncaseweight1}
Suppose that $g$ has weight $1$, $f$ has weight $\geq 2$ and  is non-CM, and that $N_f$ and $N_g$ are coprime. Then, for all but finitely many $p$ which do not divide $N_fN_g$ and for all $n$, we have
\[ \tupH^1(\Gamma_1^{p^n}, A_j^{\textup{Gal}(\QQ_\Sigma/K_\infty)})=0.\] 
\end{proposition}

\begin{proof}
The proof of the statement is similar to the one of Proposition \ref{invariantsforalln} and it heavily relies on \cite[Theorem 4.4.1]{loeffler2017}.

As $g$ has weight one, $\rho_{g,p}$ is an Artin representation. Since $N_g$ is coprime to $pN_f$, the splitting field $F_g$ of $\rho_{g,p}$, the one of $\rho_{f,p}$, and $\QQ(\mu_{p^\infty})$ are linearly disjoint over $\QQ$. Thus, for every $\gamma \in \rho_{g,p}(G_\QQ)$ and $\delta \in \rho_{f,p}(G_{\QQ(\mu_{p^\infty})})$, there exists $\tau \in G_{\QQ(\mu_{p^\infty})}$  such that $\rho_{g,p}(\tau)=\gamma$ and $\rho_{f,p}(\tau)=\delta$.

For all but finitely many $p$,  $\rho_{f,p}(G_{\QQ(\mu_{p^\infty})})$ contains the elements $\delta_x = \left( \begin{smallmatrix} x & \\ & x^{-1} \end{smallmatrix} \right)$, for $x \in \Zp^\times$. Moreover,  since $\rho_{g,p}$ is odd, there is $\gamma \in \rho_{g,p}(G_\QQ)$ conjugate to $\left( \begin{smallmatrix} -1 & \\ & 1 \end{smallmatrix} \right)$. One can choose $x$ so that the image of $\delta_x \otimes \gamma$ in $\mathbf{GL}_4$ has all eigenvalues different from 1; thus there exists $\tau \in G_{\QQ(\mu_{p^\infty})}$ which acts without fixed points on $V(1+j).$ Similarly, this shows that $ (\overline{T}(1+j))^{G_{K_\infty}} = 0$.
As in Proposition \ref{invariantsforalln}, these facts imply that 
\[ \tupH^1(\Gamma_1^{p^n}, A_j^{\textup{Gal}(\QQ_\Sigma/K_\infty)})=0.\] 
\end{proof}

\subsection{The supersingular case}\label{sec:controlpm}
We now turn our attention to the case where $f$ corresponds to an elliptic curve $E/\QQ$  with $a_p(E)=0$ and $g$ is a weight-one form, as in \S\ref{sec:setupSS}. We do not know how to prove control theorems when $f$ is a more general $p$-non-ordinary form, which seems to be  the biggest obstacle of generalizing results of \cite{KK} beyond our current setting. Let $K_g/\Qp$ denote the smallest extension such that the action of $G_{K_g}$ on $V_g$ is trivial.  Throughout, we assume that the following hypothesis holds, which allows us to apply results from \cite{Kim}.
\begin{itemize}
    \item[\textbf{(H.Kim)}] $4\nmid [K_g:\Qp]$.
\end{itemize}

\begin{defn}
For an integer $n\ge0$, we define $\omega_n^\pm$ to be the following polynomials:
\[
    \omega_n^\pm(X)=X\prod_{m\in S_{n+1}^\pm\setminus\{0\}}\Phi_m(X).
\]
Given a $\Lambda$-module $M$, we write 
\[
M/\omega_n^\pm=M/\omega_n^\pm(\gamma-1)M.
\]

For $0\le n\le \infty$ and $\bullet\in\{+,-\}$, we define $\cX_n^\bullet$ to be the Pontryagin dual
\[
\Sel_{\bullet\bullet}(K_n,A(1))^\vee.
\]
\end{defn}

\begin{theorem}\label{thm:controlpm}
Let $\bullet\ne\circ$ be a choice of the two elements in $\{+,-\}$. The natural maps
\[
\cX_\infty^\bullet/\omega_n^\circ\rightarrow \cX_n^\bullet/\omega_n^\circ
\]
are surjective and have finite kernel of bounded orders as $n$ varies. 
\end{theorem}
\begin{proof}
Given a $\Lambda$-module $M$, we define
$$M^{\omega_n^\circ}=\{x\in M:\omega_n^\circ (\gamma-1)x=0 \}.$$ 
Note that $M^{\omega_n^\circ}$ is a submodule of $M^{\Gamma_1^{p^n}}$.

Consider the commutative diagram
\[ \xymatrix{ 0 \ar[r] & \Sel_{\bullet\bullet}(K_n,A(1))^{\omega_n^\circ} \ar[r] \ar[d]^{\alpha_n} & \tupH^{1}(\QQ_{\Sigma}/K_n, A(1))^{\omega_n^\circ} \ar[r] \ar[d]^{\beta_n} & \prod\limits_{\mathfrak{p} \neq v \in \Sigma_n} \frac{\tupH^1(K_{n,v},A(1))^{\omega_n^\circ} }{\tupH^1_{f}(K_{n,v},A(1))^{\omega_n^\circ} }\times \frac{\tupH^1(K_{n,\p},A(1))^{\omega_n^\circ} }{\tupH^1_{\bullet\bullet}(K_{n,\p},A(1))^{\omega_n^\circ} }\ar[d]^{\gamma_n=\prod \gamma_{n,v}} \\
0 \ar[r] & \Sel_{\bullet\bullet}(K_\infty,A(1))^{\omega_n^\circ} \ar[r] & \tupH^{1}(\QQ_{\Sigma}/K_\infty, A(1))^{\omega_n^\circ} \ar[r] & \prod\limits_{\mathfrak{p} \neq v \in \Sigma_n} \frac{\tupH^1(K_{\infty,v},A(1))^{\omega_n^\circ} }{\tupH^1_{f}(K_{\infty,v},A(1))^{\omega_n^\circ} }\times \frac{H^1(K_{\infty,\p},A(1))^{\omega_n^\circ} }{H^1_{\bullet\bullet}(K_{\infty,\p},A(1))^{\omega_n^\circ} }.}\]

It follows from the proof of Lemma~\ref{lem:restriction} that $A(1)^{G_{K_\infty}}=0$. Thus, the inflation-restriction exact sequence gives an isomorphism
\[
\tupH^1(\QQ_\Sigma/K_n,A(1))\stackrel{\sim}{\longrightarrow}\tupH^1(\QQ_\Sigma/K_\infty,A(1))^{\Gamma_1^{p^n}}.
\]
Since $M^{\Gamma_1^{p^n}}=M^\omega_n$ for any $\Gamma$-module $M$ and $\omega_n^\circ$ divides $\omega_n$, we deduce that $\beta_n$ is an isomorphism. In particular,  $\alpha_n$ is injective and  $\coker\alpha_n$ is bounded by $ \ker\gamma_n$. As in the proof of Theorem~\ref{controlordinary}, when $v\ne \mathfrak{p}$, $\ker\gamma_{n,v}$ is finite and bounded independently of $n$ thanks to \cite[Lemma 2.8]{ochiaiBK} (in fact our kernel is even smaller after taking $(-)^{\omega_n^\circ}$). It remains to study $\ker\gamma_{n,\mathfrak{p}}$, which we will show is in fact zero.

Fix $\lambda\in\{\alpha,\beta\}$. Let $K_{\lambda,n}$ (resp. $K_{\lambda,\infty}$) be the compositum of $K_n$ (resp. $K_\infty$) and $(\overline{\Qp})^{\ker\theta_\lambda}\subset K_g$, as in Lemma~\ref{lem:projectlambda}. We define $E^\bullet(K_{\lambda,n})$ and $E^\bullet(K_{\lambda,\infty})$ to be the $\mathbf 1$-isotypic components of $E^\bullet(K_{\lambda}(\mu_{p^{n+1}}))$ and $E^\bullet(K_{\lambda}(\mu_{p^\infty}))$ respectively. It can be checked using the description in \eqref{eq:Epm} that 
\[
E^\bullet(K_{\lambda,n})=\left\{P\in \hat E(K_{\lambda,n}):\Tr_{n/m+1}P\in \hat E(K_{\lambda,m})\ \forall m\in S_n^\circ\right\}.
\]

Consider the map
\[
\frac{\tupH^1(K_{\lambda,n},E[{p^\infty}])}{E^\bullet(K_{\lambda,n})\otimes\Qp/\Zp}\rightarrow \frac{\tupH^1(K_{\lambda,\infty},E[{p^\infty}])}{\left(E^\bullet(K_{\lambda,\infty})\otimes\Qp/\Zp\right)^{\omega_n^\circ}},
\]
which is injective by \cite[Proposition~2.18]{Kim} (bearing in mind our opposite choice of signs as explained in Convention~\ref{conv:pm}). Via the projection given in Lemma~\ref{lem:projectlambda}, we deduce that
\[
\frac{\tupH^1(K_{n,\p},A(1))}{\tupH^1_{\bullet\bullet}(K_{n,\p},A(1))}\rightarrow \frac{\tupH^1(K_{\infty,\p},A(1))}{\tupH^1_{\bullet\bullet}(K_{\infty,\p},A(1))^{\omega_n^\circ}}
\]
is injective. Hence, $\gamma_{n,p}$ is injective as required.
\end{proof}

\section{Results on Fitting ideals}

The goal of this section is to discuss and prove the main theorems discussed in the introduction. 

\subsection{The ordinary case}\label{ordinarysectionmaintheorems}
We suppose throughout this subsection that our fixed modular forms $f$ and $g$ are ordinary at $p$, that their residual representations $\overline{T}_f,\overline{T}_g$ are irreducible, and that
$f$ is non-CM. Moreover, we suppose that $\epsilon_f \epsilon_g$ is non-trivial,
that $(N_f,N_g)=1$, and one of the following conditions: \begin{itemize}
    \item $g$ is non-CM and has weight $\geq 2$;
    \item $g$ is of CM-type and $\epsilon_g$ is not 1 nor the quadratic character attached to the corresponding CM field of $g$;
    \item $g$ has weight 1.
\end{itemize} 
These hypotheses ensure that, for $p$ big enough and coprime with $N_fN_g$, the conclusions of Propositions~\ref{invariantsforalln} and  \ref{invariantsforallncaseweight1} hold as well as the big image assumption \textbf{Hyp(BI)} of \cite{KLZ2} (cf. \cite[Hypothesis 11.1.2, Remark 11.1.3]{KLZ2} and \cite[Proposition 4.2.1]{loeffler2017}). Moreover, they imply that the residual representation $\overline{T}(1+j)$ has no $G_\QQ$-stable sub-quotient isomorphic to $\mu_p$.
 From now on, we fix  a big enough prime $p$ for which these assumptions are satisfied.

We fix unit roots $\alpha_f$ and $\alpha_g$ for $f$ and $g$ respectively. By \cite[Corollary~2.2.4]{BL20}, for $h\in\{f,g\}$, $\{\omega_h,v_{h,\alpha}\}$ is an $\cO$-basis of $\Dcris(T_h)$. Consequently, $v_{\alpha,\alpha}=v_{f,\alpha}\otimes v_{g,\alpha}$ gives rise to an integral generator $\xi$ as in \cite[Definition~11.6.3]{KLZ2}. In particular, our geometric  $p$-adic $L$-function $L_p(f_\alpha,g)$ agrees with the $L_p(f,g,1+\mathbf{j})/\Omega$ considered in \emph{loc. cit.}
By Theorem 11.6.4 of \emph{op. cit.}, if $L_p(f_\alpha,g)(\theta\chi_{\rm cyc}^j) \ne 0$ for $\theta$ a Dirichlet character factoring through $\Gamma_1$, then  $\textup{Sel}_{\Gr}(K_\infty,A(1+j))$ is cotorsion and we have that 
\[ \pi_\Delta \circ \Tw^{j}L_p(f_\alpha,g) \in \textup{char}_\Lambda \textup{Sel}_{\Gr}(K_\infty,A(1+j))^\vee= \textup{char}_\Lambda(  \cX_{j,\infty}^\textup{Gr} ). \]
Note in particular that
\begin{equation}
    L_p(\alpha,j,n)\in\Lambda_n.\label{eq:integralord}
\end{equation}

Recall from Remark~\ref{rk:geo} that, for $\lambda_f \in \{ \alpha_f, \beta_f\}$,  the interpolation factor $\cE(\lambda_f,j)$ of the geometric $p$-adic $L$-function $L_p(f_\lambda,g)$  is given by
\[
\frac{\left(1-\frac{p^j}{\lambda_f\alpha_g}\right)\left(1-\frac{p^j}{\lambda_f\beta_g}\right)\left(1-\frac{\lambda_f'\alpha_g}{p^{1+j}}\right)\left(1-\frac{\lambda_f'\beta_g}{p^{1+j}}\right)}{\left(1-\frac{\lambda_f'}{p\lambda_f}\right)\left(1-\frac{\lambda_f'}{\lambda_f}\right)},
\]
where $\lambda_f' \in \{\alpha_f, \beta_f \} \setminus \{\lambda_f \}$.
\begin{defn}\label{def:nonanomalous}
We say that the triple $(f,g,j)$, with $k_g +1 \leq j \leq k_f$, is non-anomalous ordinary at $p$ if $\cE( \alpha_f, j )\in \cO^\times$.
\end{defn}

\begin{remark}\label{rmk:non-anomalous}
Notice that $(f,g,j)$ is  non-anomalous ordinary at $p$ whenever  $k_g+1 < j < k_f$. In the remaining cases, $\cE( \alpha_f, j )\in\cO^\times$ if the following inequalities hold: \begin{itemize}
    \item If $k_g+1 < j=k_f$, $\alpha_f \beta_gp^{-k_g-1} \not \equiv 1\mod\varpi$;
    \item If $0 < k_g+1 = j < k_f$, $\alpha_g \beta_f p^{-k_f-1} \not \equiv 1\mod\varpi$;
    \item If $0 < k_g+1 = j = k_f$, $\alpha_f \beta_g p^{-k_g-1} \not \equiv 1\mod\varpi$ and $\alpha_g \beta_f p^{-k_f-1} \not \equiv 1\mod\varpi$;
    \item If $k_g+1 = j = k_f=0$, 
    \[{\left\{ 
	\begin{array} {l l} 
	\alpha_f \alpha_g \not \equiv 1 \;\;\;\;\;\;\,\, \mod\varpi \\
    \alpha_f \beta_g \not \equiv 1 \;\;\;\;\;\;\;\, \mod\varpi\\
    \alpha_g \beta_f  p^{-1} \not \equiv 1 \;\; \mod\varpi \\
    \beta_g \beta_f p^{-1} \not \equiv 1 \;\,\, \mod\varpi
	\end{array}
	\right. } \]
	(Recall that  $\varpi$ is a fixed uniformizer of $\cO$.)
\end{itemize}  
\end{remark}

Before stating the main result of the section, we have the following crucial lemma. Its proof is based on the calculations for proving a similar statement on the Mazur--Tate elements for elliptic curves in \cite[\S 2.2]{KK}. We are very grateful to Chan-Ho Kim for having explained to us their strategy.

\begin{lemma}\label{auxiliarylemma}
Suppose that there exists a Dirichlet character $\theta$ of conductor a power of $p$ such that $L_p(f_\alpha,g)(\theta\chi_{\rm cyc}^j) \ne 0$, with $k_g+1 \leq j \leq k_f$. Let $(f,g,j)$ be non-anomalous ordinary at $p$, then there exists $C_n \in \varpi \Lambda_n$ such that \[ \Theta_{j,n}=(\alpha_f^{2n+2} + C_n)L_p(\alpha,j,n) \in \Lambda_n. \] 
\end{lemma}

\begin{proof}
The strategy of the proof, which consists of an induction argument on $n$, follows from the interpolation formulas and norm relations that these elements satisfy, as we now explain. 

Recall that in \S \ref{sec:Iwalg} we have introduced the natural projection $ \pi_n: \Lambda_n \to \Lambda_{n-1}$ and the trace map $ \nu_n: \Lambda_{n-1} \to \Lambda_n$  given by multiplication by $\Phi_n$.  We notice that, from the interpolation formulas that $\Theta_{j,0}$ satisfies, we have that \[ L_p(\alpha, j , 0) = C(f,j) \cdot \Theta_{j,0},\quad \text{where}\quad C(f,j): = \frac{(\beta_f^2 - \alpha_f^2) \cE( \alpha_f, j ) }{\beta_f^4\cE( \beta_f, j ) - \alpha_f^4 \cE( \alpha_f, j )}. \] 
It follows from our hypothesis that $(f,g,j)$ is non-anomalous ordinary at $p$ that  $\beta_f^4\cE(\beta_f,j)-\alpha_f^4\cE(\alpha_f,j)$ and $\cE(\alpha_f,j)$ are units in $\cO$. Thus, $C(f,j) \in \mathcal{O}^\times$. By \eqref{eq:integralord}, $L_p(\alpha, j , 0) \in \cO$, which tells us that \[\Theta_{j,0} = C(f,j)^{-1} L_p(\alpha, j , 0) \in \cO.\] 
This relation can be written as 
\[\Theta_{j,0} = (\alpha_f^2 + C_0) L_p(\alpha, j , 0) \in \cO,\]
where \[C_0 = C(f,j)^{-1} - \alpha_f^2 = \beta_f^2 \frac{\beta_f^2\cE(\beta_f,j)-\alpha_f^2\cE(\alpha_f,j)}{(\beta_f^2 - \alpha_f^2)\cE(\alpha_f,j)}\in \varpi \cO.\] 
On combining this with Proposition \ref{prop:Theta-L}, we deduce that
\begin{align*} L_p(\alpha ,j, 1 ) &= \alpha_f^{-4}\left(\Theta_{j,1}-\frac{\beta_f^2}{p}\nu_1(\Theta_{j,0})\right) \\ &= \alpha_f^{-4}\left(\Theta_{j,1}-\frac{\beta_f^2 }{p C(f,j)}\nu_1(L_p(\alpha ,j, 0 ))\right) \\
&= \alpha_f^{-4}\left(\Theta_{j,1}-C_1 L_p(\alpha ,j, 1 )\right)
\end{align*}
with $C_1 = \frac{\beta_f^2 }{p C(f,j)} \Phi_1 \in \varpi \Lambda_1$. It gives \[\Theta_{j,1} =( \alpha_f^{4}+C_1) L_p(\alpha ,j, 1)\in \Lambda_1. \]

Suppose that  \[\Theta_{j,n-1} =B_{n-1} L_p(\alpha ,j, n-1), \]
with $B_{n-1} \in \Lambda_{n-1}$. Then, on applying Proposition \ref{prop:Theta-L} again,  we deduce that
\begin{align*} L_p(\alpha ,j, n ) &= \alpha_f^{-2n-2}\left(\Theta_{j,n}-\frac{\beta_f^2}{p}\nu_n(\Theta_{j,n-1})\right) \\ &= \alpha_f^{-2n-2}\left(\Theta_{j,n}-\frac{\beta_f^2 }{p}\nu_n(B_{n-1} L_p(\alpha ,j, n-1))\right)
\\ &= \alpha_f^{-2n-2}\left(\Theta_{j,n}-\frac{\beta_f^2 }{p}\nu_n(B_{n-1} \pi_n(L_p(\alpha ,j, n)))\right) \\
 &= \alpha_f^{-2n-2}\left(\Theta_{j,n}-C_n L_p(\alpha ,j, n))\right),
\end{align*}
with $C_n= \frac{\beta_f^2}{p}\Phi_nB_{n-1}\in\varpi\Lambda_n$. This implies that \[\Theta_{j,n} =(\alpha_f^{2n+2} + C_n) L_p(\alpha ,j, n), \]
which proves our claim.
\end{proof}

We are now ready to prove our generalization of Theorem~\ref{mainthmkk} in the ordinary case (Theorem~\ref{maintheorem1} in the introduction).

\begin{theorem}\label{mainordinary}
Suppose that there exists a Dirichlet character $\theta$ of conductor a power of $p$ such that $L_p(f_\alpha,g)(\theta\chi_{\rm cyc}^j) \ne 0$, with $k_g+1 \leq j \leq k_f$. If $(f,g,j)$ is non-anomalous ordinary at $p$, then
\[ \left(\Theta_{j,n} \right) =  \left(\Theta_{j,n}, \nu_n(\Theta_{j,n-1}) \right)    \subseteq \Fitt_{\Lambda_n}\cX_{j,n}^\textup{Gr}. \]
\end{theorem}

\begin{proof}
By Corollary \ref{cor:Fittchar} and the inclusion of the Iwasawa main conjecture, we have that 
\[ (\pi_\Delta \circ \Tw^{j} L_p(f_\alpha,g)) \subseteq \textup{char}_\Lambda(\cX_{j,\infty}^\textup{Gr}) = \textup{Fitt}_\Lambda(\cX_{j,\infty}^\textup{Gr}).  \]
Thanks to our assumptions Theorem \ref{controlordinary} asserts that the natural maps $ r_n: \cX_{j,\infty}^\textup{Gr}/\omega_n \rightarrow \cX_{j,n}^\textup{Gr}$ are surjective. Thus, by \cite[Lemma~A.1]{KK}, we obtain the inclusion
\[
(L_p(\alpha ,j, n ))=(\pi_\Delta \circ \Tw^{j} L_p(f_\alpha,g)\mod \omega_n) \subseteq \Fitt_{\Lambda_n}(\cX_{j,\infty}^\textup{Gr})/\omega_n \subseteq \Fitt_{\Lambda_n}(\cX_{j,n}^\textup{Gr}).
\]
By Proposition \ref{prop:Theta-L},
 we have
\[
(L_p(\alpha ,j, n )) = \left( \alpha_f^{-2n-2}\left(\Theta_{j,n}-\frac{(\beta_f)^2}{p}\nu_n(\Theta_{j,n-1})\right)  \right) \subseteq \Fitt_{\Lambda_n}\cX_{j,n}^\textup{Gr},
\]
where $\nu_n : \Lambda_{n-1} \to \Lambda_n$ denotes the trace map. 
Under the non-anomalous ordinary condition at $p$, Lemma~\ref{auxiliarylemma} says that  $\Theta_{j,n}$ is a multiple of $L_p(\alpha,j, n )$ by an invertible element in $\Lambda_n$. Similarly, by Lemma \ref{auxiliarylemma}, we also deduce that $\nu_n(\Theta_{j,n-1})$ is a multiple of $L_p(\alpha,j,n)$ by an element in $\Lambda_n$. Indeed, we have \begin{align*}
\Theta_{j,n-1} &=(\alpha_f^{2n} + C_{n-1}) L_p(\alpha ,j, n-1)    \\
&=(\alpha_f^{2n} + C_{n-1}) \pi_n( L_p(\alpha ,j, n))
\end{align*}
Applying $\nu_n$, we get \[ \nu_n(\Theta_{j,n-1}) =(\alpha_f^{2n} + C_{n-1}) \Phi_n L_p(\alpha ,j, n). \]
This implies that the ideal $ \left(\Theta_{j,n}, \nu_n(\Theta_{j,n-1}) \right)$ is the principal ideal generated by $\Theta_{j,n}$ and that we get the inclusion
\[ (\Theta_{j,n}) = \left(\Theta_{j,n}, \nu_n(\Theta_{j,n-1}) \right) \subseteq \Fitt_{\Lambda_n}\cX_{j,n}^\textup{Gr}.\]
\end{proof}

\subsection{The supersingular case}\label{sec:ssFitting}

We now generalize  the calculations in \cite[\S4]{KK} to the setting of \S\ref{sec:controlpm}. In particular, \textbf{(H.Kim)} is in effect throughout. We shall link our Theta elements to the Bloch--Kato Selmer groups of $A(1)$, whose definition is recalled below.

\begin{defn}
For an integer $ n \ge 0$, we write $\Sel_\BK(K_n,A(1))$ for the Bloch--Kato Selmer group of $A(1)$ over $K_n$, that is
\[
\Sel_\BK(K_n,A(1))=\ker\Bigg( \tupH^{1}(\QQ_{\Sigma}/K_n, A(1)) \rightarrow \prod_{\nu|\Sigma} \tupH^{1}_{/ f}(K_{n,\nu}, A(1))    \Bigg),
\]
where $\Sigma$ is as defined in Definition~\ref{defn:selmergroupsordinary} and  $\tupH^{1}_{/ f}(K_{n,\nu}, A(1))=\tupH^{1}(K_{n,\nu}, A(1)) /\tupH^{1}_{ f}(K_{n,\nu}, T(1))\otimes\Qp/\Zp $  for $\nu|p$, with $\tupH^{1}_f(K_{n,\nu}, T(1)) $  as defined in \cite[(3.7.3)]{BK}.
We let $\cX_n^\BK$ denote the Pontryagin dual of $\Sel_\BK(K_n,A(1))$.
\end{defn}

We will study  $\cX_n^\BK$ via the plus and minus Selmer groups over $K_n$ introduced in Definition~\ref{defn:pmSel}, making use of Theorem~\ref{thm:controlpm}. The strategy is:
\begin{itemize}
    \item[(I)] Link the Theta elements $\Theta_n^\pm$ to $\cX_n^\pm$ (where  we write $\Theta^\pm_n$ in place of  $\Theta^\pm_{0,n}$ for simplicity);
    \item[(II)]Study relations between $\cX_n^\BK$ and $\cX_n^\pm$;
    \item[(III)]Combine (I) and (II) to study the growth of $\cX_n^\BK$ in terms of $\Theta_n^\pm$.
\end{itemize}

We introduce the following polynomials which allow us to carry out step (I).
\begin{defn}
For an integer $n\ge 0$, we define $\tilde\omega_n^\pm(X)$ to be the polynomials
\[
\frac{\omega_n(X)}{\omega_n^\mp(X)}=\frac{\omega_n^\pm(X)}{X}=\prod_{m\in S_{n+1}^\pm\setminus\{0\}}\Phi_m(X).
\]
\end{defn}

\begin{lemma}\label{cor:generalizing-pollack}
If $n$ is even, then
\begin{align*}
    -p\Theta_n^+&\equiv (\tilde\omega_n^+)^2\pi_\Delta L_p^{++}(f,g)\mod \omega_n;\\
    \Theta_n^-&\equiv (\tilde\omega_n^-)^2 \pi_\Delta L_p^{--}(f,g)\mod \omega_n.
\end{align*}
If $n$ is odd, then
\begin{align*}
    \Theta_n^+&\equiv (\tilde\omega_n^+)^2 \pi_\Delta L_p^{++}(f,g)\mod \omega_n;\\
    -p\Theta_n^-&\equiv (\tilde\omega_n^-)^2\pi_\Delta L_p^{--}(f,g)\mod \omega_n.
\end{align*}
\end{lemma}
\begin{proof}
A direct calculation gives
\begin{align*}
    \log_p^+&\equiv p^{-\lfloor \frac{n}{2}\rfloor-1}\tilde\omega_n^+\mod\omega_n;\\
    \log_p^-&\equiv p^{-\lfloor \frac{n+1}{2}\rfloor-1}\tilde\omega_n^-\mod\omega_n.
\end{align*}
On  combining this with Proposition~\ref{prop:pmTheta}, the lemma follows.
\end{proof}

\begin{proposition}\label{prop:inclueFittpm}
Suppose that the inclusion \eqref{eq:halfpmIMC} of Conjecture~\ref{conj:pmIMC} holds for $\bullet=\circ\in\{+,-\}$ (in particular, we have $\pi_\Delta L_p^{\bullet\bullet}(f,g)\in \Lambda$). Then we have the inclusion
\[
\left(( \tilde\omega_n^\bullet)^2\pi_\Delta L_p^{\bullet\bullet}(f,g)\mod\omega_n\right)\subset \left(\tilde\omega_n^{\bullet}\right)^2\Fitt_{\Lambda_n}\left(\cX_n^\bullet\right).
\]
\end{proposition}
\begin{proof}
By Corollary~\ref{cor:Fittchar}, the inclusion \eqref{eq:halfpmIMC} implies that
\[
(\pi_\Delta  L_p^{\bullet\bullet}(f,g))\subset \Fitt_\Lambda\cX_\infty^\bullet.
\]
 Let $\circ$ be the unique element of $\{+,-\}\setminus\{\bullet\}$. We deduce from Theorem~\ref{thm:controlpm} and   \cite[Appendix, properties 1 and 4 on pages 324-325]{MW}) (see also \cite[Lemmas~A.1 and A.6]{KK}) the inclusion
\[
(\pi_\Delta  L_p^{\bullet\bullet}(f,g)\mod\omega_n^{\circ})\subset \Fitt_{\Lambda_n/\omega_n^\circ}\cX_n^\bullet/\omega_n^\circ.
\]
As in \cite[proof of Corollary~3.10]{KK}, we have
\[
 \Fitt_{\Lambda_n/\omega_n^\circ}\cX_n^\bullet/\omega_n^\circ= \frac{\Fitt_{\Lambda_n}\cX_n^\bullet+(\omega_n^\circ)}{(\omega_n^\circ)},\]
 which in turn implies
 \[
 (\pi_\Delta  L_p^{\bullet\bullet}(f,g)\mod\omega_n)+(\omega_n^\circ)\subset \Fitt_{\Lambda_n}\cX_n^\bullet+(\omega_n^\circ).
 \]
 Thus, the result follows on multiplying by $(\tilde\omega_n^{\bullet})^2$ and the fact that $\tilde\omega_n^\bullet\omega_n^{\circ}=\omega_n$.
\end{proof}

On combining Corollary \ref{cor:generalizing-pollack} and Proposition~\ref{prop:inclueFittpm}, step (I) is established. For step (II), we begin with the following lemma due to Kim and Kurihara.

\begin{lemma}\label{lem:Fittlocal}
Let $\bullet,\circ\in\{+,-\}$. Then
\[
\Fitt_{\Lambda_n}\left(\frac{\tupH^1_f(K_{n,\p},A(1))}{\tupH^1_{\bullet\circ}(K_{n,\p},A(1))}\right)^\vee=\left(\tilde\omega_n^{\bullet}\tilde\omega_n^{\circ}\right)\Lambda_n,
\]
where $\p$ is the unique prime of $K_n$ lying above $p$.
\end{lemma}
\begin{proof}
Let $\mu\in\{\alpha_f,\beta_f\}$. On combing the short exact sequence of Corollary~\ref{cor:pmSES} and \cite[Propositions~4.11 and 5.8]{IP}, we have the following commutative diagram with exact rows:
\[
\xymatrix{
0\ar[r]&\tupH^1_f(\Qp,T_{f,\mu}(1))\ar[d]^{\cong}\ar[r]&\tupH^1_+(K_{n,\p},T_{f,\mu}(1))\oplus \tupH^1_-(K_{n,\p},T_{f,\mu}(1))\ar[r]\ar[d]^{\cong}&\tupH^1_f(K_{n,\p},T_{f,\mu}(1))\ar[r]\ar[d]^{\cong}&0\\
0\ar[r]&\tilde{\omega}^+_n\tilde{\omega}^-_n\Lambda_n\ar[r]&\tilde\omega^+_n\Lambda_n\oplus\tilde\omega^-_n\Lambda_n\ar[r]&(\tilde\omega_n^+,\tilde\omega_n^-)\Lambda_n\ar[r]&0.
}
\]
Thus,
\[
\frac{\tupH^1_f(K_{n,\p},A_{f,\mu}(1))}{\tupH^1_{\pm}(K_{n,\p},A_{f,\mu}(1))}\cong \left((\tilde\omega_n^+,\tilde\omega_n^-)\Lambda_n/\tilde{\omega}_n^\pm\right)\otimes\Qp/\Zp.
\]
The result now follows from \cite[Proposition~4.1]{KK} (again bearing in mind the choice of sign as explained in Convention~\ref{conv:mp}).
\end{proof}

\begin{corollary}\label{cor:includeFitt}
Let $\bullet\in\{+,-\}$. We have the inclusion
\[
\left(\tilde\omega_n^{\bullet}\right)^2\Fitt_{\Lambda_n}\left(\cX_n^\bullet\right)\subset \Fitt_{\Lambda_n}\left(\cX_n^\BK \right).
\]
\end{corollary}
\begin{proof}
This follows from Lemma~\ref{lem:Fittlocal} and the tautological exact sequence
\[
\left(\frac{\tupH^1_f(K_{n,\p},A(1))}{\tupH^1_{\bullet\bullet}(K_{n,\p},A(1))}\right)^\vee\rightarrow \cX_n^\BK \rightarrow \cX_n^\bullet\rightarrow 0.
\]
\end{proof}

This establishes step (II). We are now ready to carry out the final step (III), which proves Theorem~\ref{intro:maintheoremss} in the introduction.

\begin{theorem}\label{thm:final}
Let $E/\QQ$ be an elliptic curve with good supersingular reduction at $p$ and $a_p(E)=0$. Suppose that  \textbf{(H.Kim)} holds and that the inclusion  \eqref{eq:halfpmIMC} of Conjecture~\ref{conj:pmIMC} holds for $\bullet=\circ\in\{+,-\}$ (in particular, we have $\pi_\Delta L_p^{\bullet\bullet}(f,g)\in \Lambda$). {If $n \geq 2$ is even, then \[
\left(p\Theta_{n}^+,\Theta_{n}^-\right)\subset \Fitt_{\Lambda_n}\cX_n^\BK.
\]}
If $n \geq 1$ is odd, then \[
\left(\Theta_{n}^+,p\Theta_{n}^-\right)\subset \Fitt_{\Lambda_n}\cX_n^\BK.
\]
\end{theorem}
\begin{proof}
We consider the case $n$ is even. On combining Proposition~\ref{prop:inclueFittpm} and Corollary~\ref{cor:includeFitt}, we  have
\[
\left((\tilde\omega_n^\bullet)^2\pi_\Delta L_p^{\bullet\bullet}(f,g) \mod\omega_n\right)\subset \Fitt_{\Lambda_n}\cX_n^\BK
\]
for $\bullet\in\{+,-\}$.
The theorem now follows from Corollary~\ref{cor:generalizing-pollack}.
\end{proof}

The following lemma allows us to rewrite Theorem~\ref{thm:final} in a similar form to the (ss) case of \cite[Theorem 1.14]{KK}.

\begin{lemma}\label{lem:trace}
Suppose that $n\ge2$ is even, then 
\[
\Theta_{n}^+=-\nu_n\left(\Theta_{n-1}^+\right).
\]
Otherwise, if $n\ge 1$ is odd, we have
\[
\Theta_{n}^-=-\nu_n\left(\Theta_{n-1}^-\right).
\]
\end{lemma}
\begin{proof}
We only prove the even case since the odd case can be dealt with in the same way. Let $\theta$ be a Dirichlet character on $G_n$. If its conductor is $p^{n+1}$, then  Lemma \ref{lem:evaluateTheta} tells us that
\[
\Theta_{n}^+(\theta)=0.
\]
 Since $\nu_n\left(\Theta_{n-1}^+\right)=\Phi_n\Theta_{n-1}^+$ is divisible by $\Phi_n$, it vanishes at $\theta$ as well.
 
 Suppose that $\theta$ is a of conductor $p^m$, with $m\le n$. Then, our interpolation formulae give
 \begin{align*}
     \Theta_{n}^+(\theta)&=\frac{\alpha_f^{2n+2}}{4}\pi_\Delta\left(L_p(f_\alpha,g)+L_p(f_\beta,g)+L_p^?(f_\beta,g_\beta)+L_p^?(f_\alpha,g_\beta)\right)(\theta),\\
     \Theta_{n-1}^+(\theta)&=\frac{\alpha_f^{2n}}{4}\pi_\Delta\left(L_p(f_\alpha,g)+L_p(f_\beta,g)+L_p^?(f_\beta,g_\beta)+L_p^?(f_\alpha,g_\beta)\right)(\theta).
      \end{align*}
      Recall that $\alpha_f^2=-p$. This allows us to combine the two equations above to obtain
      \[
         \Theta_{n}^+(\theta)=-p\Theta_{n-1}^+(\theta)=-\nu_n\left(\Theta_{n-1}^+\right)(\theta)
 \]
as $\Phi_n(\theta)=p$. Thus, we deduce the equality
 \[
\Theta_{n}^+=-\nu_n\left(\Theta_{n-1}^+\right)
 \]
 as required.
\end{proof}
Thus, we may now rewrite Theorem~\ref{thm:final} as follows:
\begin{corollary}
 Suppose that the inclusion \eqref{eq:halfpmIMC} of Conjecture~\ref{conj:pmIMC} holds for $\bullet\in\{+,-\}$. If $n\ge 2$ is even, then {
\[
\left(p\nu_n\left(\Theta_{n-1}^+\right),\Theta_{n}^-\right)\subset \Fitt_{\Lambda_n}\cX_n^\BK.
\]
If $n\ge 1$ is odd, then
\[
\left(\Theta_{n}^+,p\nu_n\left(\Theta_{n-1}^-\right)\right)\subset \Fitt_{\Lambda_n}\cX_n^\BK.
\]}
\end{corollary}

\subsection{Study of elliptic curves over number fields}\label{subsec:ellipticcurvesovernumberfields}

Let $E$ be an elliptic curve over $\QQ$ without CM, of conductor $N_E$. Let $\rho$ be a two dimensional odd irreducible Artin representation of $G_\QQ$ and denote by $F$ its splitting field and by $N_\rho$ its conductor.  The representation $\rho$ takes values in a finite extension $L_\rho$ of $\QQ$. Fix $\mathfrak{P}$ a prime of $L_\rho$ above $p$. We let $L$ be a finite extension of $L_{\rho,\mathfrak{P}}$ and denote by $\cO$ its ring of integers. We assume for the rest of the section that the following list of hypotheses hold.  

\begin{enumerate}
    \item $N_E$ and $N_\rho$ are coprime;
    \item The prime $p\geq 5$ and $p\nmid N_EN_\rho[F:\QQ]$; 
    \item The representation $G_\QQ \to \operatorname{Aut}_{\Zp}(T_pE)$ is surjective;
    \item $\rho(\operatorname{Frob}_p)$ has distinct eigenvalues modulo $\mathfrak{P}$;
    \item If $E$ has ordinary reduction at $p$, the triple $(E,\rho,0)$ is non-anomalous, i.e. $(f,g,0)$ is (cf. Remark \ref{rmk:non-anomalous});
    \item If $E$ has supersingular reduction at $p$, we assume that $a_p(E)=0$ and that \textbf{(H.Kim)} holds, that is, $4\nmid [F:\QQ]$.
\end{enumerate}

Let $F_{\infty}=F K_\infty$ be the cyclotomic $\Zp$-extension of $F$, with finite layers $F_{n}$ so that $[F_{n}:F]=p^n$ and $F_{n}$ is the compositum of $K_n$ and $F$. 

\begin{defn}\label{def:selmerellipticc}
Let $\textup{Sel}_{n}( E/F)$ denote the Selmer group of $E$ over $F_n$:
\[  \ker\left( \tupH^1(\operatorname{Gal}(\QQ_\Sigma/F_{n}), E[p^\infty]) \to \prod_{v \in \Sigma}\frac{\tupH^1(F_{n,v}, E[p^\infty])}{E(F_{n,v}) \otimes \Qp/\Zp} \right),\] 
where $\Sigma$ is a finite set of primes containing primes above $p$, $\infty$, and the primes dividing $N_E N_\rho$. 
\end{defn}

We are interested in studying the $\rho$-isotypic component of these Selmer groups, whose definitions we review below.
\begin{defn}\label{def:selmerrhoisot}
For all $n \geq 0$, define the $\rho$-isotypic component \[ \textup{Sel}_{n}( E/F)_{(\rho)} := \operatorname{Hom}_{\operatorname{Gal}(F_{n}/K_n)}(\rho, \textup{Sel}_{n}( E/F)), \]
and its Pontryagin dual $\cX_n(E)_{(\rho)}:=\textup{Sel}_{n}( E/F)_{(\rho)}^\vee$.
\end{defn}

Notice that $E$ and $\rho^*$ (the contragradient representation of $\rho$) are modular, with $E$ corresponding to a weight 2 modular form $f$, while  $\rho^*$ to a weight 1 modular form $g$, such that there is an isomorphism of $G_\QQ$-representations $E[p^\infty] \otimes \rho^* \simeq A_{f,g}(1)$.

\begin{notation}
Denote by $\textup{Sel}(\star, E[p^\infty]\otimes \rho^*)$ the Selmer group $\textup{Sel}_\Gr(\star, A_{f,g}(1))$ (resp. $\textup{Sel}_\BK(\star, A_{f,g}(1))$)
if $E$ is ordinary (resp. supersingular) at $p$.
\end{notation}

\begin{remark}\label{rem:comparisonselmer}
The $\rho$-isotypic component $\textup{Sel}_{n}( E/F)_{(\rho)}$ can be identified with the $\operatorname{Gal}(F_{n}/K_n)$-invariant classes of $\textup{Sel}(F_{n}, E[p^\infty]\otimes \rho^*)$.
\end{remark}

In the following lemma, we study the relation between $\textup{Sel}_{n}( E/F)_{(\rho)}$ and $\textup{Sel}(K_{n}, E[p^\infty]\otimes \rho^*)$.
\begin{lemma}\label{comparisonrhoisotypiccompandGr}
For every $0 \leq n \leq \infty$, the natural restriction map
\[ \textup{Sel}(K_n, E[p^\infty] \otimes \rho^*) \longrightarrow \textup{Sel}_{n}( E/F)_{(\rho)}\]
is an isomorphism.
\end{lemma}

\begin{proof}
For every $0 \leq n \leq \infty$, by the inflation-restriction sequence, we have a map
\begin{equation}\label{restrictionmap}
\textup{Sel}(K_n, E[p^\infty] \otimes \rho^*) \longrightarrow \textup{Sel}(F_{n}, E[p^\infty] \otimes \rho^*)^{H_n}=\textup{Sel}_{n}( E/F)_{(\rho)}, \end{equation}
where $H_n:= \operatorname{Gal}(F_{n}/K_n)$, which sits in the following diagram: 

\[ \xymatrix{ \textup{Sel}(K_n, E[p^\infty] \otimes \rho^*) \ar@{^{(}->}[r] \ar[d]^{r_n} & \tupH^{1}(K_n, E[p^\infty] \otimes \rho^*) \ar[r] \ar[d]^{s_n} & \mathcal{Q}(K_n,E[p^\infty] \otimes \rho^*) \ar[d]^{t_n=\prod t_{n,v}} \\
 \textup{Sel}_{n}( E/F)_{(\rho)} \ar@{^{(}->}[r] & \tupH^{1}(F_{n},E[p^\infty] \otimes \rho^*)^{H_n} \ar[r] & \mathcal{Q}(F_{n},E[p^\infty] \otimes \rho^*)^{H_n},}\]
where, in view of Remark \ref{rem:comparisonselmer}, we have \begin{align*}
    \mathcal{Q}(K_n,E[p^\infty] \otimes \rho^*)&=\prod\limits_{\mathfrak{p} \neq v } \tupH^{1}_{/f}(K_{n,v}, A_{f,g}(1))  \times \tupH^{1}_{/\mathcal{L}_p}(K_{n,\mathfrak{p}}, A_{f,g}(1))\\ \mathcal{Q}(F_{n},E[p^\infty] \otimes \rho^*)^{H_n}&=\prod\limits_{\mathfrak{p} \neq v} \tupH^{1}_{/f}(F_{n,v},A_{f,g}(1))^{H_{n,v}} \times \prod\limits_{\mathcal{P}|\mathfrak{p}}\tupH^{1}_{/\mathcal{L}_p}(F_{n,\mathcal{P}}, A_{f,g}(1))^{H_{n,\mathcal{P}}},
\end{align*}
where $\mathcal{L}_p \in \{f, \Gr \}$.

By the inflation-restriction sequence the kernel of $s_n$ is equal to $\tupH^{1}({\rm Gal}(F_{n}/K_n),(A_{f,g}(1))^{G_{F_{n}}})$, while its cokernel is equal to $\tupH^{2}({\rm Gal}(F_{n}/K_n),(A_{f,g}(1))^{G_{F_{n}}})$. Since $p\nmid [F:\QQ]$, these two spaces are zero. By the snake lemma, $r_n$ is injective and it has cokernel bounded by the kernel of $t_n$. 

Let $v \ne \mathfrak{p}$ be a place of $K_n$ above the rational prime $\ell \ne p$. We have 
\[ t_{n,v}:\tupH^{1}_{/f}(K_{n,v},A_{f,g}(1)) \to \prod\limits_{w | v} \tupH^{1}_{/f}(F_{n,w}, A_{f,g}(1))^{H_{n,w}},\]
where $H_{n,w}= {\rm Gal}(F_{n,w}/ K_{n,v})$. Note that
\[ \tupH^{1}_{f}(B,A_{f,g}(1))=0,\]
where $B \in \{K_{n,v},F_{n,w}\}$, as $E(C)\otimes\Qp/\Zp=0$ for any $\ell$-adic field $C$ with $\ell \ne p$ (e.g. \cite[Proposition 2.1]{greenberg1999iwasawaell}). Thus, by the inflation-restriction sequence, we have  \[ \ker(t_{n,v}) = \tupH^{1}({\rm Gal}(F_{n,w}/K_{n,v}),E(F_{n,w})[p^\infty] \otimes \rho^*)=0,\]
since $[F_{n,w}:K_{n,v}]$ is coprime with $p$.

We are left with the analysis of the kernel of  $t_{n,\mathfrak{p}}$ at places $\mathfrak{p}$ above $p$. Let $\mathcal{P}$ be a prime of $F_{n}$ above $\mathfrak{p}$; the map $t_{n,\mathfrak{p}}$ sits in the following diagram: 
\[ \xymatrix{ \tupH^{1}_{\mathcal{L}_p}(K_{n,\mathfrak{p}}, A_{f,g}(1)) \ar@{^{(}->}[r] \ar[d]^{t_{n,\mathfrak{p}}''} & \tupH^{1}(K_{n,\mathfrak{p}},  A_{f,g}(1)) \ar[r] \ar[d]^{t_{n,\mathfrak{p}}'} & \tupH^{1}_{/\mathcal{L}_p}(K_{n,\mathfrak{p}},   A_{f,g}(1)) \ar[d]^{t_{n,\mathfrak{p}}} \\
 \tupH^{1}_{\mathcal{L}_p}(F_{n,\mathcal{P}},  A_{f,g}(1))^{H_{n,\mathcal{P}}} \ar@{^{(}->}[r] & \tupH^{1}(F_{n,\mathcal{P}},   A_{f,g}(1))^{H_{n,\mathcal{P}}} \ar[r] & \tupH^{1}_{/\mathcal{L}_p}(F_{ n,\mathcal{P}},   A_{f,g}(1))^{H_{n,\mathcal{P}}}  }\]
where $H_{n,\mathcal{P}}= {\rm Gal}(F_{n,\mathcal{P}}/ K_{n,\mathfrak{p}})$ and $\mathcal{L}_p \in \{f, \Gr \}$. 
As above, by applying the inflation-restriction exact sequence and the fact that $p \nmid [F_{n,\mathcal{P}}: K_{n,\mathfrak{p}}]$, we conclude that ${t_{n,\mathfrak{p}}'}$ is an isomorphism. Thus, ${t_{n,\mathfrak{p}}''}$ is injective and the kernel of ${t_{n,\mathfrak{p}}}$ is isomorphic to the cokernel of ${t_{n,\mathfrak{p}}''}$. We now show that ${t_{n,\mathfrak{p}}''}$ is surjective.  Recall from \S\ref{sec:setupSS}  that 
\[ A_{f,g}(1)|_{G_{\Qp}} = (E[p^\infty] \otimes \rho^*)|_{G_{\Qp}}\cong E[p^\infty](\theta_\alpha)\oplus E[p^\infty](\theta_\beta).\]
Let $\mu \in \{\alpha_f,\beta_f\}$ and $J_\mu$ be the finite unramified extension of $\Qp$ given by $(\overline{\Qp})^{{\rm ker}\theta_\mu}$. Similarly to Lemma \ref{lem:projectlambda}, we have \begin{equation}\label{eq:compatp}
    E(J_\mu B) \otimes \Qp / \Zp = \tupH^{1}_{\mathcal{L}_p}(J_\mu B,  E[p^\infty])  \simeq \tupH^{1}_{\mathcal{L}_p}(B,  E[p^\infty](\theta_\mu)) ,
\end{equation} 
where $B \in \{K_{n,\mathfrak{p}},F_{n,\mathcal{P}}\}$. 
As $ \tupH^{1}_{\mathcal{L}_p}(\star,  A_{f,g}(1)) = \tupH^{1}_{\mathcal{L}_p}(\star,  E[p^\infty](\theta_\alpha)) \oplus \tupH^{1}_{\mathcal{L}_p}(\star, E[p^\infty](\theta_\beta))  $, \eqref{eq:compatp} identifies \[\tupH^{1}_{\mathcal{L}_p}(F_{n,\mathcal{P}},A_{f,g}(1))^{H_{n,\mathcal{P}}} \simeq (E(J_\alpha F_{n,\mathcal{P}})^{H_{n,\mathcal{P}}} \oplus E(J_\beta F_{n,\mathcal{P}})^{H_{n,\mathcal{P}}}) \otimes \Qp / \Zp \]
as a submodule of $(E(J_\alpha  K_{n,\mathfrak{p}}) \oplus E(J_\beta K_{n,\mathfrak{p}})) \otimes \Qp/\Zp$. 
This shows that ${t_{n,\mathfrak{p}}}$ is injective and thus that ${t_{n,\mathfrak{p}}''}$ is surjective.
\end{proof}

\begin{notation}
If $E$ has ordinary reduction at $p$, denote by $\Theta_n(E,\rho) $ the Theta element $\Theta_{0,n}$ corresponding to the weight two modular form $f$, the weight one modular form $g$, and the (only) twist $j=0$.

Similarly, if $E$ has supersingular reduction at $p$, we write $\Theta_n^{\pm}(E,\rho) $ for the Theta elements $\Theta_{0,n}^\pm$ corresponding to the weight two modular form $f$, the weight one modular form $g$, and the  twist $j=0$. 
\end{notation}

\begin{remark}
As $p \geq 5$, Proposition \ref{cohomconstruction} shows that   $\Theta_n(E,\rho), \Theta_n^{\pm}(E,\rho) \in \Lambda_n$.
\end{remark}

We are now ready to prove Theorem~\ref{introthm:EC} in the introduction.
\begin{theorem}\label{thm:EC}
Let us assume the hypotheses (1)--(4) made at the beginning of \S \ref{subsec:ellipticcurvesovernumberfields}. Furthermore, suppose that $L(E,\rho, \theta^{-1},1) \ne 0$ for some finite order character $\theta$ on $\Gamma_1$.
\begin{itemize}
    \item[(i)] If $E$ has ordinary reduction at $p$ and hypothesis (5) holds, then for all $n \geq 0$, we have 
\[\left(\Theta_n(E,\rho) \right) \subseteq {\rm Fitt}_{\Lambda_n}\cX_n(E)_{(\rho)}.\]
\item[(ii)] Suppose that $E$ has supersingular reduction at $p$ and that hypothesis (6) holds and let $p^m$ be the conductor of $\theta$. 
{If $m$ is even, then for all odd $n\ge0$, we have
\[\left(\Theta^+_n(E,\rho) \right) \subseteq {\rm Fitt}_{\Lambda_n}\cX_n(E)_{(\rho)}.\]
If $m$ is odd, then for all even $n\ge0$, we have
\[\left(\Theta^-_n(E,\rho) \right) \subseteq {\rm Fitt}_{\Lambda_n}\cX_n(E)_{(\rho)}.\]}
\end{itemize}
\end{theorem}

\begin{proof}
We consider the ordinary case first. The proof follows from Theorem \ref{mainordinary} and Lemma \ref{comparisonrhoisotypiccompandGr}, as we now show. Recall that $E$ and $\rho$ correspond to the modular forms $f$ and $g$. 

As $L(E,\rho, \theta^{-1},1) \ne 0$, the $p$-adic $L$-value $L_p(f_\alpha,g)(\theta)$ is non-zero. This together with hypotheses (1)-(4) shows that the hypotheses of \cite[Theorem 11.6.4]{KLZ2} hold (see also \cite[Theorem 11.7.4]{KLZ2}). Thus, $\textup{Sel}(K_\infty, E[p^\infty] \otimes \rho^*)=\textup{Sel}_{\Gr}(K_\infty, A_{f,g}(1))$ is $\Lambda$-cotorsion  and \[\pi_\Delta  L_p(f_\alpha,g) \in \textup{char}_\Lambda(  \cX_{\infty,0}^\textup{Gr} ) = \textup{Fitt}_\Lambda(  \cX_{\infty,0}^\textup{Gr} ). \]
By the additional hypothesis (5), we apply Theorem \ref{mainordinary} to deduce that
\[\left(\Theta_n(E,\rho) \right) \subseteq {\rm Fitt}_{\Lambda_n}\cX_{n,0}^\textup{Gr}. \]
As the degree $[F:\QQ]$ is coprime with $p$, we apply Lemma \ref{comparisonrhoisotypiccompandGr}, which implies the equality ${\rm Fitt}_{\Lambda_n}\cX_{n,0}^\textup{Gr} ={\rm Fitt}_{\Lambda_n}\cX_n(E)_{(\rho)}$.

{We now turn our attention to the supersingular case and suppose that $m$ is even (the odd case is proved analogously). It follows from \eqref{eq:interpolation}, \eqref{eq:interpolate?}, \eqref{eq:Lp++} and the fact that $\log_p^+$ does not vanish at $\theta$ that $L_p^{++}(f,g)\ne 0$. Then our running hypotheses allow us to apply \cite[Theorem~6.2.4]{BLLV}, which gives the inclusion 
\[
  \pi_\Delta L_p^{++}(f,g)\in\Char_\Lambda\Sel_{++}(K_\infty,A_{f,g}(1))^\vee
  \]
  (see Remark~\ref{conjssknownres}). The proof of Theorem~\ref{thm:final} then gives
  \[
  \Theta_n^+(E,\rho)\in  {\rm Fitt}_{\Lambda_n}\cX_{n}^\textup{BK}.
  \]
  Hence the result follows from Lemma~\ref{comparisonrhoisotypiccompandGr}.}
\end{proof}

We now conclude this section by giving an upper bound to the dimension of the $\rho$-isotypic component of the Mordell--Weil group of $E(F_{n})$. Before doing so, we recall the definition of the order of vanishing at characters of an element in $\Lambda_n$.

For $\chi:G_n \to \overline{\QQ}_p^\times$ a character, let $\cO[\chi]$ be the $\cO$-algebra generated by the image of $\chi$. Denote also by $\chi$ the induced homomorphism $\Lambda_n \to \cO[\chi]$ and define by $I_\chi$ its augmentation ideal.   

\begin{defn}
An element $z \in \Lambda_n$ vanishes to order $r$ at $\chi$ if $z \in I_\chi^r \backslash I_\chi^{r+1}$. Moreover, we say that $z$ vanishes to infinite order if it belongs to all powers of $I_\chi$. We denote by $\ord_\chi z$ the order of vanishing of $z$ at $\chi$.
\end{defn}

 Consider the $\rho$-isotypic component 
\[ E(F_{n})_{(\rho)}:= {\rm Hom}_{{\Gal(F_{n}/K_n)}}(\rho, E(F_{n})  \otimes L). \]
Moreover, given $\chi:G_n \to \overline{\QQ}_p^\times$, we let  \[ E(F_{n})_{(\rho)}^\chi:=\{P \in E(F_{n})_{(\rho)} \otimes \overline{\QQ}_p \;:\; \sigma \cdot P = \chi(\sigma)P\; \text{ for all }\sigma \in G_n \}.\]  
As $E(F_{n}) \otimes \Qp/ \Zp \hookrightarrow {\rm Sel}_n(E/F)$, Theorem \ref{thm:EC} gives an upper bound on the rank of $E(F_{n})_{(\rho)}^\chi$ in terms of the order of vanishing of the Theta elements at $\chi$, which is Corollary~\ref{cor:introfinal} in the introduction.
\begin{corollary}\label{cor:final}
We keep the same assumptions of Theorem~\ref{thm:EC}. Let $\chi$ be a character on $G_n$. 
\begin{itemize}
    \item[(i)] Suppose that $E$ has ordinary reduction at $p$, then $\dim_{\overline{\QQ}_p} E(F_{n})_{(\rho)}^\chi \le \ord_\chi\Theta_{n}(E,\rho)$. 
    \item[(ii)] Suppose that $E$ has supersingular reduction at $p$, then 
    \[
    \dim_{\overline{\QQ}_p} E(F_{n})_{(\rho)}^\chi\le\begin{cases}
    \ord_\chi\Theta_n^+(E,\rho)&\text{if $n$ is odd,}\\
    \ord_\chi\Theta_n^-(E,\rho)&\text{if $n$ is even.}
    \end{cases}
    \]
\end{itemize}

 \end{corollary}
 \begin{proof}
This is a consequence of Theorem~\ref{thm:EC}, following from the same proof as \cite[Proposition 3]{MaTa} (see also \cite[Corollary~1.16]{KK}).
 \end{proof}

\appendix

\section{Perrin-Riou pairings and arithmetic construction of Theta elements}\label{arithmeticconstructionappendix}

In this appendix, we show how to construct the Theta elements defined in \S \ref{sec:thetaelements} as images of certain Beillinson--Flach classes under the pairing of Perrin-Riou at finite levels. This is the result of the explicit reciprocity laws that Beillinson--Flach classes satisfy and the relation between Perrin-Riou's big logarithm and the corresponding pairing at finite levels. Our result should be compared to \cite[Lemma 7.2]{Kurihara2002}, where the modular elements defined by Mazur and Tate are related to values under Perrin-Riou's pairing of the zeta elements of Kato.

We have already seen in Remark~\ref{rk:integral} that the Theta elements we have defined have bounded denominators as $n$ varies. In the case when the weight of $g$ is 1, we show that the Theta elements are integral under a Fontaine-Laffaille condition via their links with Beilinson--Flach classes.

\subsection{The Perrin-Riou pairing}\label{PRpairingfinlev}
In this section $V$ denotes a $d$-dimensional $L$-vector space equipped with a continuous crystalline representation of $G_{\Qp}$. 
We recall the definition of the Perrin-Riou pairing over $K_n$, which will be utilized in \S \ref{BFversustheta} to establish the connection of our Theta elements with the Beillinson--Flach classes.

 Let $[-,-]$ denote the natural pairing $\Dcris(V(1+j)) \times \Dcris(V^*(-j)) \to L$. By linearity, we extend it to 
\[[-,-] :  \Dcris(V(1+j)) \otimes L(\mu_{p^n}) \times \Dcris(V^*(-j)) \otimes L(\mu_{p^n})\to  L(\mu_{p^n}).\]
Recall we have the twisting operator \[ \gamma_{j,n} : \Dcris(V(j)) \to  \Dcris(V(j)) \otimes L(\mu_{p^n}),\] defined by sending \[v \mapsto p^{-n} \left( \sum_{i=0}^{n-1} \varphi^{i-n}(v) \otimes \zeta_{p^{n-i}} + (1-\varphi)^{-1}(v)  \right ) .\] 
We define the following Perrin-Riou pairing.
\begin{defn}
Let $\mathcal{G}_n$ denote the Galois group ${\rm Gal}(\Qp(\mu_{p^n})/\Qp)$. We define the Perrin-Riou pairing over $\Qp(\mu_{p^n})$ by
\[P_n( -,-) : \Dcris(V(1+j)) \times H^1(\Qp(\mu_{p^n}),V^*(-j)) \to  L (\mu_{p^n})[\mathcal{G}_n]\] 
sending
\[ (v,z) \mapsto \left[ \sum_{\sigma \in \mathcal{G}_n} \gamma_{n,1+j}(v)^\sigma \sigma , \sum_{\sigma \in \mathcal{G}_n}  {\operatorname{exp}}^*(z^\sigma) \sigma^{-1} \right].\]
\end{defn}

\begin{lemma}[\cite{Kurihara2002}, Lemma 3.2] Let $v\in \Dcris(V(1+j))$ and $v\in H^1(\Qp(\mu_{p^n}),V^*(-j))$.  We have 
\[ P_n(v,z)= \sum_{\sigma \in \mathcal{G}_n}  \operatorname{Tr}_{L(\mu_{p^n})/L}[ \gamma_{n,1+j}(v)^\sigma ,  {\operatorname{exp}}^*(z)] \sigma \in  L[\mathcal{G}_n]. \]
\end{lemma}

Note in particular that $P_n$ takes values in $L[\mathcal{G}_n]$. Finally, we  recall a result on the values of this pairing when evaluated at a finite character of $\cG_n$.
\begin{lemma}
Let $v\in \Dcris(V(1+j))$ and $v\in H^1(\Qp(\mu_{p^n}),V^*(-j))$. 
\begin{itemize}
\item Let $\psi$ be a character of $\mathcal{G}_n$ of conductor $p^m$ for $0 < m \leq n$; it extends to $\psi: L[\mathcal{G}_n] \to L(\mu_{p^m})$. Then, we have \[ \psi ( P_n(v,z) ) = p^{-m} \tau(\psi) \sum_{\sigma \in \mathcal{G}_n} \psi^{-1}(\sigma) \left[\varphi^{-m}(v), \operatorname{exp}^*(z^\sigma)\right ],\]
where $ \tau(\psi)$ denotes the Gauss sum $\sum_{\sigma \in \mathcal{G}_m} \psi(\sigma) \zeta_{p^m}^\sigma$. 
\item Let $\mathbbm{1}$ be the trivial character of $ \mathcal{G}_n$, then
 \[ \mathbbm{1}( P_n(v,z) ) =  \left[(1-\varphi)^{-1}(1-\frac{\varphi^{-1}}{p})(v), \operatorname{exp}^*(\operatorname{cores}_{\Qp}^{\Qp(\mu_{p^n})}(z)) \right].\]
 \end{itemize}
\end{lemma}
\begin{proof}
This is proven in \cite[Lemmas 3.4-5]{Kurihara2002} and \cite[Lemma 3.5]{Lei2011}.
\end{proof}

Let $\Qpn \subset \Qp(\mu_{p^{n+1}})$ be the $n$-th layer of the cyclotomic $\Zp$-extension of $\Qp$ with Galois group $G_n:=\operatorname{Gal}(\Qpn/\Qp)$. Similarly to \cite{Kurihara2002}, we define the Perrin-Riou pairing over $\Qpn$ by
\[ \mathcal{P}_n(-,-): \Dcris(V(1+j)) \times H^1(\Qpn,V^*(-j)) \to  L[G_n],\] 
sending $(v,z) \mapsto  \pi_n P_{n+1}(v,\iota_n(z))$, 
where $\pi_n: L[\mathcal{G}_{n+1}] \to L[G_n]$ is the natural projection and $\iota_n: H^1(\Qpn,V^*(-j)) \to H^1(\Qp(\mu_{p^{n+1}}),V^*(-j))$ is the map induced by $\Qpn \subset \Qp(\mu_{p^{n+1}})$. Notice that our definition of $\mathcal{P}_n(-,-)$ differs from the one of \cite{Kurihara2002} by a factor of $1/(p-1)$.

\subsection{Wach modules and integral elements}

In this section, we prove a general result on the integrality of the Perrin-Riou pairings defined in the previous section. This will be employed in \S \ref{BFversustheta} to show the integrality of the Theta elements of Definition \ref{analyticdeftheta}.

We introduce the following convention.
 \begin{conv}\label{conv}
  Let $n\ge 1$ be an integer and $U$ an $E$-vector space. If $M=(m_{ij})$ is an $n\times n$ matrix defined over $E$ and $u_1,\ldots u_n$ are elements in $U$, we write
  \[
  \begin{pmatrix}
  u_1&\cdots &u_n
  \end{pmatrix}\cdot M
  \]
  for the row vector of elements in $U$ given by $\sum_{i=1}^nu_im_{ij}$, $j=1,\ldots,n$.
 \end{conv}

Throughout this section,  $T$ denotes a free $\cO$-module of rank $d$ equipped with a continuous crystalline action of $G_{\Qp}$. Furthermore, we assume that all Hodge-Tate weights of $T$ are non-negative and that the Fontaine-Laffaille condition is satisfied:
\begin{itemize}
    \item[\textbf{(H.FL)}] The Hodge-Tate weights of $T$ are inside an interval $[a-p+1;a]$ for some integer $a\le 0$. Furthermore, the slope of the action of $\vp$ on $\Dcris(T\otimes L)$ does not attain $-a$ and $-a+p-1$ simultaneously.
\end{itemize}

 The Wach module $\NN(T^*)$ is a free module of rank $d$ over $\AA^+_{\Qp}$, equipped with a canonical isomorphism 
\begin{equation}\label{eq:comparison1}
 \NN(T^*) / \pi \NN(T^*) \cong \Dcris(T^*).
\end{equation}
Let $v_i\in\Dcris(T^*)$ be an $\cO$-basis of $\Dcris(T^*)$ respecting its filtration in the sense of \cite[\S V.2]{berger04}. Let $\{n_i:i=1,\ldots,d\}$ be the $\AA^+_{\Qp}$-basis of $\NN(T^*)$ given by Proposition V.2.3 of \emph{op. cit.} (In particular $v_i$ equals the image of $n_i$ under \eqref{eq:comparison1}).
If $A$ and $P$ are  the matrices of $\vp$ with respect to the bases $\{v_i\}$ and $\{n_i\}$ respectively, under Convention~\ref{conv},  we have the equations
 \begin{align}
 \begin{pmatrix}
   \vp(v_1)&\cdots&\vp(v_d)
  \end{pmatrix}&=
  \begin{pmatrix}
   v_1&\cdots&v_d
  \end{pmatrix}\cdot A,\notag\\
    \begin{pmatrix}
   \vp(n_1)&\cdots&\vp(n_d)
  \end{pmatrix}&=
  \begin{pmatrix}
   n_1&\cdots &n_d
  \end{pmatrix}\cdot P.\label{eq:mutliplyP}
 \end{align}

There is an isomorphism
\begin{equation} 
 \Brig\left[ \tfrac{t}{\pi}\right] \otimes_{\AA^+_{\Qp}} \NN(T^*) \cong \Brig\left[ \tfrac{t}{\pi}\right] \otimes_{\Zp} \Dcris(T^*)
\end{equation}
compatible with \eqref{eq:comparison1} via reduction mod $\pi$. Let $M \in \mathbf{GL}_4\left(\Brig\left[ \tfrac{t}{\pi}\right]\right)$ be the matrix of this isomorphism with respect to our bases $\{v_i\}$ and $\{n_i\}$, so that
\begin{equation}\label{eq:matricepassage}
 \begin{pmatrix}
  n_1&\cdots &n_d
 \end{pmatrix}=\begin{pmatrix}
  v_1&\cdots &v_d
 \end{pmatrix}\cdot M
\end{equation}
{under Convention~\ref{conv}.}
By \cite[Proposition~4.2]{leitohoku}, we can (and do) choose the $n_i$ such that 
\begin{equation}\label{eq:congruentid}
 M\equiv I_d\mod \pi^r,
\end{equation}
where $r$ denotes the highest Hodge-Tate weight of $T^*$.
If we apply $\vp$, {we deduce from \eqref{eq:mutliplyP} that}
\[
M=A\vp(M)P^{-1}.
\]
If we repeatedly apply $\vp$, we get
\[
M=A^n\vp^n(M)\vp^{n-1}(P^{-1})\cdots \vp(P^{-1})P^{-1}.
\]
So, in particular,
\begin{equation}\label{eq:congruentpi}
 M\equiv A^n\vp^{n-1}(P^{-1})\cdots \vp(P^{-1})P^{-1}\mod \vp^n(\pi^r)
\end{equation}
thanks to \eqref{eq:congruentid}.

\begin{proposition}\label{prop:integralPR}
Let $z\in \HIw(\Qp(\mu_{p^\infty}),T^*)$. Then, for $n\ge0$, there exists an element $\Xi_n(z)\in \Lambda\otimes \Dcris(T^*)$ such that
\[
\left(1\otimes\vp^{-n-1}\right)\cL_{T^*}(z)\equiv \Xi_n(z)\mod \omega_{n,r}\cH\otimes \Dcris(T^*).
\]
\end{proposition}
\begin{proof}
The proof is very similar to the one given in \cite[Lemma~3.44]{BLIntegral}. Let $x\in \NN(T^*)^{\psi=1}$.  There exist $x_i\in \AQp$ such that
\[
(1-\vp)(x)=\sum x_i(1+\pi)\vp(n_i)
\]
by \cite[Theorem~3.5]{LLZ1}. By \eqref{eq:matricepassage} and \eqref{eq:mutliplyP}, we may rewrite this as
\[
(1-\vp)(x)=\begin{pmatrix}
v_1&\cdots &v_d
\end{pmatrix}A(1+\pi)\vp(M)\begin{pmatrix}
x_1\\\vdots \\x_d
\end{pmatrix}.
\]
If we apply $1\otimes\vp^{-n-1}$ on both sides, since $A$ is the matrix of $\vp$ with respect to  $\{v_i\}$, we deduce
\[
(1\otimes\vp^{-n-1})\circ(1-\vp)(x)=\begin{pmatrix}
v_1&\cdots &v_d
\end{pmatrix}A^{-n}(1+\pi)\vp(M)\begin{pmatrix}
x_1\\ \vdots \\x_d
\end{pmatrix}.
\]
If we apply $\vp$ to \eqref{eq:congruentpi}, we see that
\[
A^{-n}\vp(M)\equiv\vp^{n}(P^{-1})\cdots \vp(P^{-1})\mod \vp^{n+1}(\pi^{r}).
\]
Therefore, 
\[
(1\otimes\vp^{-n-1})\circ(1-\vp)(x)\equiv \xi_n(x)\mod  \vp^{n+1}(\pi^{r})
\]
for some $\xi_n(x)\in (\AQp)^{\psi=0}\otimes \Dcris(T^*)$. 
By composing with the Mellin transform $\fM: (\AQp)^{\psi = 0}  \simeq \Lambda$, we get
\[
(\fM\otimes\vp^{-n-1})\circ(1-\vp)(x)
\equiv \Xi_n(x)\mod  \omega_{n,r},
\]
where $\omega_{n,r}=\prod_{i=0}^{r-1}\Tw^{-i}\omega_n$, for some $\Xi_n(x)\in \Lambda\otimes\Dcris(T^*)$. We are done as the Perrin-Riou regulator map satisfies the equality
\[
\cL_{T^*}=(\fM^{-1} \otimes1) \circ (1-\vp ) \circ h_{T^*}^{-1},\]
where $h_{T^*}$ is the isomorphism
$\NN(T^*)^{\psi=1}\stackrel{\sim}{\longrightarrow}\HIw(\Qp(\mu_{p^\infty}),T^*)$ (cf. \cite[Proposition 2.9]{BLIntegral}).
\end{proof}

Proposition~\ref{prop:integralPR} now allows us to deduce the following integrality result   of the  Perrin-Riou pairing over $\QQ_p(\mu_{p^{n+1}})$.

\begin{proposition}\label{integralelements}
Let $0\le j\le r-1$, $n\ge0$, $z\in H^1(\Qp(\mu_{p^{n+1}}),T^*(-j))$ and $v\in \Dcris(T(1+j))$. Then
\[
P_{n+1}\left((p\vp)^{n+1}(v),z\right)\in \cO[\cG_{n+1}].
\]
\end{proposition}
\begin{proof}
Let $\bz$ be any element in $\HIw(\Qp(\mu_{p^\infty}),T^*)$ whose natural image in  $H^1(\Qp(\mu_{p^{n+1}}),T^*(-j))$ is $z$. Let us write 
\[
(1\otimes \vp^{-n-1})\circ\cL_{T^*}(\bz)=\sum_{i=1}^d\cL_{T^*,n,i}(\bz)v_i.
\]
Let $\{v_i'\}$ be the dual basis of $\{v_i\}$ inside $\Dcris(T^*(1))$. Note that $v$ is a linear combination of this dual basis over $\cO$. We have 
\begin{align*}
\cL_{T^*,n,i}(\bz)&=[v_i',(1\otimes \vp^{-n-1})\circ\cL_{T^*}(\bz)]\\
&=[(p\vp)^{n+1}(v_i'),\cL_{T^*}(\bz)]
\end{align*}
since the adjoint of $\vp^{-1}$ under $[,]$ is $p\vp$.
Following \cite[(5) and (6)]{LLZ2}, this can be further rewritten as
\[
\cL_{T^*,n,i}(\bz)=\langle \Omega_{T(1),1}((1+\pi)\otimes (p\vp)^{n+1}( v_i')),\bz\rangle,
\]
where $\Omega_{T(1),1}$ is the Perrin-Riou map on $(\AQp)^{\psi=0}\otimes \Dcris(T(1))$ and $\langle-,-\rangle$ is the Perrin-Riou pairing on
\[
\HIw(\Qp(\mu_{p^\infty}),T(1))\times \HIw(\Qp(\mu_{p^\infty}),T^*)
\]
extended $\cH(\Gamma)$-linearly as defined in \cite[\S3.2]{Lei2011}.

As explained in \cite[\S3.2]{Lei2011}, 
\[
\cL_{T^*,n,i}(\bz)\equiv P_{n+1}( (p\vp)^{n+1}( v_{i,j}'),z)\mod \Tw^{-j}\omega_n(X)
\]
where $v_{i,j}'$ is the natural image of $v_i'$ in $\Dcris(T(1+j))$. Proposition~\ref{prop:integralPR} tells us that this is defined over $\Lambda$. Hence the result follows.
\end{proof}

\subsection{Arithmetic construction of Theta elements and integrality}\label{BFversustheta}
We now go back to the setting considered in the main body of the article and take $T$ in the previous section to be the representation $T_f\otimes T_g$. The goal of this section is to give an alternative definition of our Theta elements in terms of Beilinson--Flach classes and show that they are integral if  \textbf{(H.FL)} holds for $T_f\otimes T_g$. 

\begin{remark}
Note that the smallest and largest Hodge--Tate weights of $T$ are $-k_f-k_g-2$ and $0$ respectively. In the case where at least one of $f$ and $g$ is non-ordinary at $p$, the slope of $\vp$ on $\Dcris(T)$ is slightly inside the interval $(0,k_f+k_g+2)$. In particular, \textbf{(H.FL)} is satisfied if we assume $p-1\ge k_f+k_g+2$. If both $f$ and $g$ are $p$-ordinary, then we would have to assume that $p-1> k_f+k_g+2$ (since the slopes of $\vp$ attain both $0$ and $k_f+k_g+2$).
\end{remark}

We finally show how to obtain the Theta elements defined in Definition~\ref{analyticdeftheta} as images of the certain Beillinson--Flach classes under the pairing of Perrin-Riou at finite levels as a result of the explicit reciprocity laws of the Beillinson--Flach classes and the congruence between Perrin-Riou's big logarithm and the pairing introduced in \S \ref{PRpairingfinlev}.

We begin by the following elementary lemma.

\begin{lemma}\label{elementary}
The linear map $\vp|_{\Dcris(V_f)}$ satisfies the relation \[ \varphi^n = C_n \varphi - \alpha_f \beta_f  C_{n-1},\]
where $C_0=0$, $C_1=1$, and $C_n = \tfrac{\alpha_f^n - \beta_f^n}{\alpha_f - \beta_f}$. (Note that these $C_n$'s have nothing to do with those studied in Lemma~\ref{auxiliarylemma}.)
\end{lemma}

\begin{proof}
It follow from the fact that $\vp|_{\Dcris(V_f)}$ satisfies the relation $ \varphi^2 = (\alpha_f + \beta_f) \varphi - \alpha_f \beta_f$. 
\end{proof}

We recall from \cite[Conjecture~3.5.1]{BLLV} that the Beilinson--Flach classes $\BF_{\lambda,\mu,1}$ are conjectured to be the image of a rank-two Euler system under the composition of the Perrin-Riou map with  $v_{f,\lambda}\otimes v_{g,\mu}$. This conjecture, together with  the explicit description of  $v_{f,\lambda}$ in terms of $\omega_f$ and $\vp(\omega_f)$ in Definition~\ref{def:eigenvectors} and  Lemma \ref{elementary}, lead us to give the following modified Beillinson--Flach elements.

\begin{defn}  For $\mu\in\{\alpha,\beta\}$, we define
\begin{itemize}
\item $\BF_{\omega_{f}^\pm,\mu,1}:= \BF_{\alpha,\mu,1} \mp \BF_{\beta,\mu,1}$,
\item $\BF_{\varphi(\omega_{f}^\pm),\mu,1}:=  \alpha_f\BF_{\alpha,\mu,1} \mp \beta_f \BF_{\beta,\mu,1}$,
\item $\BF_{\varphi^n(\omega_{f}^\pm),\mu,1}:= C_n \BF_{\varphi(\omega_{f}^{\pm }),\mu,1} -\alpha_f \beta_f C_{n-1} \BF_{\omega_{f}^\pm,\mu,1}$, for $n\ge 2$.
\end{itemize}
 \end{defn}

Notice that for all $n\ge0$, we have
 \begin{equation}\label{eqforthetael}
      \BF_{\varphi^{ n}(\omega_{f}^\pm),\mu,1} = \alpha_f^{n} \BF_{\alpha,\mu,1} \mp \beta_f^{n} \BF_{\beta,\mu,1}.
 \end{equation}
The following lemma is an elementary consequence of the   explicit reciprocity laws satisfied by $\BF_{\alpha,\mu,1}$ and $\BF_{\beta,\mu,1}$, as given in Definitions~\ref{defn:geom} and \ref{defn:extra}.

\begin{lemma}\label{lemmaonphiBFclasses}For all $n\ge0$, we have the  equality
\begin{align*}
     \langle \mathcal{L}_{T^*}(\BF_{\varphi^{ n}(\omega_{f}^\pm),\mu,1}), \varphi^{n}(\omega_f) \otimes v_{g,\mu'} \rangle &=  \tfrac{C_\mu \langle\vp(\omega_f),\omega_{f^*}\rangle}{\alpha_f - \beta_f}(\alpha_f^{2n}  L_p(f_\alpha,g) \pm \beta_f^{2n}  L_p(f_\beta,g) \\
     &- \alpha_f^n\beta^n_f( L_p^?(f_\alpha,g_\mu) \pm  L_p^?(f_\beta,g_\mu) )),
\end{align*}
where $C_\mu=\frac{\log_{p,1+k_g}}{\mu_g'-\mu_g}$.
\end{lemma}
\begin{proof}
From Definition~\ref{def:eigenvectors}, we immediately get relations \begin{align*}
    \omega_f &= \tfrac{\langle\vp(\omega_f),\omega_{f^*}\rangle}{\alpha_f - \beta_f}(v_{f,\alpha} - v_{f,\beta}), \\ 
    \varphi(\omega_f) &= \tfrac{\langle\vp(\omega_f),\omega_{f^*}\rangle}{\alpha_f - \beta_f}(\alpha_f v_{f,\alpha} - \beta_f v_{f,\beta}).
\end{align*}  
From these and Lemma~\ref{elementary}, we get the expression \[ \varphi^n(\omega_f ) =\tfrac{\langle\vp(\omega_f),\omega_{f^*}\rangle}{\alpha_f - \beta_f}( \alpha_f^n v_{f,\alpha} - \beta_f^n v_{f,\beta}), \]
for all $n \geq 0$. The result then follows from an explicit manipulation of these relations with \eqref{eqforthetael} and from the aforementioned explicit reciprocity laws.
\end{proof}

\begin{defn}\label{arithmeticdeftheta}
Let $n\ge0$ and $k_g+1\le j\le k_f$ be integers. We define \begin{align*}
\Theta_{\omega_{f}^\pm ,n}^{{\rm coh},j} &:= \tfrac{(p\mu_g')^{n+1}}{C_\mu \langle\vp(\omega_f),\omega_{f^*}\rangle} \mathcal{P}_n( \varphi^{n+2}(\omega_f) \otimes v_{g,\mu'},\BF_{\varphi^{n+2}(\omega_{f}^\pm),\mu,n}^j ) \in L[G_n] \\
     \widetilde{\Theta}_{\omega_{f}^\pm, n}^{{\rm coh},j} &:= \tfrac{(p\mu_g')^{n+1}}{C_\mu \langle\vp(\omega_f),\omega_{f^*}\rangle}\mathcal{P}_n( {\varphi^{n+1}}(\omega_f) \otimes v_{g,\mu'},\BF_{\varphi^{n+1}(\omega_{f}^\pm),\mu,n}^j )\in L[G_n],
\end{align*}
where $\BF_{\varphi^{n+\star}(\omega_{f}^\pm),\mu,n}^j$ denotes the image of $\BF_{\varphi^{n+\star}(\omega_{f}^\pm),\mu,1}$ in $H^1(\mathbb{Q}_{p,n},V^*(-j))$ under the corestriction map.
\end{defn}

\begin{lemma}\label{lem:integralBF}
Suppose that $T=T_f\otimes T_g$ satisfies \textbf{(H.FL)}. Let $n\ge1$ and $\kappa=\min(\ord_p(\alpha_f),\ord_p(\beta_f))$. Suppose that $k_g=-1$.  For $0\le j\le k_f$, { we have \begin{align*}
\BF_{\varphi^{n+1}(\omega_{f}^\pm),\mu,n}^j &\in H^1(\Qpn,T^*(-j)),\\ 
    \BF_{\varphi^{n+2}(\omega_{f}^\pm),\mu,n}^j &\in p^\kappa H^1(\Qpn,T^*(-j)).
\end{align*}}
\end{lemma}
\begin{proof}Let $\lambda,\mu\in\{\alpha,\beta\}$.
Let $\cF$ and $\cG$ be  Coleman families passing through $f_\lambda$ and $g_\mu$ respectively. Let $V_1$ and $V_2$ be two affinoid discs of the weight space containing $k_f$ and $k_g$ respectively as in \cite[\S5.4]{LZ1}. Let us write $\cBF^{[\cF,\cG]}$ for the class defined as in Theorem~5.4.2 of \emph{op. cit.} (we have suppressed the subscripts $c$, $m$ and $a$; we have in fact taken $a=m=1$). If $k_i \in V_i$ are integers such that $0 \le j \le \min(k_1,k_2)$, then the class $\cBF^{[\cF,\cG]}$ specializes to
\[
\frac{1}{(a_p(\cF_{k_1})a_p(\cG_{k_2}))^{n+1}}\cdot \frac{\cBF_{p^{n+1}}^{[\cF_{k_1},\cG_{k_2},j]}}{(-1)^jj!\binom{k_1}{j}\binom{k_2}{j}},
\]
where $\cF_{k_1}$ and $\cG_{k_2}$ denote the specializations of $\cF$ and $\cG$ at $k_1$ and $k_2$ respectively and 
\[
\cBF_{p^{n+1}}^{[\cF_{k_1},\cG_{k_2},j]}\in H^1(\QQ(\mu_{p^{n+1}}),T_{\cF_{k_1}}^*\otimes T_{\cF_{k_1}}^*(-j)).
\]

We show that  $ j!\binom{k_1}{j}\binom{k_2}{j}$ is a  $p$-adic unit. Note that $0\le j\le k_f\le p-2$ under \textbf{(H.FL)}, so $j!$ is always a unit.
 If $j=0$, then $\binom{k_1}{j}=\binom{k_2}{j}=1$. If $j\ne0$, consider the function
 \[
 \binom{X}{j}=\frac{X(X-1)\cdots (X-j+1)}{j!}.
 \]
Note that  $\binom{k_f}{j}$ and  $\binom{-1}{j}$ are both elements of $ \Zp^\times$ since $0\le j\le k_f\le p-2$.
  Therefore, if $k_1$ (resp. $k_2$) is sufficiently close to $k_f$ (resp. $-1$), then $\binom{k_1}{j}$ (resp. $\binom{k_2}{j}$) is  a $p$-adic unit.

We have $\ord_p(\lambda_f)=\ord_p(a_p(\cF_{k_1}))$ and  $\ord_p(\mu_g)=\ord_p(a_p(\cG_{k_2}))=0$. We deduce that the image of $\lambda_f^{n+1}\cBF^{[\cF,\cG]}$ in $H^1(\QQ(\mu_{p^{n+1}}), T_\cF^*\, \hat{\otimes}\, T_\cG^*(-j))[1/p]$ is integral by a density argument. This tells us that  the image of $\lambda_f^{n+1}\BF_{\lambda,\mu,1}$ in $H^1(\QQ(\mu_{p^{n+1}}),V^*(-j))$ belongs to $H^1(\QQ(\mu_{p^{n+1}}),T^*(-j))$.
Our result now follows from \eqref{eqforthetael}.
\end{proof}

\begin{lemma}\label{lem:integral-element}
Suppose that $k_g=-1$ and that $p>k_f+1$, then  $$\frac{1}{\langle \vp(\omega_f),\omega_{f^*}\rangle}\vp(\omega_f)\otimes v_{g,\mu}\in \Dcris(T).$$
\end{lemma}
\begin{proof}
Since $g$ is of weight 1, it is ordinary at $p$. Then, \cite[Corollary~2.2.4]{BL20} tells us that $v_{g,\mu}\in\Dcris(T_g)$. It is thus enough to show that $\frac{1}{\langle \vp(\omega_f),\omega_{f^*}\rangle}\vp(\omega_f)\in \Dcris(T_f)$. Note that $\Dcris(T_f^*)$ is generated by $\omega_{f^*}$ and $\vp(\omega_{f^*})$ as an $\cO$-module (this is proved in \cite[Proposition~2.2.3]{BL20} for the ordinary case and in \cite[Lemma~3.1]{LLZCJM} for the non-ordinary case). Since $\frac{1}{\langle \vp(\omega_f),\omega_{f^*}\rangle}\vp(\omega_f)$ pairs with these two elements to $1$ and $0$ respectively, the result follows.
\end{proof}

\begin{lemma}\label{integralityelts}
Suppose that $T=T_f\otimes T_g$ satisfies \textbf{(H.FL)}. Let $n\ge 0$ and $\kappa$ as defined in Lemma~\ref{lem:integralBF}. Suppose that $k_g=-1$. For $0\le j\le k_f$, {we have  $\Theta_{\omega_{f}^\pm ,n}^{{\rm coh},j} \in p^\kappa \Lambda_n$ and $\widetilde{\Theta}_{\omega_{f}^\pm ,n}^{{\rm coh},j} \in  \Lambda_n$.}
\end{lemma}
\begin{proof} {We only show the result for $\Theta_{\omega_{f}^\pm ,n}^{{\rm coh},j}$, as the proof of the other case is identical.} 
By Lemma \ref{lem:integralBF}, the class $\BF_{\varphi^{n+2}(\omega_{f}^\pm),\mu,n}^j$ lies inside $p^\kappa H^1(\Qpn,T^*(-j))$. Thus, we may combine this fact with  Proposition \ref{integralelements} and Lemma~\ref{lem:integral-element} to deduce that 
\[ P_{n+1}( (p\varphi)^{n+1}(v),\BF_{\varphi^{n+2}(\omega_{f}^\pm),\mu,n}^j )  \in p^\kappa \cO[ \mathcal{G}_{n+1} ],\]
where $v=\frac{1}{\langle \vp(\omega_f),\omega_{f^*}\rangle}\vp(\omega_f)\otimes v_{g,\mu'}$. This gives
\[\frac{(p\mu_g')^{n+1}}{\langle \vp(\omega_f),\omega_{f^*}\rangle} P_{n+1}( \varphi^{n+2}(\omega_f)\otimes v_{g,\mu'},\BF_{\varphi^{n+2}(\omega_{f}^\pm),\mu,n}^j )  \in p^\kappa \cO[ \mathcal{G}_{n+1} ].\]
Since $k_g=-1$, we have $\frac{1}{C_\mu}=\mu_g'-\mu_g\in\cO$. In particular, 
\[\frac{(p\mu_g')^{n+1}}{C_\mu\langle \vp(\omega_f),\omega_{f^*}\rangle} P_{n+1}( \varphi^{n+2}(\omega_f)\otimes v_{g,\mu'},\BF_{\varphi^{n+2}(\omega_{f}^\pm),\mu,n}^j )  \in p^\kappa \cO[ \mathcal{G}_{n+1} ].\]
Hence, the result follows by projecting to $\Lambda_n$.
\end{proof}

We now compare the Theta elements of Definition \ref{arithmeticdeftheta} with the ones given in Definition \ref{analyticdeftheta}.

\begin{proposition}\label{cohomconstruction}
We have the following equalities. \begin{itemize}
   \item  Let $\alpha_f \ne -\beta_f$, then \[ \Theta_{j,n} =\frac{1}{2}\left[\tfrac{1}{\alpha_f + \beta_f}( \Theta_{\omega_{f}^- ,n}^{{\rm coh},j} - \alpha_f \beta_f \widetilde{\Theta}_{\omega_{f}^- ,n}^{{\rm coh},j}) + \tfrac{1}{\alpha_f - \beta_f}( \Theta_{\omega_{f}^+ ,n}^{{\rm coh},j} - \alpha_f \beta_f \widetilde{\Theta}_{\omega_{f}^+,n}^{{\rm coh},j}) \right];\]
   
    \item Let $\alpha_f=-\beta_f$, then
    \begin{align*}
    \Theta^+_{j,n} &= \begin{cases} \tfrac{1}{2 \alpha_f}\Theta_{\omega_{f}^+ ,n}^{{\rm coh},j} & \text{if }n\text{ is odd} \\ \tfrac{ \alpha_f}{2}\widetilde{\Theta}_{\omega_{f}^+ ,n}^{{\rm coh},j} & \text{if }n\text{ is even}  \end{cases} \\ 
         \Theta^-_{j,n} &= \begin{cases}\tfrac{ \alpha_f}{2}\widetilde{\Theta}_{\omega_{f}^+ ,n}^{{\rm coh},j}  & \text{if }n\text{ is odd} \\ \tfrac{1}{2 \alpha_f}\Theta_{\omega_{f}^+ ,n}^{{\rm coh},j}   & \text{if }n\text{ is even.}  \end{cases} 
    \end{align*}
\end{itemize}
Furthermore, suppose that $T$ satisfies (\textbf{H.FL}) and that  $k_g=-1$.
\begin{itemize}
    \item If $\ord_p(\alpha_f)\ne \ord_p(\beta_f)$, then $\Theta_{j,n}\in \Lambda_n$;
    \item If $\alpha_f=-\beta_f$, then, { if $n$ is odd (resp. even) $\Theta^{+}_{j,n}\in \Lambda_n$,  $\Theta^{-}_{j,n}\in p^{k_f}\Lambda_n$ (resp.  $\Theta^{-}_{j,n}\in \Lambda_n$,  $\Theta^{+}_{j,n}\in p^{k_f}\Lambda_n$).}
\end{itemize}
\end{proposition}

\begin{proof}
By \cite[\S3.2]{Lei2011}, we have that 
\[ \pi_\Delta \circ \Tw^{j} \langle \cL_{T^*}(\BF_{\varphi^{(n+2)}(\omega_{f}^\pm),\mu,1} ), \varphi^{n+2}(\omega_f) \otimes v_{g,\mu'} \rangle \equiv \mathcal{P}_{n}( (p\vp)^{n+1}( \varphi(\omega_f) \otimes v_{g,\mu'}),\BF_{\varphi^{n+2}(\omega_{f}^\pm),\mu,n}^j) \mod \omega_{n}(X). \]
This together with Lemma \ref{lemmaonphiBFclasses} let us conclude that  \[ \Theta_{\omega_{f}^\pm ,n}^{{\rm coh},j} = \frac{\alpha_f^{2n+4}}{\alpha_f - \beta_f} L_p(\alpha, j,n) \pm \frac{\beta_f^{2n+4}}{\alpha_f - \beta_f} L_p(\beta, j,n)-\frac{(\alpha_f\beta_f)^{n+2}}{\alpha_f - \beta_f}( L_p^?(\alpha,\mu, j,n) \pm L_p^?(\beta,\mu, j,n)). \] 
By the same argument, we obtain the equality
\[ \widetilde{\Theta}_{\omega_{f}^\pm ,n}^{{\rm coh},j} = \frac{\alpha_f^{2n+2}}{\alpha_f - \beta_f} L_p(\alpha, j,n) \pm \frac{\beta_f^{2n+2}}{\alpha_f - \beta_f} L_p(\beta, j,n)-\frac{(\alpha_f\beta_f)^{n+1}}{\alpha_f - \beta_f}( L_p^?(\alpha,\mu, j,n) \pm L_p^?(\beta,\mu, j,n)). \] 
Combining the two formulae, we get
\[\Theta_{\omega_{f}^\pm ,n}^{{\rm coh},j} - \alpha_f \beta_f  \widetilde{\Theta}_{\omega_{f}^\pm ,n}^{{\rm coh},j} = \alpha_f^{2n+3}L_p(\alpha, j,n) \mp \beta_f^{2n+3}L_p(\beta, j,n). \]
These relations allow us to write $\alpha_f^{2n+3}L_p(\alpha, j,n) $ and $\beta_f^{2n+3}L_p(\beta, j,n)$ as linear combinations of these \emph{arithmetic} Theta elements:
\begin{align*}
     2 \alpha_f^{2n+3}L_p(\alpha, j,n) &= \Theta_{\omega_{f}^- ,n}^{{\rm coh},j}  +  \Theta_{\omega_{f}^+ ,n}^{{\rm coh},j} - \alpha_f \beta_f ( \widetilde{\Theta}_{\omega_{f}^- ,n}^{{\rm coh},j}+    \widetilde{\Theta}_{\omega_{f}^+ ,n}^{{\rm coh},j}) \\ 
     2 \beta_f^{2n+3}L_p(\beta, j,n) &= \Theta_{\omega_{f}^- ,n}^{{\rm coh},j} -  \Theta_{\omega_{f}^+ ,n}^{{\rm coh},j} - \alpha_f \beta_f ( \widetilde{\Theta}_{\omega_{f}^- ,n}^{{\rm coh},j} -    \widetilde{\Theta}_{\omega_{f}^+ ,n}^{{\rm coh},j})
\end{align*}
Suppose that $\alpha_f \ne - \beta_f$, then, by Definition \ref{analyticdeftheta} and the relations above, we have
\begin{align*}
    \Theta_{j,n} :&=
\frac{1}{\alpha_f^2-\beta_f^2} \left(\alpha_f^{2n+4}L_p(\alpha,j,n) - \beta_f^{2n+4}L_p(\beta,j,n) \right) \\ &=\tfrac{1}{2(\alpha_f^2-\beta_f^2)}\left[(\alpha_f - \beta_f)\Theta_{\omega_{f}^- ,n}^{{\rm coh},j} + (\alpha_f + \beta_f)\Theta_{\omega_{f}^+ ,n}^{{\rm coh},j} + (\alpha_f \beta_f^2 - \alpha_f^2\beta_f)\widetilde{\Theta}_{\omega_{f}^- ,n}^{{\rm coh},j} - (\alpha_f \beta_f^2 + \alpha_f^2\beta_f)\widetilde{\Theta}_{\omega_{f}^+ ,n}^{{\rm coh},j} \right] \\ &=\frac{1}{2}\left [\tfrac{1}{\alpha_f + \beta_f} \Theta_{\omega_{f}^- ,n}^{{\rm coh},j} + \tfrac{1}{\alpha_f - \beta_f}\Theta_{\omega_{f}^+ ,n}^{{\rm coh},j} -\alpha_f \beta_f( \tfrac{1}{\alpha_f + \beta_f}\widetilde{\Theta}_{\omega_{f}^- ,n}^{{\rm coh},j} + \tfrac{1}{\alpha_f - \beta_f}\widetilde{\Theta}_{\omega_{f}^+ ,n}^{{\rm coh},j}) \right], 
\end{align*} 
as desired. Suppose now that $\alpha_f = - \beta_f$. We have that (choosing $\mu=\beta_g$)
 \begin{align*}
     \Theta_{\omega_{f}^+ ,n}^{{\rm coh},j} &= \frac{\alpha_f^{2n+3}}{2}\left( L_p(\alpha, j,n) + L_p(\beta, j,n)- (-1)^n (L_p^?(\alpha,\beta, j,n) + L_p^?(\beta,\beta, j,n)) \right), \\ \widetilde{\Theta}_{\omega_{f}^+ ,n}^{{\rm coh},j} &= \frac{\alpha_f^{2n+1}}{2}\left( L_p(\alpha, j,n) + L_p(\beta, j,n)- (-1)^{n+1} (L_p^?(\alpha,\beta, j,n) + L_p^?(\beta,\beta, j,n)) \right).
 \end{align*}  
 Since by Definition \ref{analyticdeftheta}
\begin{align*}
    \Theta_{j,n}^\pm :&=
\frac{\alpha_f^{2n+2}}{4} \left(L_p(\alpha,j,n) + L_p(\beta,j,n) \pm L_p^?(\alpha,\beta, j,n) \pm L_p^?(\beta,\beta, j,n) \right),
\end{align*} 
we conclude that 
\begin{align*}
    \Theta^+_{j,n} &= \begin{cases} \tfrac{1}{2 \alpha_f}\Theta_{\omega_{f}^+ ,n}^{{\rm coh},j} & \text{if }n\text{ is odd} \\ \tfrac{ \alpha_f}{2}\widetilde{\Theta}_{\omega_{f}^+ ,n}^{{\rm coh},j} & \text{if }n\text{ is even}  \end{cases} \\ 
         \Theta^-_{j,n} &= \begin{cases}\tfrac{ \alpha_f}{2}\widetilde{\Theta}_{\omega_{f}^+ ,n}^{{\rm coh},j}  & \text{if }n\text{ is odd} \\ \tfrac{1}{2 \alpha_f}\Theta_{\omega_{f}^+ ,n}^{{\rm coh},j}   & \text{if }n\text{ is even.}  \end{cases} 
    \end{align*}

For the integrality of the Theta elements, note that if $\ord_p(\alpha_f)\ne\ord_p(\beta_f)$, the denominators $\alpha_f \pm \beta_f$ have the same $p$-adic valuation equal to $\kappa$, where $\kappa$ is as defined in Lemma~\ref{lem:integralBF}. {In the case $\alpha_f=-\beta_f$, the element $\tfrac{1}{2\alpha_f}$ has $p$-adic valuation equal to $-\kappa$}. Therefore, under the additional hypotheses that $T$ satisfies (\textbf{H.FL}) and $k_g=0$ the integrality of the Theta elements now follows from Lemma \ref{integralityelts}.
\end{proof}

\begin{remark}
Suppose that $f$ corresponds to an elliptic curve $E/\QQ$ with $a_p(E)=0$ and $g$ is a weight one modular form. Then, the  formulae for $\Theta^{+}_{n}:=\Theta^{+}_{0,n}$ and $\Theta^-_n := \Theta^{-}_{0,n}$ of Proposition \ref{cohomconstruction} are in concordance with the ones of Lemma \ref{lem:trace}.
\end{remark}

\begin{remark}
Proposition \ref{cohomconstruction} should be compared to \cite[Lemma 7.2]{Kurihara2002}, where the modular elements defined by Mazur and Tate are related to values of the pairing $\mathcal{P}_n$ evaluated at the zeta elements of Kato.
\end{remark}

\bibliographystyle{alpha}

\bibliography{Cauchi_Lei_-_On_analogues_of_Mazur-Tate_type_conjectures_in_the_Rankin-Selberg_setting}

\begin{thebibliography}{BLLV19}

\bibitem[Ber04]{berger04}
Laurent Berger.
\newblock Limites de repr\'{e}sentations cristallines.
\newblock {\em Compos. Math.}, 140(6):1473--1498, 2004.

\bibitem[BK90]{BK}
Spencer Bloch and Kazuya Kato.
\newblock {$L$}-functions and {T}amagawa numbers of motives.
\newblock In {\em The {G}rothendieck {F}estschrift, {V}ol. {I}}, volume~86 of
  {\em Progr. Math.}, pages 333--400. Birkh\"{a}user Boston, Boston, MA, 1990.

\bibitem[BL17]{BLIntegral}
K{\^a}z{\i}m B\"uy\"ukboduk and Antonio Lei.
\newblock Integral {I}wasawa theory of {G}alois representations for
  non-ordinary primes.
\newblock {\em Math. Z.}, 286(1-2):361--398, 2017.

\bibitem[BL20a]{BL20}
K{\^a}z{\i}m B\"uy\"ukboduk and Antonio Lei.
\newblock Semi-ordinary {I}wasawa theory for {R}ankin-{S}elberg products, 2020.
\newblock preprint, arXiv:2008.08411.

\bibitem[BL20b]{BFSuper}
K{\^a}z{\i}m B\"uy\"ukboduk and Antonio Lei.
\newblock {I}wasawa theory of elliptic modular forms over imaginary quadratic
  fields at non-ordinary primes.
\newblock to appear in IMRN (arXiv:1605.05310), $\geq$2020.

\bibitem[BLLV19]{BLLV}
K{\^a}z{\i}m B\"uy\"ukboduk, Antonio Lei, David Loeffler, and Guhan Venkat.
\newblock Iwasawa theory for {R}ankin-{S}elberg products of {$p$}-nonordinary
  eigenforms.
\newblock {\em Algebra Number Theory}, 13(4):901--941, 2019.

\bibitem[Edi92]{ed}
Bas Edixhoven.
\newblock The weight in {S}erre's conjectures on modular forms.
\newblock {\em Invent. Math.}, 109(3):563--594, 1992.

\bibitem[EPW06]{EPW}
Matthew Emerton, Robert Pollack, and Tom Weston.
\newblock Variation of iwasawa invariants in hida families.
\newblock {\em Inventiones mathematicae}, 163(3):523--580, 2006.

\bibitem[Gre89]{greenberg1989iwasawa}
Ralph Greenberg.
\newblock Iwasawa theory for {$ p $}-adic representations.
\newblock In {\em Algebraic Number Theory, in Honor of K. Iwasawa}, pages
  97--137. Mathematical Society of Japan, 1989.

\bibitem[Gre99]{greenberg1999iwasawaell}
Ralph Greenberg.
\newblock Iwasawa theory for elliptic curves.
\newblock In {\em Arithmetic theory of elliptic curves}, pages 51--144.
  Springer, 1999.

\bibitem[Gre06]{greenbergstructuredoc}
Ralph Greenberg.
\newblock On the structure of certain galois cohomology groups.
\newblock {\em Doc. Math.}, pages 335--391, 2006.

\bibitem[Gre10]{greenberg2010surjectivity}
Ralph Greenberg.
\newblock Surjectivity of the global-to-local map defining a selmer group.
\newblock {\em Kyoto Journal of Mathematics}, 50(4):853--888, 2010.

\bibitem[Gre16]{greenbergselmergroups}
Ralph Greenberg.
\newblock On the structure of {S}elmer groups.
\newblock In {\em Elliptic curves, modular forms and {I}wasawa theory}, volume
  188 of {\em Springer Proc. Math. Stat.}, pages 225--252. Springer, Cham,
  2016.

\bibitem[HL14]{Harron}
Robert Harron and Antonio Lei.
\newblock Iwasawa theory for symmetric powers of {CM} modular forms at
  non-ordinary primes.
\newblock {\em J. Th\'{e}or. Nombres Bordeaux}, 26(3):673--708, 2014.

\bibitem[IP06]{IP}
Adrian Iovita and Robert Pollack.
\newblock Iwasawa theory of elliptic curves at supersingular primes over
  {$\mathbb{ Z_p}$}-extensions of number fields.
\newblock {\em J. Reine Angew. Math.}, 598:71--103, 2006.

\bibitem[Kim14]{Kim}
Byoung~Du Kim.
\newblock Signed-{S}elmer groups over the {$\mathbb{Z}_p^2$}-extension of an
  imaginary quadratic field.
\newblock {\em Canad. J. Math.}, 66(4):826--843, 2014.

\bibitem[KK19]{KK}
Chan-Ho Kim and Masato Kurihara.
\newblock On the refined conjectures on fitting ideals of {S}elmer groups of
  elliptic curves with supersingular reduction, 2019.
\newblock International Mathematics Research Notices, rnz129,
  https://doi.org/10.1093/imrn/rnz129.

\bibitem[KLZ17]{KLZ2}
Guido Kings, David Loeffler, and Sarah~Livia Zerbes.
\newblock Rankin-{E}isenstein classes and explicit reciprocity laws.
\newblock {\em Camb. J. Math.}, 5(1):1--122, 2017.

\bibitem[KLZ20]{KLZ1}
Guido Kings, David Loeffler, and Sarah~Livia Zerbes.
\newblock Rankin--{E}isenstein classes for modular forms.
\newblock {\em American J. Math.}, 142(1), 2020.

\bibitem[KO18]{KO}
Takahiro Kitajima and Rei Otsuki.
\newblock On the plus and the minus {S}elmer groups for elliptic curves at
  supersingular primes.
\newblock {\em Tokyo J. Math.}, 41(1):273--303, 2018.

\bibitem[Kob03]{kobayashi03}
Shin-ichi Kobayashi.
\newblock Iwasawa theory for elliptic curves at supersingular primes.
\newblock {\em Invent. Math.}, 152(1):1--36, 2003.

\bibitem[Kur02]{Kurihara2002}
Masato Kurihara.
\newblock On the {T}ate {S}hafarevich groups over cyclotomic fields of an
  elliptic curve with supersingular reduction. {I}.
\newblock {\em Invent. Math.}, 149(1):195--224, 2002.

\bibitem[Lei11]{Lei2011}
Antonio Lei.
\newblock Iwasawa theory for modular forms at supersingular primes.
\newblock {\em Compos. Math.}, 147(3):803--838, 2011.

\bibitem[Lei17]{leitohoku}
Antonio Lei.
\newblock Bounds on the {T}amagawa numbers of a crystalline representation over
  towers of cyclotomic extensions.
\newblock {\em Tohoku Math. J. (2)}, 69(4):497--524, 2017.

\bibitem[LLZ10]{LLZ1}
Antonio Lei, David Loeffler, and Sarah~Livia Zerbes.
\newblock Wach modules and {I}wasawa theory for modular forms.
\newblock {\em Asian J. Math.}, 14(4):475--528, 2010.

\bibitem[LLZ11]{LLZ2}
Antonio Lei, David Loeffler, and Sarah~Livia Zerbes.
\newblock Coleman maps and the {$p$}-adic regulator.
\newblock {\em Algebra Number Theory}, 5(8):1095--1131, 2011.

\bibitem[LLZ14]{LLZ14}
Antonio Lei, David Loeffler, and Sarah~Livia Zerbes.
\newblock Euler systems for {R}ankin-{S}elberg convolutions of modular forms.
\newblock {\em Ann. of Math. (2)}, 180(2):653--771, 2014.

\bibitem[LLZ17]{LLZCJM}
Antonio Lei, David Loeffler, and Sarah~Livia Zerbes.
\newblock On the asymptotic growth of {B}loch-{K}ato-{S}hafarevich-{T}ate
  groups of modular forms over cyclotomic extensions.
\newblock {\em Canad. J. Math.}, 69(4):826--850, 2017.

\bibitem[Loe17]{loeffler2017}
David Loeffler.
\newblock Images of adelic {G}alois representations for modular forms.
\newblock {\em Glasgow Mathematical Journal}, 59(1):11--25, 2017.

\bibitem[Loe18]{loeffler18}
David Loeffler.
\newblock A note on {$p$}-adic {R}ankin-{S}elberg {$L$}-functions.
\newblock {\em Canad. Math. Bull.}, 61(3):608--621, 2018.

\bibitem[LZ16]{LZ1}
David Loeffler and Sarah~Livia Zerbes.
\newblock Rankin--{E}isenstein classes in {C}oleman families.
\newblock {\em Res. Math. Sci.}, 3:Paper No. 29, 53, 2016.

\bibitem[MT87]{MaTa}
B.~Mazur and J.~Tate.
\newblock Refined conjectures of the ``{B}irch and {S}winnerton-{D}yer type''.
\newblock {\em Duke Math. J.}, 54(2):711--750, 1987.

\bibitem[MTT86]{MTT}
Barry Mazur, John Tate, and Jeremy Teitelbaum.
\newblock On $p$-adic analogues of the conjectures of {B}irch and
  {S}winnerton-{D}yer.
\newblock {\em Invent. math}, 84(1):1--48, 1986.

\bibitem[MW84]{MW}
B.~Mazur and A.~Wiles.
\newblock Class fields of abelian extensions of {${\bf Q}$}.
\newblock {\em Invent. Math.}, 76(2):179--330, 1984.

\bibitem[Nek06]{SelmerComplexes}
Jan Nekov{\'a}{\v{r}}.
\newblock {\em Selmer complexes}.
\newblock Number 310 in Ast\'erisque. Soci\'et\'e math\'ematique de France,
  2006.

\bibitem[Och00]{ochiaiBK}
Tadashi Ochiai.
\newblock Control theorem for {B}loch--{K}ato's {S}elmer {G}roups of {$p$-a}dic
  representations.
\newblock {\em Journal of Number Theory}, 82(1):69--90, 2000.

\bibitem[Och01]{ochiai}
Tadashi Ochiai.
\newblock Control theorem for {G}reenberg's {S}elmer groups of {G}alois
  deformations.
\newblock {\em Journal of Number Theory}, 88(1):59--85, 2001.

\bibitem[Pol03]{pollack03}
Robert Pollack.
\newblock On the {$p$}-adic {$L$}-function of a modular form at a supersingular
  prime.
\newblock {\em Duke Math. J.}, 118(3):523--558, 2003.

\bibitem[Pol05]{pollack05}
Robert Pollack.
\newblock An algebraic version of a theorem of {K}urihara.
\newblock {\em Journal of Number Theory}, 110:164--177, 2005.

\end{thebibliography}

\end{document}